\def\bu{\bullet}
\def\marker{\>\hbox{${\vcenter{\vbox{
    \hrule height 0.4pt\hbox{\vrule width 0.4pt height 6pt
    \kern6pt\vrule width 0.4pt}\hrule height 0.4pt}}}$}\>}
\def\gpic#1{#1
     \smallskip\par\noindent{\centerline{\box\graph}} \medskip}
\documentclass[12pt]{article}
\usepackage{amsthm,amsmath,amsfonts,amssymb}
\usepackage{fullpage}
\usepackage{xcolor}
\usepackage[colorlinks=false,linkbordercolor=red]{hyperref}
\hypersetup{%
  colorlinks=false,
  linkbordercolor=red,
  pdfborderstyle={/S/U/W 1}
}

\newtheorem{thm}{Theorem}[section]
\newtheorem{conj}[thm]{Conjecture}

\newtheorem{lem}[thm]{Lemma}

\newtheorem{theorem}[thm]{Theorem}
\newtheorem{conjecture}[thm]{Conjecture}
\newtheorem{proposition}[thm]{Proposition}
\newtheorem{lemma}[thm]{Lemma}

\theoremstyle{definition}
\newtheorem{definition}[thm]{Definition}
\newtheorem{example}[thm]{Example}
\newtheorem{remark}[thm]{Remark}
\newtheorem{question}[thm]{Question}
\newtheorem{exercise}{Exercise}[section]

\expandafter\ifx\csname dplus\endcsname\relax \csname newbox\endcsname\dplus\fi
\expandafter\ifx\csname dplustemp\endcsname\relax
\csname newdimen\endcsname\dplustemp\fi
\setbox\dplus=\vtop{\vskip -6pt\hbox{%
    \kern .2em
    \special{pn 6}%
    \special{pa 0 40}%
    \special{pa 40 80}%
    \special{fp}%
    \special{pa 40 80}%
    \special{pa 80 40}%
    \special{fp}%
    \special{pa 80 40}%
    \special{pa 40 0}%
    \special{fp}%
    \special{pa 40 80}%
    \special{pa 40 0}%
    \special{fp}%
    \special{pa 40 0}%
    \special{pa 0 40}%
    \special{fp}%
    \special{pa 0 40}%
    \special{pa 80 40}%
    \special{fp}%
    \hbox{\vrule depth0.080in width0pt height 0pt}%
    \kern .7em
  }%
}%

\def\FL#1{\left\lfloor#1\right\rfloor}
\def\ceil#1{\left\lceil#1\right\rceil}
\def\CL#1{\left\lceil#1\right\rceil}

\def\C#1{\left|#1\right|}
\def\VEC#1#2#3{#1_{#2},\ldots,#1_{#3}}
\def\FR{\frac}
\def\mad{{\rm{mad}}}
\def\Mad{{\rm{mad}}}

\def\avd{{\overline{d}}}

\def\chil{\chi_{\ell}}
\def\larb{{\rm la}}

\def\st{\colon\,}

\def\NN{{\mathbb N}}
\def\chia{a}
\def\chial{a_\ell}
\def\chis{s}

\def\chic{\chi_c}
\def\chil{\chi_\ell}

\def\la{\langle}
\def\ra{\rangle}
\def\esub{\subseteq}
\def\nul{\varnothing}
\def\al{\alpha}
\def\SM#1#2{\sum_{#1\in#2}}

\def\cdgh{\CL{\FR12\Delta(G)}}

\def\chic{\chi_c}

\begin{document}

\title{An Introduction to the Discharging Method\\ via Graph Coloring}
\author{
Daniel W. Cranston\thanks{Virginia Commonwealth University,
dcranston@vcu.edu}\,\and
Douglas B. West\thanks{Zhejiang Normal University and University of Illinois,
west@math.uiuc.edu.}
}

\date{\today}

\maketitle

\vspace{-1.5pc}

\begin{abstract}
We provide a ``how-to'' guide to the use and application of the Discharging
Method.  Our aim is not to exhaustively survey results proved by this
technique, but rather to demystify the technique and facilitate its wider use,
using applications in graph coloring as examples.  Along the way, we present
some new proofs and new problems.
\end{abstract}

\baselineskip 16pt
\section{Introduction}
Arguments that can be phrased in the language of the Discharging Method have
been used in graph theory for more than 100 years, though that name is much
more recent.  The most famous application of the method is the proof of the
Four Color Theorem, stating that graphs embeddable in the plane have chromatic
number at most $4$.  However, the method remains mysterious to many.  Our aim
is to explain its use and make the method more widely accessible.  Although we
mention many applications, including stronger versions of results proved here,
cataloguing applications is not our goal.  Borodin~\cite{Bsurv} presents a
survey of applications of discharging to coloring of plane graphs.

Discharging is most commonly used as a tool in a two-pronged approach to
inductive proofs, typically for sparse graphs.  In this context, it is used to
prove that a global sparseness hypothesis guarantees the existence of some
desired local structure.  The method has been applied to many types of problems
(including graph embeddings and decompositions, spread of infections in
networks, geometric problems, etc.).  Nevertheless, we present only 
applications in graph coloring (where it has been used most often), in order
to emphasize the discharging techniques.

In the simplest version of discharging involves just reallocation of vertex
degrees in the context of a global bound on the average degree.  We view each
vertex as having an initial ``charge'' equal to its degree.  To show that
average degree less than $b$ forces the presence of a desired local structure,
we show that the absence of such a structure allows charge to be moved (via
``discharging rules'') so that the final charge at each vertex is at least $b$.
This violates the hypothesis, and hence the desired structure must occur.

In an application of the resulting structure theorem, one shows that each such
local structure is ``reducible'', meaning that it cannot occur in a minimal
counterexample to the desired conclusion.  This motivates the phrase ``an
unavoidable set of reducible configurations'' to describe the overall process.

\begin{definition}
A {\it configuration} in a graph $G$ can be any structure in $G$ (often a
specified sort of subgraph).  A configuration is {\it reducible} for a graph
property $Q$ if it cannot occur in a minimal graph not having property $Q$.
Let $d_G(v)$ or simply $d(v)$ denote the degree (number of neighbors) of vertex
$v$ in $G$, and let $\avd(G)$ denote the average of the vertex degrees in $G$.
{\it Degree charging} is the assignment to each vertex $v$ of an ``initial
charge'' equal to $d(v)$.
\end{definition}

The notion of configuration is vague to permit use in various contexts.
``Minimal'' refers to some partial order on the graphs being considered;
usually it is just minimality with respect to taking subgraphs, and the
property $Q$ is {\it monotone} (preserved by taking subgraphs).

Sparse local configurations aid in inductive proofs about coloring.  For
example, when $\avd(G)<k$ with $k\in\NN$, the pigeonhole principle guarantees
a vertex with degree less than $k$ in $G$.  Also, when $d(v)<k$, a proper
$k$-coloring of $G-v$ extends to a proper $k$-coloring of $G$.  (A
{\it $k$-coloring} is a function that assigns labels to vertices from a set of
size $k$, a coloring of a graph $G$ is {\it proper} if adjacent vertices have
distinct colors, $G$ is {\it $k$-colorable} if it admits a proper $k$-coloring,
and the {\it chromatic number} $\chi(G)$ is the least $k$ such that $G$ is
$k$-colorable.)

In other words, vertices of degree less than $k$ are reducible for the property
$\chi(G)\le k$.  However, guaranteeing such a vertex from the global bound 
$\avd(G)<k$ does not need discharging.  To illustrate how discharging works and
interacts with reducibility, we consider another elementary example after
introducing notation convenient for discussing vertex degrees.

\begin{definition}
A {\it $j$-vertex}, {\it $j^+$-vertex}, or {\it $j^-$-vertex} is a vertex with
degree equal to $j$, at least $j$, or at most $j$, respectively.  A
{\it $j$-neighbor} of $v$ is a $j$-vertex that is a neighbor of $v$.
We write $\delta(G)$ for the minimum and $\Delta(G)$ for the maximum 
of the vertex degrees in $G$.
\end{definition}

\begin{lemma}\label{avd3-25}
If $\avd(G)<3$, then $G$ has a $1^-$-vertex or a $2$-vertex with a
$5^-$-neighbor.
\end{lemma}
\begin{proof}
We use degree charging; each vertex $v$ starts with charge $d(v)$.  Suppose
that $G$ has no $1^-$-vertex and that no $2$-vertex in $G$ has a
$5^-$-neighbor.  We move charge so that each vertex ends with charge at least
$3$.  The $2$-vertices need charge; $4^+$-vertices can give charge.

Let each $2$-vertex take $\FR12$ from each neighbor.  Now each $2$-vertex has
charge $3$, since no two $2$-vertices are adjacent.  Vertices of degrees
$3,4,5$ lose no charge, since we assumed that no $2$-vertex has a
$5^-$-neighbor.  Every $6^+$-vertex $v$ loses charge at most $\FR12$ to
each neighbor, leaving it with charge at least $d(v)/2$, which is at least $3$
when $d(v)\ge6$.  Thus $\avd(G)\ge3$ when no $2$-vertex has a $5^-$-neighbor.
\end{proof}

A $2$-vertex with a $5^-$-neighbor is a local sparseness condition, somehow
more sparse than a $2$-vertex with high-degree neighbors.  We first consider
its use for edge-coloring.  (A {\it $k$-edge-coloring} of a graph $G$ assigns
labels to edges from a set of size $k$; it is {\it proper} if incident edges
have distinct colors, $G$ is {\it $k$-edge-colorable} if it has a proper
$k$-edge-coloring, and the {\it edge-chromatic number} $\chi'(G)$ is the
least $k$ such that $G$ is $k$-edge-colorable.)

Here we phrase the reducibility statement in more generality.
The {\it weight} of a subgraph $H$ of a graph $G$ is $\sum_{v\in V(H)} d_G(v)$;
we sum the degrees in the full graph $G$.

\begin{lemma}\label{edgewt}
An edge with weight at most $k+1$ is a reducible configuration for the 
property of being $k$-edge-colorable.
\end{lemma}
\begin{proof}
Let $G$ be a graph having an edge $e$ of weight at most $k+1$.  If the graph
$G-e$ is $k$-edge-colorable, then a color is available to extend the coloring
to $e$, because $e$ is incident to a total of at most $k-1$ other edges at 
its two endpoints.  Thus a minimal graph $G$ with $\chi'(G)>k$ cannot contain
such a configuration.
\end{proof}

To complete an inductive proof of $\chi'(G)\le6$ from Lemmas~\ref{avd3-25}
and \ref{edgewt}, we also need average degree less than $3$ in subgraphs of
$G$.

\begin{definition}
The {\it maximum average degree} of a graph $G$, denoted $\mad(G)$,
is the maximum of the average degree over all subgraphs of $G$.
\end{definition}

The application is now easy.  Note that always $\chi'(G)\ge\Delta(G)$.  In
fact, Vizing's Theorem~\cite{G,V2} states that always $\chi'(G)\le\Delta(G)+1$,
and distinguishing between $\chi'(G)=\Delta(G)$ and
$\chi'(G)=\Delta(G)+1$ is an important and difficult problem.

\begin{theorem}\label{mad3edge}
If $\mad(G)<3$ and $\Delta(G)\ge 6$, then $\chi'(G)=\Delta(G)$.
\end{theorem}
\begin{proof}
Fix an integer $k$ at least $6$.  We prove more generally that if $\mad(G)<3$
and $\Delta(G)\le k$, then $\chi'(G)\le k$.  That is, among graphs with
$\mad(G)<3$ and $\Delta(G)\le k$ there is no minimal graph satisfying $\chi'>k$.
Note that the hypotheses also hold in subgraphs.

We may discard isolated vertices.  By Lemma~\ref{avd3-25}, $G$ then has a
$1$-vertex or has a $2$-vertex with a $5^-$-neighbor.  The edge incident
to a $1$-vertex has weight at most $\Delta(G)+1$; an edge joining a $2$-vertex
to a $5^-$-neighbor has weight at most $7$.  In either case, the weight of this
edge $e$ is at most $k+1$, and Lemma~\ref{edgewt} implies that $G$ is not a
minimal graph satisfying $\chi'(G)>k$.  Hence there is no minimal
counterexample.
\end{proof}

Before leaving Theorem~\ref{mad3edge}, we note that many reducibility arguments
for coloring problems involve deleting some parts of a graph (such as a
$1$-vertex or the edge $e$ in the proof above) and then choosing colors for the
missing pieces as they are replaced.  Suitable choices can be made if there are
enough available colors; it does not matter what the colors are.  In this
situation, the arguments yield stronger statements about coloring from lists.

\begin{definition}\label{Dlist}
A \emph{list assignment} $L$ on a graph $G$ gives each $v\in V(G)$ a set $L(v)$
of colors, called its \emph{list}.  In a \emph{$k$-uniform} list assignment,
each list has size $k$.  Given a list assignment $L$, an \emph{$L$-coloring} of
$G$ is a proper coloring $\phi$ of $G$ such that $\phi(v)\in L(v)$ for all
$v\in V(G)$.  A graph $G$ is \emph{$k$-choosable} if $G$ is $L$-colorable
whenever each list has size at least $k$ (we may assume $L$ is $k$-uniform).
The \emph{list chromatic number} or \emph{choice number} of $G$, written
$\chil(G)$, is the least $k$ such that $G$ is $k$-choosable.  Analogous
language is used for edge-colorings chosen from list assignments to edges.
\end{definition}

Since the lists can be identical, always $\chil(G)\ge \chi(G)$.  Thus
$\chil(G)\le b$ is stronger than $\chi(G)\le b$.  For example, $\mad(G)<k$
inductively yields $\chil(G)\le k$.  Similarly, an edge of weight at most
$k+1$ is reducible for $k$-edge-choosability, and the proof of
Theorem~\ref{mad3edge} yields $\chil(G)=\Delta(G)$ when $\mad(G)<3$ and
$\Delta(G)\ge6$.

\medskip

We present various classical applications, some with new proofs.  We emphasize
discharging arguments but include reducibility arguments to show how
discharging is applied.  For clarity and simplicity in illustrating the method,
we often assume more restrictive hypotheses than used in the strongest known
results.  Often those results are proved similarly, but with more detail in the
discharging arguments and more configurations to be proved reducible.

The basic idea of discharging proofs is simple, and the proofs are usually
easy to follow, though they may have many details.  The mystery arises in the
choice of reducible configurations, the rules for moving charge, and how to
find the best hypothesis.  We will explain the interplay among these and
suggest how the proofs are discovered, starting with the context of $\mad(G)<b$
in Section~\ref{secmad}.  We include related results as exercises to aid
in self-study; most exercises have relatively short solutions (items labeled
``Question'' are unsolved).

As we have illustrated, structural results proved by discharging when
$\avd(G)<b$ are applied inductively to obtain coloring conclusions under the
hypothesis $\mad(G)<b$.  The point is that every subgraph $H$ satisfies
$\mad(H)<b$.  For natural hereditary families like planar graphs, bounds on
$\mad(G)$ are easily obtained.  The families satisfying $\mad(G)<b$ for all
positive $b$ provide a rich spectrum for study.

Discharging has been used to prove many results on coloring or structure of
planar graphs (or planar graphs with large girth).  Euler's Formula implies
that (every subgraph of) a planar graph with girth at least $g$ has average
degree less than $\FR {2g}{g-2}$.  Some results on such graphs in fact hold
whenever $\Mad(G)< \FR{2g}{g-2}$, regardless of planarity, often with the same
proof by discharging.  Others, as discussed in Section~\ref{secplanar}, truly
need planarity and may assign charge to both the faces and the vertices (the
dual graph is also sparse).  This is the basic reason why discharging is so
useful for planar graphs.  Subsequent sections will discuss additional
techniques of discharging, especially with examples from ``list coloring''.

Finally, we note that in addition to its usefulness as a proof technique, the
discharging method also has algorithmic implications, often yielding fast
constructive algorithms for good colorings or embeddings.  Iterative
application of the structure theorem yields reductions to smaller graphs.
After a good coloring of a base graph is found, the intermediate graphs receive
good colorings using the reducibility arguments, until the original graph is
restored and its coloring obtained (see Section~6 of~\cite{CK}).

\section{Structure and coloring of sparse graphs}\label{secmad}

In studying discharging, the principle and the details are simple.  The mystery
is the source of the discharging rule and the hypothesis on $\avd(G)$.  The
secret is that the discharging rule is found before knowing the hypothesis of
the theorem and is used to discover it.  To explain such aspects of
discharging, we study the forcing of local configurations with small weight.

\begin{remark}\label{2bound}
{\it Finding the best bound on $\mad(G)$.}
Consider Lemma~\ref{avd3-25} more generally.  When we want $G$ with $\mad(G)<b$
to have a $1^-$-vertex or have a $2$-vertex with a $j^-$-neighbor, what is the
best choice of $b$?  Actually, we start with the proof and let it produce the
statement.  We must make $b$ at most $3$, since otherwise $G$ may be
$3$-regular with no $2$-vertex.  Given that, when we exclude $1^-$-vertices and
use degree charging, only $2$-vertices will need charge.  The most natural way
for them to obtain it is to take it from their neighbors.

If each $2$-vertex takes $\rho$ from each neighbor, then final charge is at
least $b$ at each vertex if and only if $2$-vertices obtain enough charge and
vertices with degree larger than $j$ do not lose too much.  Such vertices can
lose $\rho$ to each neighbor, so we need $2+2\rho\ge b$ and $d-d\rho\ge b$ when
$d\ge j+1$.  To find the largest $b$ that works, set $2+2\rho=(j+1)(1-\rho)$,
yielding $\rho=\FR{j-1}{j+3}$ and hence $b=2+2\rho=4\FR{j+1}{j+3}$.  When
$j=5$, we obtain Lemma~\ref{avd3-25}.

What we did was find the weakest hypothesis allowing the discharging proof to
work.  The proof also provides sharpness examples showing that the condition
$\mad(G)<b$ cannot be weakened.  If every $2$-vertex has only
$(j+1)$-neighbors, every $(j+1)$-vertex has only $2$-neighbors, and there
are no other vertices, then all the equalities are tight, no $2$-vertex has
a $j^-$-neighbor, all vertices end with charge exactly $b$, and the average
degree is $b$.  Hence we obtain a sharpness example by taking a $(j+1)$-regular
graph and subdividing every edge.

What the discharging argument does is count part of the degree of higher-degree
vertices at their $2$-neighbors.  In this sense discharging is ``amortized
counting''; the counting of the degree of a vertex is allocated to (or
``charged to'') other vertices.
\end{remark}

The discharging argument for a structure theorem guaranteeing local
configurations is quite separate from the reducibility arguments used to give
an inductive proof of the desired conclusion.  Thus the unavoidable set
resulting from a particular sparseness condition may be usable to prove other
results.  In practice, usually the configurations are those already known to be
reducible for the desired property in the application.  Nevertheless,
Lemma~\ref{avd3-25} does apply to another coloring problem.

\begin{definition}\label{Dacyclic}
An {\it acyclic coloring} of a graph is a proper coloring such that the 
union of any two color classes induces an acyclic subgraph; equivalently,
no cycle is $2$-colored.
\end{definition}

\begin{theorem}\label{acyc6}
If $\mad(G)<3$, then $G$ is acyclically $6$-choosable.
\end{theorem}
\begin{proof}
It suffices to show that the configurations forced by Lemma~\ref{avd3-25}
when $\mad(G)<3$ are reducible for the existence of an acyclic coloring chosen
from a $6$-uniform list assignment $L$.  By definition, $\mad(G-v)<3$.  To show
reducibility, we assume an acyclic $L$-coloring $\phi$ of $G-v$ and obtain such
a coloring of $G$.  The cases appear in Figure~\ref{6acyc}.
\begin{figure}[h]
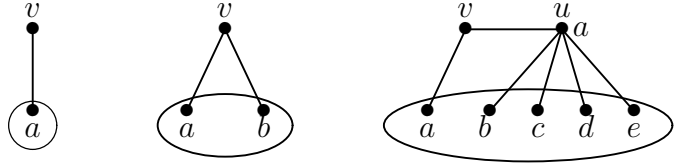

\gpic{
\expandafter\ifx\csname graph\endcsname\relax \csname newbox\endcsname\graph\fi
\expandafter\ifx\csname graphtemp\endcsname\relax \csname newdimen\endcsname\graphtemp\fi
\setbox\graph=\vtop{\vskip 0pt\hbox{%
    \graphtemp=.5ex\advance\graphtemp by 0.101in
    \rlap{\kern 0.126in\lower\graphtemp\hbox to 0pt{\hss $\bu$\hss}}%
    \graphtemp=.5ex\advance\graphtemp by 0.529in
    \rlap{\kern 0.126in\lower\graphtemp\hbox to 0pt{\hss $\bu$\hss}}%
    \graphtemp=.5ex\advance\graphtemp by 0.529in
    \rlap{\kern 0.932in\lower\graphtemp\hbox to 0pt{\hss $\bu$\hss}}%
    \graphtemp=.5ex\advance\graphtemp by 0.101in
    \rlap{\kern 1.133in\lower\graphtemp\hbox to 0pt{\hss $\bu$\hss}}%
    \graphtemp=.5ex\advance\graphtemp by 0.529in
    \rlap{\kern 1.335in\lower\graphtemp\hbox to 0pt{\hss $\bu$\hss}}%
    \special{pn 8}%
    \special{ar 126 604 126 126 0 6.28319}%
    \special{pn 11}%
    \special{ar 1133 604 353 164 0 6.28319}%
    \special{pa 126 101}%
    \special{pa 126 529}%
    \special{fp}%
    \special{pa 932 529}%
    \special{pa 1133 101}%
    \special{fp}%
    \special{pa 1133 101}%
    \special{pa 1335 529}%
    \special{fp}%
    \graphtemp=.5ex\advance\graphtemp by 0.000in
    \rlap{\kern 0.126in\lower\graphtemp\hbox to 0pt{\hss $v$\hss}}%
    \graphtemp=.5ex\advance\graphtemp by 0.629in
    \rlap{\kern 0.126in\lower\graphtemp\hbox to 0pt{\hss $a$\hss}}%
    \graphtemp=.5ex\advance\graphtemp by 0.629in
    \rlap{\kern 0.932in\lower\graphtemp\hbox to 0pt{\hss $a$\hss}}%
    \graphtemp=.5ex\advance\graphtemp by 0.000in
    \rlap{\kern 1.133in\lower\graphtemp\hbox to 0pt{\hss $v$\hss}}%
    \graphtemp=.5ex\advance\graphtemp by 0.629in
    \rlap{\kern 1.335in\lower\graphtemp\hbox to 0pt{\hss $b$\hss}}%
    \graphtemp=.5ex\advance\graphtemp by 0.529in
    \rlap{\kern 2.191in\lower\graphtemp\hbox to 0pt{\hss $\bu$\hss}}%
    \graphtemp=.5ex\advance\graphtemp by 0.101in
    \rlap{\kern 2.392in\lower\graphtemp\hbox to 0pt{\hss $\bu$\hss}}%
    \graphtemp=.5ex\advance\graphtemp by 0.101in
    \rlap{\kern 2.896in\lower\graphtemp\hbox to 0pt{\hss $\bu$\hss}}%
    \graphtemp=.5ex\advance\graphtemp by 0.529in
    \rlap{\kern 2.518in\lower\graphtemp\hbox to 0pt{\hss $\bu$\hss}}%
    \graphtemp=.5ex\advance\graphtemp by 0.529in
    \rlap{\kern 2.770in\lower\graphtemp\hbox to 0pt{\hss $\bu$\hss}}%
    \graphtemp=.5ex\advance\graphtemp by 0.529in
    \rlap{\kern 3.022in\lower\graphtemp\hbox to 0pt{\hss $\bu$\hss}}%
    \graphtemp=.5ex\advance\graphtemp by 0.529in
    \rlap{\kern 3.273in\lower\graphtemp\hbox to 0pt{\hss $\bu$\hss}}%
    \special{pa 2896 101}%
    \special{pa 2392 101}%
    \special{fp}%
    \special{pa 2896 101}%
    \special{pa 2518 529}%
    \special{fp}%
    \special{pa 2896 101}%
    \special{pa 2770 529}%
    \special{fp}%
    \special{pa 2896 101}%
    \special{pa 3022 529}%
    \special{fp}%
    \special{pa 2896 101}%
    \special{pa 3273 529}%
    \special{fp}%
    \special{pa 2191 529}%
    \special{pa 2392 101}%
    \special{fp}%
    \special{ar 2719 604 755 189 0 6.28319}%
    \graphtemp=.5ex\advance\graphtemp by 0.629in
    \rlap{\kern 2.191in\lower\graphtemp\hbox to 0pt{\hss $a$\hss}}%
    \graphtemp=.5ex\advance\graphtemp by 0.000in
    \rlap{\kern 2.392in\lower\graphtemp\hbox to 0pt{\hss $v$\hss}}%
    \graphtemp=.5ex\advance\graphtemp by 0.000in
    \rlap{\kern 2.896in\lower\graphtemp\hbox to 0pt{\hss $u$\hss}}%
    \graphtemp=.5ex\advance\graphtemp by 0.101in
    \rlap{\kern 2.996in\lower\graphtemp\hbox to 0pt{\hss $a$\hss}}%
    \graphtemp=.5ex\advance\graphtemp by 0.629in
    \rlap{\kern 2.493in\lower\graphtemp\hbox to 0pt{\hss $b$\hss}}%
    \graphtemp=.5ex\advance\graphtemp by 0.629in
    \rlap{\kern 2.770in\lower\graphtemp\hbox to 0pt{\hss $c$\hss}}%
    \graphtemp=.5ex\advance\graphtemp by 0.629in
    \rlap{\kern 3.022in\lower\graphtemp\hbox to 0pt{\hss $d$\hss}}%
    \graphtemp=.5ex\advance\graphtemp by 0.629in
    \rlap{\kern 3.273in\lower\graphtemp\hbox to 0pt{\hss $e$\hss}}%
    \graphtemp=.5ex\advance\graphtemp by 0.529in
    \rlap{\kern 3.399in\lower\graphtemp\hbox to 0pt{\hss $~$\hss}}%
    \hbox{\vrule depth0.793in width0pt height 0pt}%
    \kern 3.500in
  }%
}%
}

\vspace{-1pc}

\caption{Reducibility for acyclic $6$-coloring\label{6acyc}}
\end{figure}

If $d_G(v)\le 1$, then we extend $\phi$ by letting $\phi(v)$ be a color in
$L(v)$ not used on the neighbor of $v$.  If $d_G(v)=2$ and $\phi$ gives
distinct colors on $N_G(v)$, then again we just avoid them on $v$.
If $d_G(v)=2$ and $\phi$ gives the same color to both vertices of $N_G(v)$,
then there is danger of completing a $2$-colored cycle.  However, since $v$ has
a $5^-$-neighbor $u$, at most four other colors appear on the neighbors of $u$,
so a color in $L(v)$ remains available for $v$.
\end{proof}

If a structure theorem with hypothesis $\avd(G)<b$ is sharp, then when
$\avd(G)$ exceeds $b$ we must add other configurations to obtain a structure
theorem.  At $\avd(G)=3$ we may have no $2$-vertices with $5^-$-neighbors and
perhaps no $2$-vertices at all.  Nevertheless, when $\avd(G)<4$ the graph must
have a $2^-$-vertex or have a $3$-vertex with a $5^-$-neighbor
(Exercise~\ref{3nbr}).

In the other direction, when we reduce the bound on $\avd(G)$ we can impose
more sparseness.  By Remark~\ref{2bound}, $\avd(G)<\FR{12}5$ implies that $G$
has two adjacent $2$-vertices if it has no $1^-$-vertex.  What sparser local
configuration can we force when the average degree declines even further?

\begin{definition}
An {\it $\ell$-thread} in a graph $G$ is a trail of length $\ell+1$ in $G$
whose $\ell$ internal vertices have degree $2$ in the full graph $G$.
\end{definition}

Under this definition, an $\ell$-thread contains two $(\ell-1)$-threads, and
the ends of a thread may be the same vertex.

\begin{lemma}\label{firstavd}
If $\avd(G)<2+\FR2{3\ell-1}$ and $G$ has no $2$-regular component, then $G$
contains a $1^-$-vertex or an $\ell$-thread.
\end{lemma}
\begin{proof}
Let $\rho=\FR{1}{3\ell-1}$, so the hypothesis is $\avd(G)<2+2\rho$.
Use degree charging.  If neither stated configuration occurs, then we
redistribute charge to leave each vertex with at least $2+2\rho$.  Since $G$
has no $1^-$-vertex, $\delta(G)\ge2$.
Since $G$ has no $2$-regular component, each $2$-vertex lies in a unique
maximal thread.  Redistribute charge as follows:

\medskip\noindent
(R1) Each $2$-vertex $v$ takes charge $\rho$ from each end of its maximal
thread.
\medskip

Since each $2$-vertex lies on a unique maximal thread, it ends with charge
$2+2\rho$.  Since $\ell$-threads are forbidden, each $j$-vertex with $j\ge3$
gives charge to at most $\ell-1$ vertices along the thread started by each
incident edge, losing at most $j(\ell-1)\rho$.  To show that its final charge
is at least $2+2\rho$, we compute
$$
j-j(\ell-1)\rho\ge 3\left[1-\FR{\ell-1}{3\ell-1}\right]=2+\FR2{3\ell-1}.
$$
Hence avoiding the specified configurations requires
$\avd(G)\ge 2+\FR2{3\ell-1}$.
\end{proof}

\begin{remark}
Once again, the hypothesis of Lemma~\ref{firstavd} is discovered from the
proof, and the structure theorem is sharp.  When using degree charging with
$\avd(G)<2+2\rho$, only $2$-vertices need charge (once we restrict to
$\delta(G)\ge2$), and the natural (local) sources of charge are the nearest
vertices of larger degree.  This yields the discharging rule, taking $\rho$
from each.

We choose $\rho$ by finding the weakest hypothesis that avoids taking too much
from $3^+$-vertices.  The inequality $j-j(\ell-1)\rho\ge 2+2\rho$ implies that
the proof guarantees an $\ell$-thread when $\rho\le \FR{j-2}{j(\ell-1)+2}$ for
$j\ge3$.  Thus setting $\rho=\FR1{3\ell-1}$ both makes the proof work and gives
the weakest hypothesis where it works.

Furthermore, to achieve sharpness in the proof, all vertices should have degree
$2$ or $3$.  Replace each edge of any $3$-regular graph with an
$(\ell-1)$-thread.  By the discharging computation, the average degree is
$2+\FR2{3\ell-1}$, and there are no $\ell$-threads.
\end{remark}

Discovering a discharging argument can be fun, but its value is in applications.
To apply our result on threads inductively to a coloring problem, we replace
the bound on $\avd(G)$ by the same bound on $\mad(G)$.  The condition
$\mad(G)<3$ already implies $3$-colorability, and graphs having any odd cycle
require at least three colors, so we need another coloring model to allow a
stronger bound on $\mad(G)$ to have a chance to improve on $3$-colorability.

\begin{definition}
A {\it $t$-fold coloring} of a graph $G$ assigns each vertex a set of $t$
colors so that adjacent vertices receive disjoint sets.  The {\it $t$-fold
chromatic number} $\chi_t(G)$ is the least $k$ such that $G$ has a $t$-fold
coloring using subsets of $[k]$ (where $[n]=\{1,\dots,n\}$).  The
{\it fractional chromatic number} $\chi^*(G)$ of $G$ is $\inf_t \chi_t(G)/t$.
The {\it odd girth} of $G$, written $g_o(G)$, is the length of a shortest odd
cycle in $G$ (infinite when $G$ is bipartite).
\end{definition}

An ordinary proper coloring is a $1$-fold coloring, so always
$\chi^*(G)\le\chi(G)$.  The {\it independence number} $\alpha(G)$ of a graph
$G$ is the maximum size of an independent set of vertices.  When $G$ has $n$
vertices, always a $t$-fold coloring of $G$ needs at least $nt/\alpha(G)$
colors, since each color can only be used on an independent set.  Hence 
$\chi^*(G)\ge n/\alpha(G)$; equality holds for vertex-transitive graphs using
all automorphic images of a largest independent set.  In particular,
$\chi^*(C_{2t+1})=2+\FR 1t$, where $C_n$ denotes the $n$-vertex cycle.

\begin{theorem}\label{fracol}
If $g_o(G)\ge2t+1$ and $\mad(G)< 2+\FR 1{3t-2}$, then $G$ has a $t$-fold
coloring with $2t+1$ colors, and hence $\chi^*(G)\le 2+\FR 1t$.
\end{theorem}
\begin{proof}
We have noted that $g_o(G)\ge2t+1$ is needed.  By Lemma~\ref{firstavd}, it
suffices to show that $1^-$-vertices and $(2t-1)$-threads (which may be
contained in longer threads) are reducible for $t$-fold $(2t+1)$-colorability.
If $d(v)\le1$, then a such a coloring $\phi$ of $G-v$ easily extends to $v$,
choosing $\phi(v)$ from the complement of the set assigned to its neighbor when
$d(v)=1$.

When $G$ contains a $(2t-1)$-thread with endpoints $u$ and $v$, let $G'$ be the
graph obtained by deleting its internal vertices.  The hypotheses hold for
$G'$, so $G'$ admits a $t$-fold coloring $\phi$ using $2t+1$ colors.
We want to extend $\phi$ along the thread.  When two $t$-sets in $[2t+1]$
differ by one element, a unique $t$-set lies in the complement of both.  Hence
in two steps we can switch any color in a $t$-set to any missing color.  In
fact, this is the only change achievable in two steps (we can also return to
the same $t$-set).  Since there are $2t$ steps along the thread from $u$ to
$v$, we can thus extend $\phi$ along the thread to obtain the desired coloring
of $G$.
\end{proof}

We introduced fractional coloring to illustrate the use of threads in sparse
graphs.  Similar results are available in connection with another variation on
coloring.

\begin{definition}
A {\it $(p,q)$-coloring} $\phi$ of $G$ colors $V(G)$ by elements of 
$\{0,\ldots,p-1\}$ so that adjacent vertices receive colors cyclically at
least $q$ apart; that is, $q\le\C{\phi(u)-\phi(v)}\le p-q$ when $uv\in E(G)$.
A graph having a $(p,q)$-coloring is {\it $(p,q)$-colorable}.  The
{\it circular chromatic number} of $G$, written $\chic(G)$, is
$\inf\{\FR pq\st G$ is $(p,q)$-colorable$\}$.
\end{definition}

A $(k,1)$-coloring is just a proper $k$-coloring, so $\chic(G)\le\chi(G)$.
A $(p,q)$-coloring can be viewed as a $q$-fold coloring with $p$ colors,
where the $q$-sets used are segments of $q$ cyclically consecutive colors.
Thus circular coloring is a restricted form of fractional coloring, and always
$\chi_c(G)\ge\chi^*(G)$.  In fact, always $\CL{\chic(G)}=\chi(G)$, so
$\chic(G)$ can be viewed as a refinement of $\chi(G)$.  Zhu~\cite{Z1,Z2}
surveyed results on circular coloring.

The hypotheses of Theorem~\ref{fracol} also suffice for the stronger conclusion
$\chi_c(G)\le 2+\FR1t$, since again a $(2t-1)$-thread is reducible
(Exercise~\ref{circhr}).  In fact, we can further strengthen the result
by obtaining the same conclusion when $\mad(G)$ is allowed to be somewhat
larger.  We have seen that $\mad(G)=2+\FR1{3t-2}$ does not force
$(2t-1)$-threads, but then another structure is forced that also is reducible
for $\chi_c(G)\le 2+\FR1t$.  We prove the discharging part.

\begin{lemma}\label{avdnbhd}
If $\avd(G)<2+\FR{1}{2t-1}$ and $G$ has no $2$-regular component, then $G$
contains
(1) a $1^-$-vertex, or
(2) a $3$-vertex with at least $4t-3$ vertices of degree $2$ on its maximal
incident threads, or
(3) a $4^+$-vertex incident to a $(2t-1)$-thread.
\end{lemma}
\begin{proof}
Let $\rho=\FR12\FR{1}{2t-1}$, so $\avd(G)<2+2\rho$.  Use degree charging.
We may assume $\delta(G)\ge2$ and that $G$ is
connected.  Redistribute charge using the same rule as before.

\medskip\noindent
(R1) Each $2$-vertex $v$ takes charge $\rho$ from each end of its maximal
thread.
\medskip

As in Lemma~\ref{firstavd}, each $2$-vertex ends with charge $2+\rho$.  If
no $3$-vertex has enough $2$-vertices on its incident threads, then each
$3$-vertex $v$ loses charge at most $(4t-4)\rho$ and retains at least
$3-\FR{2t-2}{2t-1}$, which equals $2+2\rho$.

Now let $v$ be a $4^+$-vertex.  With no incident $(2t-1)$-thread, $v$ gives
charge to at most $2t-2$ vertices along the thread starting at each incident
edge.  The minimum remaining charge $d(v)[1-(2t-2)\rho]$ is minimized when
$d(v)=4$.  We compute
$4[1-(2t-2)\rho]=4-2\FR{2t-2}{2t-1}=2+4\rho$.

Every vertex ends with charge at least $2+2\rho$, so avoiding the
specified configurations requires $\avd(G)\ge2+\FR1{2t-1}$.
\end{proof}

Showing that the configurations in Lemma~\ref{avdnbhd} are reducible for
$\chi_c(G)\le 2+\FR{1}{t}$ (Exercise~\ref{circhr}) proves the following
result.

\begin{theorem}\label{madcirc}
If $g_o(G)\ge 2t+1$ and $\mad(G)<2+\FR{1}{2t-1}$, then
$\chi_c(G)\le 2+\FR{1}{t}$.
\end{theorem}

The bound on $\mad(G)$ in Theorem~\ref{madcirc} is still not sharp for
$\chic(G)\le2+\FR1t$.  Borodin--Hartke--Ivanova--Kostochka--West~\cite{BHIKW}
proved for triangle-free graphs that $\mad(G)<\FR{12}5$ implies
$\chic(G)\le\FR52$, while Theorem~\ref{madcirc} with $t=2$ requires $\mad(G)<\FR94$
to obtain $\chic(G)\le\FR52$.  Sharpness of their result follows from $t=2$ in
the following construction.

\begin{example}\label{madsharp}
Let $G_t$ consist of two $(2t+1)$-cycles sharing a single edge, plus a
$(2t-2)$-thread joining the vertices opposite the shared edge on the two
cycles; Figure~\ref{G2G3} shows $G_2$ and $G_3$.  Note that $G_1\cong K_4$ and
$\avd(G_2)=\FR{12}5$.

Consider a possible $(2t+1,t)$-coloring of $G_t$.  Once colors are chosen on
the edge shared by the two $(2t+1)$-cycles in $G_t$, the colors on the
other two $3$-vertices are forced to be the same.  The coloring does not extend
to all of $G_t$, since a thread of odd length at most $2t-1$ cannot have the
same color on its endpoints.  In fact, $\chic(G_t)=2+\FR1{t-1/2}$ and
$\avd(G_t)=2+\FR2{3t-1}$.  For $t=2$, we have $\chic(G_2)=\FR83$ and
$\avd(G_2)=\FR{12}5$.
\end{example}

\begin{figure}[h]
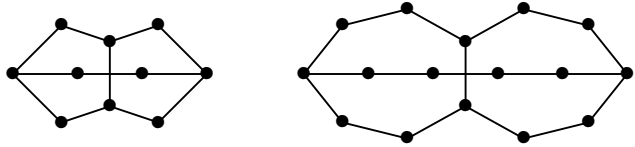

\gpic{
\expandafter\ifx\csname graph\endcsname\relax \csname newbox\endcsname\graph\fi
\expandafter\ifx\csname graphtemp\endcsname\relax \csname newdimen\endcsname\graphtemp\fi
\setbox\graph=\vtop{\vskip 0pt\hbox{%
    \graphtemp=.5ex\advance\graphtemp by 0.381in
    \rlap{\kern 0.042in\lower\graphtemp\hbox to 0pt{\hss $\bu$\hss}}%
    \graphtemp=.5ex\advance\graphtemp by 0.381in
    \rlap{\kern 0.381in\lower\graphtemp\hbox to 0pt{\hss $\bu$\hss}}%
    \graphtemp=.5ex\advance\graphtemp by 0.381in
    \rlap{\kern 0.719in\lower\graphtemp\hbox to 0pt{\hss $\bu$\hss}}%
    \graphtemp=.5ex\advance\graphtemp by 0.381in
    \rlap{\kern 1.058in\lower\graphtemp\hbox to 0pt{\hss $\bu$\hss}}%
    \graphtemp=.5ex\advance\graphtemp by 0.550in
    \rlap{\kern 0.550in\lower\graphtemp\hbox to 0pt{\hss $\bu$\hss}}%
    \graphtemp=.5ex\advance\graphtemp by 0.212in
    \rlap{\kern 0.550in\lower\graphtemp\hbox to 0pt{\hss $\bu$\hss}}%
    \special{pn 11}%
    \special{pa 42 381}%
    \special{pa 1058 381}%
    \special{fp}%
    \special{pa 550 550}%
    \special{pa 550 212}%
    \special{fp}%
    \graphtemp=.5ex\advance\graphtemp by 0.127in
    \rlap{\kern 0.296in\lower\graphtemp\hbox to 0pt{\hss $\bu$\hss}}%
    \graphtemp=.5ex\advance\graphtemp by 0.635in
    \rlap{\kern 0.296in\lower\graphtemp\hbox to 0pt{\hss $\bu$\hss}}%
    \graphtemp=.5ex\advance\graphtemp by 0.127in
    \rlap{\kern 0.804in\lower\graphtemp\hbox to 0pt{\hss $\bu$\hss}}%
    \graphtemp=.5ex\advance\graphtemp by 0.635in
    \rlap{\kern 0.804in\lower\graphtemp\hbox to 0pt{\hss $\bu$\hss}}%
    \special{pa 42 381}%
    \special{pa 296 127}%
    \special{fp}%
    \special{pa 296 127}%
    \special{pa 550 212}%
    \special{fp}%
    \special{pa 550 212}%
    \special{pa 804 127}%
    \special{fp}%
    \special{pa 804 127}%
    \special{pa 1058 381}%
    \special{fp}%
    \special{pa 42 381}%
    \special{pa 296 635}%
    \special{fp}%
    \special{pa 296 635}%
    \special{pa 550 550}%
    \special{fp}%
    \special{pa 550 550}%
    \special{pa 804 635}%
    \special{fp}%
    \special{pa 804 635}%
    \special{pa 1058 381}%
    \special{fp}%
    \graphtemp=.5ex\advance\graphtemp by 0.381in
    \rlap{\kern 1.565in\lower\graphtemp\hbox to 0pt{\hss $\bu$\hss}}%
    \graphtemp=.5ex\advance\graphtemp by 0.381in
    \rlap{\kern 1.904in\lower\graphtemp\hbox to 0pt{\hss $\bu$\hss}}%
    \graphtemp=.5ex\advance\graphtemp by 0.381in
    \rlap{\kern 2.242in\lower\graphtemp\hbox to 0pt{\hss $\bu$\hss}}%
    \graphtemp=.5ex\advance\graphtemp by 0.381in
    \rlap{\kern 2.581in\lower\graphtemp\hbox to 0pt{\hss $\bu$\hss}}%
    \graphtemp=.5ex\advance\graphtemp by 0.381in
    \rlap{\kern 2.919in\lower\graphtemp\hbox to 0pt{\hss $\bu$\hss}}%
    \graphtemp=.5ex\advance\graphtemp by 0.381in
    \rlap{\kern 3.258in\lower\graphtemp\hbox to 0pt{\hss $\bu$\hss}}%
    \graphtemp=.5ex\advance\graphtemp by 0.550in
    \rlap{\kern 2.412in\lower\graphtemp\hbox to 0pt{\hss $\bu$\hss}}%
    \graphtemp=.5ex\advance\graphtemp by 0.212in
    \rlap{\kern 2.412in\lower\graphtemp\hbox to 0pt{\hss $\bu$\hss}}%
    \special{pa 1565 381}%
    \special{pa 3258 381}%
    \special{fp}%
    \special{pa 2412 550}%
    \special{pa 2412 212}%
    \special{fp}%
    \graphtemp=.5ex\advance\graphtemp by 0.127in
    \rlap{\kern 1.768in\lower\graphtemp\hbox to 0pt{\hss $\bu$\hss}}%
    \graphtemp=.5ex\advance\graphtemp by 0.042in
    \rlap{\kern 2.107in\lower\graphtemp\hbox to 0pt{\hss $\bu$\hss}}%
    \graphtemp=.5ex\advance\graphtemp by 0.635in
    \rlap{\kern 1.768in\lower\graphtemp\hbox to 0pt{\hss $\bu$\hss}}%
    \graphtemp=.5ex\advance\graphtemp by 0.719in
    \rlap{\kern 2.107in\lower\graphtemp\hbox to 0pt{\hss $\bu$\hss}}%
    \graphtemp=.5ex\advance\graphtemp by 0.127in
    \rlap{\kern 3.055in\lower\graphtemp\hbox to 0pt{\hss $\bu$\hss}}%
    \graphtemp=.5ex\advance\graphtemp by 0.042in
    \rlap{\kern 2.716in\lower\graphtemp\hbox to 0pt{\hss $\bu$\hss}}%
    \graphtemp=.5ex\advance\graphtemp by 0.635in
    \rlap{\kern 3.055in\lower\graphtemp\hbox to 0pt{\hss $\bu$\hss}}%
    \graphtemp=.5ex\advance\graphtemp by 0.719in
    \rlap{\kern 2.716in\lower\graphtemp\hbox to 0pt{\hss $\bu$\hss}}%
    \special{pa 1565 381}%
    \special{pa 1768 127}%
    \special{fp}%
    \special{pa 1768 127}%
    \special{pa 2107 42}%
    \special{fp}%
    \special{pa 2107 42}%
    \special{pa 2412 212}%
    \special{fp}%
    \special{pa 2412 212}%
    \special{pa 2716 42}%
    \special{fp}%
    \special{pa 2716 42}%
    \special{pa 3055 127}%
    \special{fp}%
    \special{pa 3055 127}%
    \special{pa 3258 381}%
    \special{fp}%
    \special{pa 1565 381}%
    \special{pa 1768 635}%
    \special{fp}%
    \special{pa 1768 635}%
    \special{pa 2107 719}%
    \special{fp}%
    \special{pa 2107 719}%
    \special{pa 2412 550}%
    \special{fp}%
    \special{pa 2412 550}%
    \special{pa 2716 719}%
    \special{fp}%
    \special{pa 2716 719}%
    \special{pa 3055 635}%
    \special{fp}%
    \special{pa 3055 635}%
    \special{pa 3258 381}%
    \special{fp}%
    \hbox{\vrule depth0.762in width0pt height 0pt}%
    \kern 3.300in
  }%
}%
}

\vspace{-1pc}
\caption{Graphs $G_2$ and $G_3$ for Example~\ref{madsharp}\label{G2G3}}
\end{figure}

We offer a conjecture; Example~\ref{madsharp} shows that it would be sharp.

\begin{conjecture}\label{ccconj}
If $g_o(G)\ge 2t+1$ and $\mad(G)<2+\FR2{3t-1}$, then $\chic(G)\le 2+\FR1t$.
\end{conjecture}

The conjecture is trivial for $t=1$.  The proof for $t=2$ in~\cite{BHIKW} used
``long-distance'' discharging, moving charge along special long paths.

\begin{remark}
A weaker version of Conjecture~\ref{ccconj} is a special case of the
following result from Borodin--Kim--Kostochka--West~\cite{BKKW}: If $G$ has
girth at least $6t-2$ and $\mad(G)<2+\FR3{5t-2}$, then $\chic(G)\le 2+\FR1t$.
Their proof uses discharging and reducible configurations involving multiple
threads, like the $3$-vertex in Lemma~\ref{madcirc}.  The result is motivated
by the conjecture of Jaeger~\cite{Ja} that every $4t$-edge-connected graph has
``circular flow number'' at most $2+\FR1t$.  When $G$ is planar, this statement
for the dual graph $G^*$ becomes the conjecture that $\chi_c(G)\le 2+\FR1t$
when $G$ is planar with girth at least $4t$.

Lov\'asz--Thomassen--Wu--Zhang~\cite{LTWZ} proved a weaker form of Jaeger's
Conjecture, replacing $4t$ by $6t$.  Thus $\chi_c(G)\le 2+\FR1t$ when $G$ is
planar with girth at least $6t$.  By Euler's Formula, $\Mad(G)<\FR{2g}{g-2}$
when $G$ is planar with girth $g$.  Thus $\Mad(G)<2+\FR{2}{2t-1}$ when $G$ is
planar with girth at least $4t$.  Conjecture~\ref{ccconj} in some sense
proposes a trade-off: by further restricting to $\Mad(G)<2+\FR2{3t-1}$, the
girth requirement can be relaxed to $g_o(G)\ge 2t+1$ and still yield
$\chi_c(G)\le 2+\FR1t$, even without requiring planarity.
\end{remark}

We have considered only sparse graphs and small $\chic(G)$, but the problem is
more general.  Always $\chi(G)\ge\omega(G)$, where $\omega(G)$ is the maximum
number of pairwise adjacent vertices in $G$, called the {\it clique number} of
$G$.  The {\it circular clique} $K_{p:q}$ is the graph whose vertices are the
congruence classes modulo $p$, adjacent when they differ by at least $q$.  The
{\it circular clique number}, written $\omega_c(G)$, is
$\max\{p/q\st K_{p:q}\esub G\}$; always $\chic(G)\ge \omega_c(G)$.
\looseness-1

\begin{question}\label{ccques}
For graphs $G$ with $\omega_c(G)\le s$, what is the largest $\rho$ such that  
$\mad(G)<\rho$ implies $\chic(G)\le s$?
(Note: the answer is $s$ when $s$ is an integer.)
\end{question}


Next we apply a special case of Lemma~\ref{avdnbhd} to a further restriction of
acyclic coloring (Definition~\ref{Dacyclic}); like acyclic coloring, it was
introduced by Gr\"unbaum~\cite{Gru}.

\begin{definition}
A {\it star coloring} is an acyclic coloring where the union of any two color
classes induces a forest of stars; equivalently, no $4$-vertex path is
$2$-colored.  The {\it star chromatic number} $\chis(G)$ (also written
$\chi_s(G)$) is the minimum number of colors in such a coloring.

\end{definition}

Every star coloring is an acyclic coloring.  All trees are acyclically
$2$-colorable, but trees with diameter at least $3$ are not star $2$-colorable.
Extensive early results about star colorings and their relationships to other
parameters appear in papers by Fertin, Raspaud, and Reed~\cite{FRR} and by
Albertson et al.~\cite{ACKKR}, without discharging.  Our focus here is on a
structure that yields an upper bound on $\chis(G)$.

\begin{definition}
A set $I$ of vertices is a {\it $2$-independent set} if the distance between
any two vertices of $I$ exceeds $2$.  An {\it $I,F$-partition} of a graph $G$,
introduced by Albertson et al.~\cite{ACKKR}, is a partition of $V(G)$ into sets
$I$ and $F$ such that $I$ is a $2$-independent set and $G[F]$ is a forest.
\end{definition}

\begin{lemma}\label{IFpart}
Every forest is star $3$-colorable.  Hence if a graph $G$ has an
$I,F$-partition, then $\chis(G)\le 4$.
\end{lemma}
\begin{proof}
In a tree $T$, choose a root $r$ and color each vertex $v$ with $d_T(v,r)$,
reduced modulo $3$.  Each connected $2$-colored subgraph is a star consisting
of a vertex and its children.  Using a fourth color on a $2$-independent set
$I$ cannot complete a $2$-colored $4$-vertex path, since no two vertices
with that color have a common neighbor.
\end{proof}



\begin{theorem}[Timmons~\cite{Tim}]\label{star73}
If $\mad(G)<\frac73$, then $G$ has an $I,F$-partition.
\end{theorem}

\begin{proof}
We may assume that no component is a cycle, because in such a component it 
suffices to put one vertex into $I$.  Without $2$-regular components,
Lemma~\ref{avdnbhd} with $t=2$ implies that $G$ has a $1^-$-vertex, a
$3$-thread, or a $3$-vertex with at least five $2$-vertices on its incident
threads.
It therefore suffices to prove these configurations reducible for the existence
of $I,F$-partitions.  A $1^-$-vertex $v$ can be added to the forest in any such
partition of $G-v$.

Let $\la v,w,x,y,z\ra$ be a $3$-thread in $G$; vertices $w,x,y$ each have
degree $2$.  Let $I',F'$ be an $I,F$-partition of $G-\{w,x,y\}$.  If $v$ or $z$
is in $I'$, then add $\{w,x,y\}$ to $F'$ to form an $I,F$-partition of $G$.
Otherwise, add $x$ to $I'$ and $\{w,y\}$ to $F'$, as in Figure~\ref{figIF}.

\begin{figure}[h]
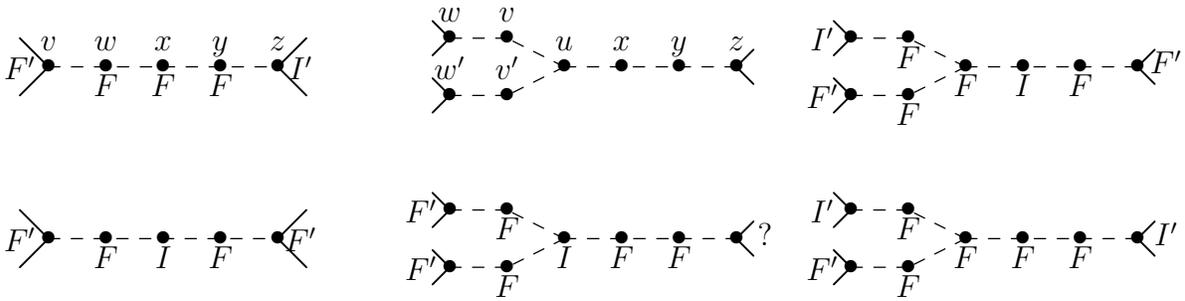

\gpic{
\expandafter\ifx\csname graph\endcsname\relax \csname newbox\endcsname\graph\fi
\expandafter\ifx\csname graphtemp\endcsname\relax \csname newdimen\endcsname\graphtemp\fi
\setbox\graph=\vtop{\vskip 0pt\hbox{%
    \graphtemp=.5ex\advance\graphtemp by 0.270in
    \rlap{\kern 0.150in\lower\graphtemp\hbox to 0pt{\hss $\bu$\hss}}%
    \graphtemp=.5ex\advance\graphtemp by 0.270in
    \rlap{\kern 0.450in\lower\graphtemp\hbox to 0pt{\hss $\bu$\hss}}%
    \graphtemp=.5ex\advance\graphtemp by 0.270in
    \rlap{\kern 0.750in\lower\graphtemp\hbox to 0pt{\hss $\bu$\hss}}%
    \graphtemp=.5ex\advance\graphtemp by 0.270in
    \rlap{\kern 1.050in\lower\graphtemp\hbox to 0pt{\hss $\bu$\hss}}%
    \graphtemp=.5ex\advance\graphtemp by 0.270in
    \rlap{\kern 1.350in\lower\graphtemp\hbox to 0pt{\hss $\bu$\hss}}%
    \special{pn 8}%
    \special{pa 150 270}%
    \special{pa 1350 270}%
    \special{da 0.060}%
    \special{pn 11}%
    \special{pa 0 120}%
    \special{pa 150 270}%
    \special{fp}%
    \special{pa 150 270}%
    \special{pa 0 420}%
    \special{fp}%
    \special{pa 1500 120}%
    \special{pa 1350 270}%
    \special{fp}%
    \special{pa 1350 270}%
    \special{pa 1500 420}%
    \special{fp}%
    \graphtemp=.5ex\advance\graphtemp by 0.300in
    \rlap{\kern 0.000in\lower\graphtemp\hbox to 0pt{\hss $F'$\hss}}%
    \graphtemp=.5ex\advance\graphtemp by 0.300in
    \rlap{\kern 1.470in\lower\graphtemp\hbox to 0pt{\hss $I'$\hss}}%
    \graphtemp=.5ex\advance\graphtemp by 0.390in
    \rlap{\kern 0.450in\lower\graphtemp\hbox to 0pt{\hss $F$\hss}}%
    \graphtemp=.5ex\advance\graphtemp by 0.390in
    \rlap{\kern 0.750in\lower\graphtemp\hbox to 0pt{\hss $F$\hss}}%
    \graphtemp=.5ex\advance\graphtemp by 0.390in
    \rlap{\kern 1.050in\lower\graphtemp\hbox to 0pt{\hss $F$\hss}}%
    \graphtemp=.5ex\advance\graphtemp by 0.150in
    \rlap{\kern 0.150in\lower\graphtemp\hbox to 0pt{\hss $v$\hss}}%
    \graphtemp=.5ex\advance\graphtemp by 0.150in
    \rlap{\kern 1.350in\lower\graphtemp\hbox to 0pt{\hss $z$\hss}}%
    \graphtemp=.5ex\advance\graphtemp by 0.150in
    \rlap{\kern 0.450in\lower\graphtemp\hbox to 0pt{\hss $w$\hss}}%
    \graphtemp=.5ex\advance\graphtemp by 0.150in
    \rlap{\kern 0.750in\lower\graphtemp\hbox to 0pt{\hss $x$\hss}}%
    \graphtemp=.5ex\advance\graphtemp by 0.150in
    \rlap{\kern 1.050in\lower\graphtemp\hbox to 0pt{\hss $y$\hss}}%
    \graphtemp=.5ex\advance\graphtemp by 1.170in
    \rlap{\kern 0.150in\lower\graphtemp\hbox to 0pt{\hss $\bu$\hss}}%
    \graphtemp=.5ex\advance\graphtemp by 1.170in
    \rlap{\kern 0.450in\lower\graphtemp\hbox to 0pt{\hss $\bu$\hss}}%
    \graphtemp=.5ex\advance\graphtemp by 1.170in
    \rlap{\kern 0.750in\lower\graphtemp\hbox to 0pt{\hss $\bu$\hss}}%
    \graphtemp=.5ex\advance\graphtemp by 1.170in
    \rlap{\kern 1.050in\lower\graphtemp\hbox to 0pt{\hss $\bu$\hss}}%
    \graphtemp=.5ex\advance\graphtemp by 1.170in
    \rlap{\kern 1.350in\lower\graphtemp\hbox to 0pt{\hss $\bu$\hss}}%
    \special{pn 8}%
    \special{pa 150 1170}%
    \special{pa 1350 1170}%
    \special{da 0.060}%
    \special{pn 11}%
    \special{pa 0 1020}%
    \special{pa 150 1170}%
    \special{fp}%
    \special{pa 150 1170}%
    \special{pa 0 1320}%
    \special{fp}%
    \special{pa 1500 1020}%
    \special{pa 1350 1170}%
    \special{fp}%
    \special{pa 1350 1170}%
    \special{pa 1500 1320}%
    \special{fp}%
    \graphtemp=.5ex\advance\graphtemp by 1.200in
    \rlap{\kern 0.000in\lower\graphtemp\hbox to 0pt{\hss $F'$\hss}}%
    \graphtemp=.5ex\advance\graphtemp by 1.200in
    \rlap{\kern 1.470in\lower\graphtemp\hbox to 0pt{\hss $F'$\hss}}%
    \graphtemp=.5ex\advance\graphtemp by 1.290in
    \rlap{\kern 0.450in\lower\graphtemp\hbox to 0pt{\hss $F$\hss}}%
    \graphtemp=.5ex\advance\graphtemp by 1.290in
    \rlap{\kern 0.750in\lower\graphtemp\hbox to 0pt{\hss $I$\hss}}%
    \graphtemp=.5ex\advance\graphtemp by 1.290in
    \rlap{\kern 1.050in\lower\graphtemp\hbox to 0pt{\hss $F$\hss}}%
    \graphtemp=.5ex\advance\graphtemp by 0.420in
    \rlap{\kern 2.250in\lower\graphtemp\hbox to 0pt{\hss $\bu$\hss}}%
    \graphtemp=.5ex\advance\graphtemp by 0.420in
    \rlap{\kern 2.550in\lower\graphtemp\hbox to 0pt{\hss $\bu$\hss}}%
    \graphtemp=.5ex\advance\graphtemp by 0.270in
    \rlap{\kern 2.850in\lower\graphtemp\hbox to 0pt{\hss $\bu$\hss}}%
    \graphtemp=.5ex\advance\graphtemp by 0.120in
    \rlap{\kern 2.550in\lower\graphtemp\hbox to 0pt{\hss $\bu$\hss}}%
    \graphtemp=.5ex\advance\graphtemp by 0.120in
    \rlap{\kern 2.250in\lower\graphtemp\hbox to 0pt{\hss $\bu$\hss}}%
    \graphtemp=.5ex\advance\graphtemp by 0.270in
    \rlap{\kern 3.150in\lower\graphtemp\hbox to 0pt{\hss $\bu$\hss}}%
    \graphtemp=.5ex\advance\graphtemp by 0.270in
    \rlap{\kern 3.450in\lower\graphtemp\hbox to 0pt{\hss $\bu$\hss}}%
    \graphtemp=.5ex\advance\graphtemp by 0.270in
    \rlap{\kern 3.750in\lower\graphtemp\hbox to 0pt{\hss $\bu$\hss}}%
    \special{pn 8}%
    \special{pa 2850 270}%
    \special{pa 2550 420}%
    \special{da 0.060}%
    \special{pa 2850 270}%
    \special{pa 2550 120}%
    \special{da 0.060}%
    \special{pa 2850 270}%
    \special{pa 3750 270}%
    \special{da 0.060}%
    \special{pa 2250 420}%
    \special{pa 2550 420}%
    \special{da 0.060}%
    \special{pa 2550 120}%
    \special{pa 2250 120}%
    \special{da 0.060}%
    \special{pn 11}%
    \special{pa 2160 330}%
    \special{pa 2250 420}%
    \special{fp}%
    \special{pa 2250 420}%
    \special{pa 2160 510}%
    \special{fp}%
    \special{pa 2160 30}%
    \special{pa 2250 120}%
    \special{fp}%
    \special{pa 2250 120}%
    \special{pa 2160 210}%
    \special{fp}%
    \special{pa 3840 180}%
    \special{pa 3750 270}%
    \special{fp}%
    \special{pa 3750 270}%
    \special{pa 3840 360}%
    \special{fp}%
    \graphtemp=.5ex\advance\graphtemp by 0.300in
    \rlap{\kern 2.250in\lower\graphtemp\hbox to 0pt{\hss $w'$\hss}}%
    \graphtemp=.5ex\advance\graphtemp by 0.300in
    \rlap{\kern 2.550in\lower\graphtemp\hbox to 0pt{\hss $v'$\hss}}%
    \graphtemp=.5ex\advance\graphtemp by 0.150in
    \rlap{\kern 2.850in\lower\graphtemp\hbox to 0pt{\hss $u$\hss}}%
    \graphtemp=.5ex\advance\graphtemp by 0.000in
    \rlap{\kern 2.550in\lower\graphtemp\hbox to 0pt{\hss $v$\hss}}%
    \graphtemp=.5ex\advance\graphtemp by 0.000in
    \rlap{\kern 2.250in\lower\graphtemp\hbox to 0pt{\hss $w$\hss}}%
    \graphtemp=.5ex\advance\graphtemp by 0.150in
    \rlap{\kern 3.150in\lower\graphtemp\hbox to 0pt{\hss $x$\hss}}%
    \graphtemp=.5ex\advance\graphtemp by 0.150in
    \rlap{\kern 3.450in\lower\graphtemp\hbox to 0pt{\hss $y$\hss}}%
    \graphtemp=.5ex\advance\graphtemp by 0.150in
    \rlap{\kern 3.750in\lower\graphtemp\hbox to 0pt{\hss $z$\hss}}%
    \graphtemp=.5ex\advance\graphtemp by 1.320in
    \rlap{\kern 2.250in\lower\graphtemp\hbox to 0pt{\hss $\bu$\hss}}%
    \graphtemp=.5ex\advance\graphtemp by 1.320in
    \rlap{\kern 2.550in\lower\graphtemp\hbox to 0pt{\hss $\bu$\hss}}%
    \graphtemp=.5ex\advance\graphtemp by 1.170in
    \rlap{\kern 2.850in\lower\graphtemp\hbox to 0pt{\hss $\bu$\hss}}%
    \graphtemp=.5ex\advance\graphtemp by 1.020in
    \rlap{\kern 2.550in\lower\graphtemp\hbox to 0pt{\hss $\bu$\hss}}%
    \graphtemp=.5ex\advance\graphtemp by 1.020in
    \rlap{\kern 2.250in\lower\graphtemp\hbox to 0pt{\hss $\bu$\hss}}%
    \graphtemp=.5ex\advance\graphtemp by 1.170in
    \rlap{\kern 3.150in\lower\graphtemp\hbox to 0pt{\hss $\bu$\hss}}%
    \graphtemp=.5ex\advance\graphtemp by 1.170in
    \rlap{\kern 3.450in\lower\graphtemp\hbox to 0pt{\hss $\bu$\hss}}%
    \graphtemp=.5ex\advance\graphtemp by 1.170in
    \rlap{\kern 3.750in\lower\graphtemp\hbox to 0pt{\hss $\bu$\hss}}%
    \special{pn 8}%
    \special{pa 2850 1170}%
    \special{pa 2550 1320}%
    \special{da 0.060}%
    \special{pa 2850 1170}%
    \special{pa 2550 1020}%
    \special{da 0.060}%
    \special{pa 2850 1170}%
    \special{pa 3750 1170}%
    \special{da 0.060}%
    \special{pa 2250 1320}%
    \special{pa 2550 1320}%
    \special{da 0.060}%
    \special{pa 2550 1020}%
    \special{pa 2250 1020}%
    \special{da 0.060}%
    \special{pn 11}%
    \special{pa 2160 1230}%
    \special{pa 2250 1320}%
    \special{fp}%
    \special{pa 2250 1320}%
    \special{pa 2160 1410}%
    \special{fp}%
    \special{pa 2160 930}%
    \special{pa 2250 1020}%
    \special{fp}%
    \special{pa 2250 1020}%
    \special{pa 2160 1110}%
    \special{fp}%
    \special{pa 3840 1080}%
    \special{pa 3750 1170}%
    \special{fp}%
    \special{pa 3750 1170}%
    \special{pa 3840 1260}%
    \special{fp}%
    \graphtemp=.5ex\advance\graphtemp by 1.350in
    \rlap{\kern 2.100in\lower\graphtemp\hbox to 0pt{\hss $F'$\hss}}%
    \graphtemp=.5ex\advance\graphtemp by 1.440in
    \rlap{\kern 2.550in\lower\graphtemp\hbox to 0pt{\hss $F$\hss}}%
    \graphtemp=.5ex\advance\graphtemp by 1.290in
    \rlap{\kern 2.850in\lower\graphtemp\hbox to 0pt{\hss $I$\hss}}%
    \graphtemp=.5ex\advance\graphtemp by 1.140in
    \rlap{\kern 2.550in\lower\graphtemp\hbox to 0pt{\hss $F$\hss}}%
    \graphtemp=.5ex\advance\graphtemp by 1.050in
    \rlap{\kern 2.100in\lower\graphtemp\hbox to 0pt{\hss $F'$\hss}}%
    \graphtemp=.5ex\advance\graphtemp by 1.290in
    \rlap{\kern 3.150in\lower\graphtemp\hbox to 0pt{\hss $F$\hss}}%
    \graphtemp=.5ex\advance\graphtemp by 1.290in
    \rlap{\kern 3.450in\lower\graphtemp\hbox to 0pt{\hss $F$\hss}}%
    \graphtemp=.5ex\advance\graphtemp by 1.170in
    \rlap{\kern 3.900in\lower\graphtemp\hbox to 0pt{\hss $?$\hss}}%
    \graphtemp=.5ex\advance\graphtemp by 0.420in
    \rlap{\kern 4.350in\lower\graphtemp\hbox to 0pt{\hss $\bu$\hss}}%
    \graphtemp=.5ex\advance\graphtemp by 0.420in
    \rlap{\kern 4.650in\lower\graphtemp\hbox to 0pt{\hss $\bu$\hss}}%
    \graphtemp=.5ex\advance\graphtemp by 0.270in
    \rlap{\kern 4.950in\lower\graphtemp\hbox to 0pt{\hss $\bu$\hss}}%
    \graphtemp=.5ex\advance\graphtemp by 0.120in
    \rlap{\kern 4.650in\lower\graphtemp\hbox to 0pt{\hss $\bu$\hss}}%
    \graphtemp=.5ex\advance\graphtemp by 0.120in
    \rlap{\kern 4.350in\lower\graphtemp\hbox to 0pt{\hss $\bu$\hss}}%
    \graphtemp=.5ex\advance\graphtemp by 0.270in
    \rlap{\kern 5.250in\lower\graphtemp\hbox to 0pt{\hss $\bu$\hss}}%
    \graphtemp=.5ex\advance\graphtemp by 0.270in
    \rlap{\kern 5.550in\lower\graphtemp\hbox to 0pt{\hss $\bu$\hss}}%
    \graphtemp=.5ex\advance\graphtemp by 0.270in
    \rlap{\kern 5.850in\lower\graphtemp\hbox to 0pt{\hss $\bu$\hss}}%
    \special{pn 8}%
    \special{pa 4950 270}%
    \special{pa 4650 420}%
    \special{da 0.060}%
    \special{pa 4950 270}%
    \special{pa 4650 120}%
    \special{da 0.060}%
    \special{pa 4950 270}%
    \special{pa 5850 270}%
    \special{da 0.060}%
    \special{pa 4350 420}%
    \special{pa 4650 420}%
    \special{da 0.060}%
    \special{pa 4650 120}%
    \special{pa 4350 120}%
    \special{da 0.060}%
    \special{pn 11}%
    \special{pa 4260 330}%
    \special{pa 4350 420}%
    \special{fp}%
    \special{pa 4350 420}%
    \special{pa 4260 510}%
    \special{fp}%
    \special{pa 4260 30}%
    \special{pa 4350 120}%
    \special{fp}%
    \special{pa 4350 120}%
    \special{pa 4260 210}%
    \special{fp}%
    \special{pa 5940 180}%
    \special{pa 5850 270}%
    \special{fp}%
    \special{pa 5850 270}%
    \special{pa 5940 360}%
    \special{fp}%
    \graphtemp=.5ex\advance\graphtemp by 0.450in
    \rlap{\kern 4.200in\lower\graphtemp\hbox to 0pt{\hss $F'$\hss}}%
    \graphtemp=.5ex\advance\graphtemp by 0.540in
    \rlap{\kern 4.650in\lower\graphtemp\hbox to 0pt{\hss $F$\hss}}%
    \graphtemp=.5ex\advance\graphtemp by 0.390in
    \rlap{\kern 4.950in\lower\graphtemp\hbox to 0pt{\hss $F$\hss}}%
    \graphtemp=.5ex\advance\graphtemp by 0.240in
    \rlap{\kern 4.650in\lower\graphtemp\hbox to 0pt{\hss $F$\hss}}%
    \graphtemp=.5ex\advance\graphtemp by 0.150in
    \rlap{\kern 4.200in\lower\graphtemp\hbox to 0pt{\hss $I'$\hss}}%
    \graphtemp=.5ex\advance\graphtemp by 0.390in
    \rlap{\kern 5.250in\lower\graphtemp\hbox to 0pt{\hss $I$\hss}}%
    \graphtemp=.5ex\advance\graphtemp by 0.390in
    \rlap{\kern 5.550in\lower\graphtemp\hbox to 0pt{\hss $F$\hss}}%
    \graphtemp=.5ex\advance\graphtemp by 0.270in
    \rlap{\kern 6.000in\lower\graphtemp\hbox to 0pt{\hss $F'$\hss}}%
    \graphtemp=.5ex\advance\graphtemp by 1.320in
    \rlap{\kern 4.350in\lower\graphtemp\hbox to 0pt{\hss $\bu$\hss}}%
    \graphtemp=.5ex\advance\graphtemp by 1.320in
    \rlap{\kern 4.650in\lower\graphtemp\hbox to 0pt{\hss $\bu$\hss}}%
    \graphtemp=.5ex\advance\graphtemp by 1.170in
    \rlap{\kern 4.950in\lower\graphtemp\hbox to 0pt{\hss $\bu$\hss}}%
    \graphtemp=.5ex\advance\graphtemp by 1.020in
    \rlap{\kern 4.650in\lower\graphtemp\hbox to 0pt{\hss $\bu$\hss}}%
    \graphtemp=.5ex\advance\graphtemp by 1.020in
    \rlap{\kern 4.350in\lower\graphtemp\hbox to 0pt{\hss $\bu$\hss}}%
    \graphtemp=.5ex\advance\graphtemp by 1.170in
    \rlap{\kern 5.250in\lower\graphtemp\hbox to 0pt{\hss $\bu$\hss}}%
    \graphtemp=.5ex\advance\graphtemp by 1.170in
    \rlap{\kern 5.550in\lower\graphtemp\hbox to 0pt{\hss $\bu$\hss}}%
    \graphtemp=.5ex\advance\graphtemp by 1.170in
    \rlap{\kern 5.850in\lower\graphtemp\hbox to 0pt{\hss $\bu$\hss}}%
    \special{pn 8}%
    \special{pa 4950 1170}%
    \special{pa 4650 1320}%
    \special{da 0.060}%
    \special{pa 4950 1170}%
    \special{pa 4650 1020}%
    \special{da 0.060}%
    \special{pa 4950 1170}%
    \special{pa 5850 1170}%
    \special{da 0.060}%
    \special{pa 4350 1320}%
    \special{pa 4650 1320}%
    \special{da 0.060}%
    \special{pa 4650 1020}%
    \special{pa 4350 1020}%
    \special{da 0.060}%
    \special{pn 11}%
    \special{pa 4260 1230}%
    \special{pa 4350 1320}%
    \special{fp}%
    \special{pa 4350 1320}%
    \special{pa 4260 1410}%
    \special{fp}%
    \special{pa 4260 930}%
    \special{pa 4350 1020}%
    \special{fp}%
    \special{pa 4350 1020}%
    \special{pa 4260 1110}%
    \special{fp}%
    \special{pa 5940 1080}%
    \special{pa 5850 1170}%
    \special{fp}%
    \special{pa 5850 1170}%
    \special{pa 5940 1260}%
    \special{fp}%
    \graphtemp=.5ex\advance\graphtemp by 1.350in
    \rlap{\kern 4.200in\lower\graphtemp\hbox to 0pt{\hss $F'$\hss}}%
    \graphtemp=.5ex\advance\graphtemp by 1.440in
    \rlap{\kern 4.650in\lower\graphtemp\hbox to 0pt{\hss $F$\hss}}%
    \graphtemp=.5ex\advance\graphtemp by 1.290in
    \rlap{\kern 4.950in\lower\graphtemp\hbox to 0pt{\hss $F$\hss}}%
    \graphtemp=.5ex\advance\graphtemp by 1.140in
    \rlap{\kern 4.650in\lower\graphtemp\hbox to 0pt{\hss $F$\hss}}%
    \graphtemp=.5ex\advance\graphtemp by 1.050in
    \rlap{\kern 4.200in\lower\graphtemp\hbox to 0pt{\hss $I'$\hss}}%
    \graphtemp=.5ex\advance\graphtemp by 1.290in
    \rlap{\kern 5.250in\lower\graphtemp\hbox to 0pt{\hss $F$\hss}}%
    \graphtemp=.5ex\advance\graphtemp by 1.290in
    \rlap{\kern 5.550in\lower\graphtemp\hbox to 0pt{\hss $F$\hss}}%
    \graphtemp=.5ex\advance\graphtemp by 1.170in
    \rlap{\kern 6.000in\lower\graphtemp\hbox to 0pt{\hss $I'$\hss}}%
    \hbox{\vrule depth1.470in width0pt height 0pt}%
    \kern 6.000in
  }%
}%
}

\vspace{-1pc}

\caption{Reducibility cases for Theorem~\ref{star73}\label{figIF}}
\end{figure}

Finally, consider a $3$-vertex $u$ with at least five $2$-vertices on its
incident threads.  If one of the threads has three $2$-vertices, then $G$ has a
$3$-thread.  Otherwise, $u$ has at least two incident $2$-threads plus a
$2$-neighbor on the third incident thread.  It suffices to consider a
$2$-thread $\la u,x,y,z\ra$ and two $1$-threads $\la u,v,w\ra$ and
$\la u,v',w'\ra$ incident to $u$ (see Figure~\ref{figIF}).  Let $R=\{w,w'\}$
and $S=\{u,x,y,v,v'\}$, and let $G'=G-S$.  Let $I',F'$ be an $I,F$-partition of
$G'$.  If $R\cap I'=\nul$, then add $u$ to $I'$ and the rest of $S$ to $F'$.
If $R\cap I'\ne\nul$ and $z\notin I'$, then add $x$ to $I'$ and the rest of $S$
to $F'$.  If $R\cap I'\ne\nul$ and $z\in I'$, then add all of $S$ to $F'$.  In
each case, the resulting sets form an $I,F$-partition of $G$.
\end{proof}

With more detailed discharging, Bu et al.~\cite{BCMRW} proved that
$\mad(G)<\FR{26}{11}$ yields an $I,F$-partition, and then
Brandt et al.~\cite{BFKLSY} proved that $\mad(G)<\FR52$ suffices.
This result is sharp, as infinitely many examples with average degree $\FR52$
have no $I,F$-partition (Exercise~\ref{noIF}).  However, the optimal value of
$\mad(G)$ implying star $4$-colorability is not known.  Chen, Raspaud, and
Wang~\cite{CRW1,CRW2} proved that $\chis(G)\le 8$ when $\mad(G)<3$ and that
$\chis(G)\le6$ when $\Delta(G)=3$ (the latter is sharp).
Unfortunately, no bound on $\chis(G)$ (or $\chia(G)$)
can be implied by $\mad(G)<4$.  The constructions for this in
Exercise~\ref{nobound} have average degree tending to $4$; what happens for
bounds between $3$ and $4$ remains open.

\begin{question}
For $3<b<4$, is it true that $\chis(G)$ or $\chia(G)$ is bounded when
$\mad(G)<b$?  Given $k$, what is the minimum number of vertices in a graph $G$
with $\mad(G)<4$ and $\chis(G)>k$ (or $\chia(G)>k$)?
\end{question}


\medskip

{\small


\begin{exercise}\label{3nbr}
Given $\avd(G)<4$, prove that $G$ has a $2^-$-vertex or a $3$-vertex with
a $5^-$-neighbor.  Explain why we cannot place any bound on the smallest
degree of a neighbor of a $2$-vertex, both by construction and by explaining
how the proof would fail.
\end{exercise}


\begin{exercise}\label{mindeg}
Given $0\le j<k$, let $G$ be a graph with $\delta(G)=k$.  Determine the largest
$\rho$ such that $\avd(G)< k+\rho$ guarantees that $G$ has a $k$-vertex having
more than $j$ neighbors of degree $k$.
\end{exercise}

\begin{exercise}\label{threadsharp}
Show that Lemma~\ref{avdnbhd} is sharp.  For each $t\in\NN$ construct
infinitely many examples with average degree $2+\FR1{2t-1}$ in which none of
the specified configurations occurs.
\end{exercise}

\begin{exercise}\label{circhr}
Prove that the configurations in Lemma~\ref{avdnbhd} are reducible
for $\chi_c(G)\le 2+\FR1t$.  Conclude that $\chi_c(G)\le 2+\FR1t$ when
$g_o(G)\ge2t+1$ and $\mad(G)<2+\FR1{2t-1}$
\end{exercise}

\begin{exercise}\label{52CKY}
(Cranston--Kim--Yu~\cite{CKY1})
Let $G$ be a connected graph with at least four vertices.
Prove that if $\avd(G)<\FR52$ and $\delta(G)\ge2$, then $G$ contains a
$2$-thread or a $3$-vertex having three $2$-neighbors, one of which has a
second $3$-neighbor.
\end{exercise}

\begin{exercise}\label{52struc}
(Cranston--Jahanbekam--West~\cite{CJW})
Prove that if $\avd(G)<\FR52$ and $G$ is connected, then $G$ contains a
$3^-$-vertex with a $1$-neighbor, a $4^-$-vertex with two $2^-$-neighbors, or a
$5^+$-vertex $v$ with at least $\FR{d(v)-1}2$ $2^-$-neighbors.  (Comment: These
configurations are reducible for the ``${1,2}$-Conjecture'' of Przyby\l o and
Woz\'niak~\cite{PW}.  Although that proves the conjecture when
$\Mad(G)<\FR52$,~\cite{PW} proved the stronger result that the conjecture holds
for $3$-colorable graphs.)
\end{exercise}

\begin{exercise}\label{madgen}
Prove that if $\avd(G)<k+\rho$ with $0<\rho\le \FR k{k+1}$, then $G$ contains
a $(k-1)^-$-vertex, two adjacent $k$-vertices, or a $(k+1)$-vertex with more
than $(\FR1\rho-1) k$ neighbors having degree $k$.  Construct sharpness
examples with $\avd(G)=k+\rho$ when $\rho=\FR12$ and when $\rho=\FR k{k+1}$
(the latter may have maximum degree $k+1$ or $k+2$).
\end{exercise}

\begin{exercise}\label{ex:square}
Prove that if $\Delta(G)=k\ge 3$ and $\avd(G)<k-\frac2{k^2+1}$, then
$G$ contains one of the following configurations: (C1) a $(k-2)^-$-vertex,
(C2) two adjacent $(k-1)$-vertices, (C3) a $k$-vertex with two
$(k-1)$-neighbors, (C4) two adjacent $k$-vertices each having a
$(k-1)$-neighbor, or (C5) a $k$-vertex having three $k$-neighbors such that
each is adjacent to a $(k-1)$-vertex.
\end{exercise}

\begin{exercise}\label{IFnonred}
Argue that the configuration consisting of two $3$-vertices each incident
to two $2$-threads is not reducible for the existence of $I,F$-partitions.
\end{exercise}

\begin{exercise}\label{noIF}
Obtain $G$ from an $n$-cycle by attaching to each vertex a new edge whose other
endpoint lies on a new triangle ($|V(G)|=4n$).  Prove that $G$ has
no $I,F$-partition, despite $\mad(G)=\FR52$.
\end{exercise}

%
%
%


\begin{exercise}\label{nobound}
Let $G_n$ be the graph obtained from the complete graph $K_n$ by subdividing
every edge once.  Determine $\mad(G_n)$.
Prove $\chis(G_n)>k$ for $n>k^2$ and $\chia(G_n)>k$ for $n>2k^2$.
\end{exercise}

}

\section{Discharging on Plane Graphs}\label{seclight}\label{secplanar}

The context of bounded $\mad(G)$ remains valid when we study planar graphs.
Euler's Formula for connected graphs embedded in the plane (``plane graphs'')
is $n-m+p=2$, where $n$, $m$, and $p$ count the vertices, edges, and faces
(``points'' of the dual graph).  When loops and multiedges are forbidden, each
face boundary has length at least $3$, yielding $m\le 3n-6$.  Since the
degree-sum is twice the number of edges, we obtain $\mad(G)<6$.  When the girth
(minimum cycle length) is $g$, the inequality generalizes to
$m\le \FR g{g-2}(n-2)$.  Deleting edges cannot create short cycles, so
$\mad(G)<\FR {2g}{g-2}$ when $G$ is a planar graph with girth $g$.

Some results on planar graphs or planar graphs with large girth hold more
generally for graphs satisfying the corresponding bound on $\mad(G)$.  However,
planar graphs form a highly restricted subfamily, and often stronger results
hold when planarity is also required.

The discharging method is well suited to exploit planarity.  The distinctive
feature of discharging for a plane graph $G$ is that charge can also be
assigned to faces, which are vertices in the dual graph $G^*$.  Since $G^*$ is
also planar, $\mad(G^*)<6$ and we can use discharging on both $G$ and $G^*$.
Even stronger is to use the two graphs in combination via Euler's Formula.
This leads to three common natural ways to assign charge on plane graphs.

\begin{proposition}
Let $V(G)$ and $F(G)$ be the sets of vertices and faces in a plane graph $G$,
and let $\ell(f)$ denote the length of a face $f$.  The following equalities
hold for $G$.
$$
\begin{array}{cll}
\SM v{V(G)}(d(v)-6)+\SM f{F(G)}(2\ell(f)-6)=-12&&{\rm vertex~charging}\\
\phantom{\Big|}
\SM v{V(G)}(2d(v)-6)+\SM f{F(G)}(\ell(f)-6)=-12&&{\rm face~charging}\\
\phantom{\Big|}
\SM v{V(G)}(d(v)-4)+\SM f{F(G)}(\ell(f)-4)=-8&&{\rm balanced~charging}
\end{array}
$$
\end{proposition}
\begin{proof}
Multiply Euler's Formula by $-6$ or $-4$ and split the term for edges to obtain
the three formulas below.
$$-6n+2m+4m-6p=-12;\quad -6n+4m+2m-6f=-12;\quad -4n+2m+2m-4f=-8.$$
Substitute $\FR12\SM v{V(G)}d(v)$ for the first occurrence of $m$ and
$\FR12\SM f{F(G)} \ell(f)$ for the second in each equation, and then 
collect the contributions by vertices and by faces.
\end{proof}

The initial charge assigned to a vertex or face is the corresponding term
in the equation being used.  The initial charges are not degree or length,
but rather an adjustment of those quantities based on Euler's Formula.
A vertex or face now is ``happy'' when it reaches nonnegative charge.  When
specified configurations are assumed not to occur, making every vertex and face
happy provides a contradiction in the same way as with degree charging;
it makes the left side nonnegative, while the right side is negative.

In principle, any result provable by one of these discharging methods can also
be proved by the others.  However, depending on the context, one type of
discharging may lead to simpler proofs than the others.
For triangulations, such as in the Four Color Problem, vertex charging is
appropriate.  All the faces have charge $0$, and often they can be ignored.
In this situation, vertex charging is very much equivalent to degree charging,
and such proofs can be phrased equally well using either approach.
For $3$-regular plane graphs, face charging is appropriate, with each vertex
given charge $0$.  Under balanced charging when $G$ and its dual $G^*$ are
simple, $3$-vertices and $3$-faces are the only objects needing charge; those
with degree or length at least $5$ have spare charge to give away.

In subsequent sections we will present some results about planar graphs that
use additional tools of discharging.  Here we emphasize the classical topic of
``light subgraphs'', that is, subgraphs of small weight (small degree-sum).
As we have seen, light subgraphs can be reducible configurations for coloring
probems; the topic was surveyed by Jendrol' and Voss~\cite{JV}.  We prove some
results by balanced charging or face charging that were originally proved by
vertex charging.

The best-known result on light edges is Kotzig's Theorem~\cite{Ko}: every
$3$-connected planar graph has an edge of weight at most $13$.
A {\it normal plane map} is a plane multigraph such that every vertex degree
and face length is at least $3$; in particular, every plane graph with
minimum degree at least $3$ is a normal plane map.  Jendrol'~\cite{Je,JV} gave
a short proof that every normal plane map $G$ has an edge with weight at most
11 or a $3$-vertex with a $10^-$-neighbor.  We modify the proof by Jendrol' to
obtain the earlier stronger extension of Kotzig's Theorem by Borodin~\cite{B89}.

\begin{lemma}[Borodin~\cite{B89}]\label{kotzig}\label{3alt4}
Every normal plane map $G$ has an edge with weight at most 11 or a $4$-cycle
through two $3$-vertices and a common $10^-$-neighbor.
\end{lemma}
\begin{proof}
Assume that $G$ has no light edge.  If any face $F$ has length more than $3$,
then adding a chord joining the neighbors along $F$ of a vertex on $F$ with
least degree cannot introduce a light edge.  Hence we may assume that every
face has length $3$.  Use vertex charging.

\medskip\noindent
(R1) Every $5^-$ vertex $v$ takes charge $\FR{6-d(v)}{d(v)}$ from each
$7^+$-neighbor.

\medskip
Since $G$ is a triangulation with no edges of weight at most $11$, a $k$-vertex
loses charge to at most $\FL{\FR k2}$ vertices.  Since a $7$-vertex loses
charge only to $5$-vertices, it loses at most $3\cdot \FR15$ and remains
positive.  An $8$-vertex sends charge only to $4^+$-vertices and loses at most
$4\cdot\FR12$, remaining nonnegative.  For $d(v) \ge 9$, neighbors of $v$ may
have degree $3$ and take charge $1$; but the final charge is at least 
$\ceil{\frac{d(v)}2}  − 6$, which is nonnegative when $d(v) \ge 11$.

Finally, suppose $d(v)\in\{9,10\}$.  Since faces have length $3$, the neighbors
of $v$ form a closed walk of length $d(v)$ when followed in order, and each
$3$-vertex appears only once in this walk.  With light edges and the specified 
$4$-cycles forbidden, $3$-vertices must be separated by at least three steps
along this walk.

Since there are no light edges, each $9$-vertex has at most four
$5^-$-neighbors.  If it has at least three $3$-neighbors, then it has exactly
three and loses charge to no other vertices.  Hence a $9$-vertex loses at
most $\max\{3\cdot 1,2\cdot 1+2\cdot\FR12\}$ and ends happy.  Similarly,
a $10$-vertex has at most five $5^-$-neighbors, and if it has at least three
$3$-neighbors then it has at most four $5^-$-neighbors.  It loses at most
$\max\{4\cdot 1,2\cdot 1+3\cdot\FR12\}$ and ends happy.
\end{proof}

We give an application of Lemma~\ref{3alt4}.  A {\it decomposition} of a graph
expresses it as a union of edge-disjoint subgraphs; its {\it size} is the
number of subgraphs.  The {\it arboricity} of a graph $G$, written
$\Upsilon(G)$, is the minimum size of a decomposition of $G$ into forests.  A
{\it linear forest} is a forest whose components are paths.  The {\it linear
arbority}, written $\larb(G)$, is the minimum size of a decomposition into
linear forests.

Trivially $\larb(G)\ge \CL{\Delta(G)/2}$, but equality for $2r$-regular graphs
would need $2$-regular color classes, which would contain cycles.  Akiyama,
Exoo, and Harary~\cite{AEH1,AEH2} conjectured that always
$\larb(G) \le \ceil{(\Delta(G)+1)/2}$.  Together, \cite{AEH1,AEH2,EP,Gul}
proved it for $\Delta(G)\in\{1,2,3,4,5,6,8,10\}$.  Given $\epsilon>0$,
Alon~\cite{Alon} proved $\larb(G)\le(\FR12+\epsilon)\Delta(G)$ for $\Delta(G)$
sufficiently large.  When $G$ is planar, the conjecture was proved for
$\Delta(G)\ge9$ by Wu~\cite{jlw} (presented below) and for $\Delta(G)=7$ by Wu
and Wu~\cite{WW}, so the proof is complete for planar graphs.


\begin{theorem}[Wu~\cite{jlw}]\label{larbthm}
If $G$ is a planar graph with $\Delta(G)\ge9$, then
$\larb(G) \le \ceil{\FR{\Delta(G)+1}2}$.
\end{theorem}
\begin{proof}
We show for $t\ge5$ that every planar graph $G$ with $\Delta(G)<2t$ decomposes
into $t$ linear forests.  View the linear forests as $t$ color classes in an
edge-coloring.

For an edge $uv$ with weight at most $2t+1$, consider such a decomposition of
$G-uv$.  Since $d_{G-uv}(u)+d_{G-uv}(v)<2t$, fewer than $t$ colors appear twice
at $u$ or twice at $v$ or once at each.  Thus a color is available at $uv$ to
extend the decomposition to $G$.

Hence we may assume weight at least $2t+2$ for every edge.  Since
$\Delta(G)<2t$, this yields $\delta(G)\ge3$.  Now Lemma~\ref{3alt4} yields a
$4$-cycle $[u,x,v,y]$ in $G$ with $d(u)=d(v)=3$.  Let $u'$ and $v'$ be the
remaining neighbors of $u$ and $v$, respectively (see Figure~\ref{figlarb}).
Note that $\{u,v\}=\{u',v'\}$ is forbidden, but $u'=v'$ is possible, requiring
a similar analysis that we omit.
\looseness-1

\begin{figure}[h]
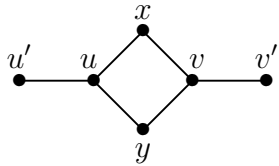

\gpic{
\expandafter\ifx\csname graph\endcsname\relax \csname newbox\endcsname\graph\fi
\expandafter\ifx\csname graphtemp\endcsname\relax \csname newdimen\endcsname\graphtemp\fi
\setbox\graph=\vtop{\vskip 0pt\hbox{%
    \graphtemp=.5ex\advance\graphtemp by 0.362in
    \rlap{\kern 0.103in\lower\graphtemp\hbox to 0pt{\hss $\bu$\hss}}%
    \graphtemp=.5ex\advance\graphtemp by 0.362in
    \rlap{\kern 0.491in\lower\graphtemp\hbox to 0pt{\hss $\bu$\hss}}%
    \graphtemp=.5ex\advance\graphtemp by 0.621in
    \rlap{\kern 0.750in\lower\graphtemp\hbox to 0pt{\hss $\bu$\hss}}%
    \graphtemp=.5ex\advance\graphtemp by 0.362in
    \rlap{\kern 1.009in\lower\graphtemp\hbox to 0pt{\hss $\bu$\hss}}%
    \graphtemp=.5ex\advance\graphtemp by 0.362in
    \rlap{\kern 1.397in\lower\graphtemp\hbox to 0pt{\hss $\bu$\hss}}%
    \graphtemp=.5ex\advance\graphtemp by 0.103in
    \rlap{\kern 0.750in\lower\graphtemp\hbox to 0pt{\hss $\bu$\hss}}%
    \special{pn 11}%
    \special{pa 491 362}%
    \special{pa 750 621}%
    \special{fp}%
    \special{pa 750 621}%
    \special{pa 1009 362}%
    \special{fp}%
    \special{pa 1009 362}%
    \special{pa 750 103}%
    \special{fp}%
    \special{pa 750 103}%
    \special{pa 491 362}%
    \special{fp}%
    \special{pa 103 362}%
    \special{pa 491 362}%
    \special{fp}%
    \special{pa 1009 362}%
    \special{pa 1397 362}%
    \special{fp}%
    \graphtemp=.5ex\advance\graphtemp by 0.259in
    \rlap{\kern 0.103in\lower\graphtemp\hbox to 0pt{\hss $u'$\hss}}%
    \graphtemp=.5ex\advance\graphtemp by 0.259in
    \rlap{\kern 0.466in\lower\graphtemp\hbox to 0pt{\hss $u$\hss}}%
    \graphtemp=.5ex\advance\graphtemp by 0.724in
    \rlap{\kern 0.750in\lower\graphtemp\hbox to 0pt{\hss $y$\hss}}%
    \graphtemp=.5ex\advance\graphtemp by 0.259in
    \rlap{\kern 1.397in\lower\graphtemp\hbox to 0pt{\hss $v'$\hss}}%
    \graphtemp=.5ex\advance\graphtemp by 0.259in
    \rlap{\kern 1.034in\lower\graphtemp\hbox to 0pt{\hss $v$\hss}}%
    \graphtemp=.5ex\advance\graphtemp by 0.000in
    \rlap{\kern 0.750in\lower\graphtemp\hbox to 0pt{\hss $x$\hss}}%
    \hbox{\vrule depth0.724in width0pt height 0pt}%
    \kern 1.500in
  }%
}%
}

\vspace{-1pc}

\caption{Reducible configuration for Theorem~\ref{larbthm}\label{figlarb}}
\end{figure}

The weight restrictions and $\Delta(G)\le2t-1$ imply that each of $x,y,u',v'$
has degree $2t-1$ in $G$.  Hence a linear $t$-decomposition of $G-\{u,v\}$
has $2t-3$ colored edges at both $x$ and $y$ and $2t-2$ colored edges at both
$u'$ and $v'$.  We conclude that at most one color is missing at each of these
vertices; call a vertex ``bad'' if a color is missing.  For a vertex $z$, let
$C(z)$ and $C'(z)$ denote the sets of colors appearing at $z$ exactly once and
at most once, respectively.

{\it Case 1: Neither $x$ nor $y$ is bad.}
Here $|C(x)|=|C(y)|=3$.
First color $uu'$ from $C'(u')$, then $ux$ from $C(x)$, and then $uy$ from
$C(y)$.  Make the colors on $ux$ and $uy$ differ.  If $u'$ is bad, then the
color on $uu'$ may be used on $ux$ or $uy$, but otherwise the three edges
have distinct colors.  Two colors from each of $C(x)$ and $C(y)$ remain, and we
ensure these remaining pairs are not equal.  Now we can choose distinct
colors on the three edges at $v$.

{\it Case 2: One of $x$ and $y$ is bad.}
We may assume by symmetry that $x$ is bad; still $|C(y)|=3$.  Give $ux$ the
color missing at $x$, and give $xv$ the one color in $C(x)$.  Color $vv'$ from
$C'(v')$, different from the color on $xv$ if $v'$ is not bad.  Now color $vy$
from $C(y)$ avoiding the colors on $xv$ and $vv'$.  Now color $uu'$ from
$C'(u')$, and choose a color for $uy$ from $C(y)$ avoiding that and the color
on $vy$.  Note that $uu'$ or $uy$ may have the same color as $ux$.

{\it Case 3: $x$ and $y$ are both bad.}
If the one missing color at each of $x$ and $y$ is the same, then use it on
$ux$ and $vy$.  Now use the color in $C(x)$ on $xv$ and the color in $C(y)$
on $yu$.  If $u'$ is bad, then its missing color can be used on $uu'$; 
otherwise, color $uu'$ from $C(u')$ to avoid the color on $yu$.  The
symmetric argument applies to color $vv'$.

If the missing colors at $x$ and $y$ are different, then use the color missing
at $x$ on $ux$ and $xv$, and use the color missing at $y$ on $uy$ and $yv$.
If a color is missing at $u'$, use it on $uu'$, and then use any color from
$C'(v')$ on $vv'$.  By symmetry, then, only $|C(u')|=|C(v')|=2$ remains, and it
suffices to give $uu'$ and $vv'$ distinct colors from these sets.
\end{proof}

When attention is confined to planar graphs, the concern about regular graphs
with large even degree vanishes.  Wu~\cite{jlw} proved that $\larb(G)=\cdgh$
when $G$ is planar with $\Delta(G)\ge 13$ (see Exercise~\ref{wu}).  Cygan,
Kowalik, and Lu\v{z}ar~\cite{CKL} proved that $\Delta(G)\ge 10$ suffices
and conjectured that $\Delta(G)\ge 6$ suffices.


A famous result on light subgraphs concerns light triangles in planar graphs
with minimum degree $5$.  It was conjectured for triangulations by Kotzig and
proved in stronger form by Borodin; see Exercise~\ref{trisharp} for a sharpness
example.  A further strengthening to a more detailed statement was proved
by Borodin and Ivanova~\cite{BI13b}.  We sketch the proof here because an
interesting wrinkle in the discharging rules can be viewed as ``redirecting''
charge.

\begin{theorem}[Borodin~\cite{B1989}]\label{lighttri}
If $G$ is a simple plane graph with $\delta(G)\ge5$, then $G$ has a $3$-face
with weight at most $17$, and the bound is sharp.
\end{theorem}
\begin{proof}
(sketch) For sharpness, add a vertex in each face of the dodecahedron joined to
the vertices of that face.  The new vertices have degree 5, and the old ones
now have degree 6.  Every face has one new vertex and two old vertices for
total weight 17.

To prove the bound, consider an edge-maximal counterexample $G$.  Every vertex
on a $4^+$-face has degree $5$, since adding a triangular chord at a
$6^+$-vertex would create only heavy $3$-faces, producing a counterexample
containing $G$.

If there is no light triangle, then use vertex charging and move charge as
follows:

\medskip\noindent
(R1) Each $4^+$-face gives $\FR12$ to each incident vertex.

\noindent
(R2) Each $7$-vertex gives $\FR13$ to each $5$-neighbor.

\noindent
(R3) Each $8^+$-vertex gives $\FR14$ through each incident $3$-face to its
$5$-neighbors on that face, split equally if there are two such neighbors.
\medskip

It suffices to show that final charges are nonnegative.  The discharging rules
are chosen so that $4^+$-faces and $7^+$-vertices do not give away too much
charge.  (Maximality of the counterexample puts a $7$-vertex on seven
triangles, and absence of light triangles then restricts a $7$-vertex to have
at most three $5$-neighbors).

Hence the task is to prove that a $5$-vertex $v$ gains charge $1$.  When all
incident faces are triangles, avoiding light triangles restricts $v$ to have at
most two $5$-neighbors.  It also forces the other neighbors to be
$8^+$-vertices when $v$ has two $5$-neighbors.  In this and the remaining
cases (such as when $v$ lies on a $4^+$-face), it is easy to check that $v$
receives enough charge.
\end{proof}

In each discharging rule in Theorem~\ref{lighttri}, the charge given away is
the most that object can afford to lose.  Only $5$-vertices need to gain charge.
Hence it would be natural to have $8^+$-vertices give $\FR14$ to each
$5$-neighbor.  However, this would give only $\FR34$ to a $5$-vertex having two
$5$-neighbors and three $8$-neighbors.  The $5$-vertex needs to get more when
it has two $8$-neighbors on a single triangle.  Guiding the charge from
$8^+$-neighbors through the incident triangles is a way to arrange that.

\medskip
Our next structure theorem is a stronger version of the $5$-degeneracy that
follows from $\mad(G)<6$ for planar graphs.  It has several applications.
 
\begin{lemma}\label{12vert}
Every planar graph has a $5^-$-vertex with at most two $12^+$-neighbors.
\end{lemma}
\begin{proof}
We may assume that $G$ is a triangulation, since adding an edge cannot 
give any vertex the desired property.

Assume that $G$ has no such vertex, so $\delta(G)\ge3$.  Use degree charging;
note that $\mad(G)<6$.  Every $5^-$-vertex has at least three $12^+$-neighbors.
Let each $5^-$-vertex $u$ take $\FR{6-d(u)}3$ from each $12^+$-neighbor.
Now $5^-$-vertices are happy, and $j$-vertices with $6\le j\le 11$ lose no
charge, so it suffices to show that $12^+$-vertices do not lose too much.

Let $v$ be a $12^+$-vertex.  Since $G$ is a triangulation, the neighbors of $v$
form a cycle $C$, possibly having chords.  Let $H$ be the subgraph of $C$
induced by its vertices having degree at most $5$ in $G$.  Each $5^-$-vertex
$w$ has at least three $12^+$-neighbors, so $d_H(w)\le d_G(w)-3$.  If all
neighbors of $v$ have degree $5$, then $v$ loses $\FR{d(v)}3$ and ends with
$\FR23d(v)$, which is at least $8$.

Otherwise, the components of $H$ are paths (bold in Figure~\ref{figtwo12}).
Combine such a $k$-vertex path $P$ with the next vertex on $C$, which receives
no charge from $v$.  If $k=1$, then $v$ gives at most $1$ to these two
vertices.  If $k>1$, then $v$ gives at most $\FR13$ to internal vertices of $P$
(degree at least $5$) and at most $\FR23$ to its endpoints (degree at least
$4$).  Hence the $k+1$ vertices receive at most $0+2(\FR23)+(k-2)(\FR13)$ from
$v$.  This equals $\FR{k+2}3$, which is less than $\FR{k+1}2$ when $k>1$.
Hence $v$ loses in total at most $\FR{d(v)}2$, leaving at least $6$ when
$d(v)\ge12$.
\end{proof}

\begin{figure}[h]
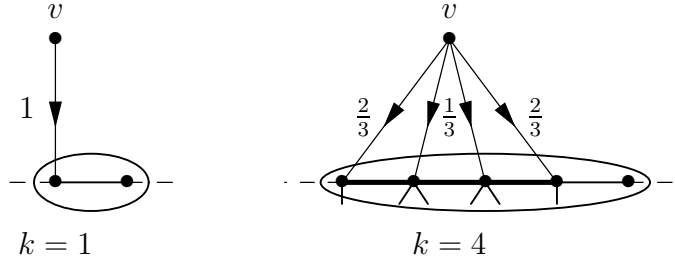

\gpic{
\expandafter\ifx\csname graph\endcsname\relax \csname newbox\endcsname\graph\fi
\expandafter\ifx\csname graphtemp\endcsname\relax \csname newdimen\endcsname\graphtemp\fi
\setbox\graph=\vtop{\vskip 0pt\hbox{%
    \graphtemp=.5ex\advance\graphtemp by 0.150in
    \rlap{\kern 0.300in\lower\graphtemp\hbox to 0pt{\hss $\bu$\hss}}%
    \graphtemp=.5ex\advance\graphtemp by 0.900in
    \rlap{\kern 0.300in\lower\graphtemp\hbox to 0pt{\hss $\bu$\hss}}%
    \graphtemp=.5ex\advance\graphtemp by 0.900in
    \rlap{\kern 0.675in\lower\graphtemp\hbox to 0pt{\hss $\bu$\hss}}%
    \special{pn 8}%
    \special{pa 300 150}%
    \special{pa 300 600}%
    \special{fp}%
    \special{sh 1.000}%
    \special{pa 330 480}%
    \special{pa 300 600}%
    \special{pa 270 480}%
    \special{pa 330 480}%
    \special{fp}%
    \special{pa 300 600}%
    \special{pa 300 900}%
    \special{fp}%
    \special{pn 11}%
    \special{pa 300 900}%
    \special{pa 675 900}%
    \special{fp}%
    \special{pn 8}%
    \special{pa 300 900}%
    \special{pa 0 900}%
    \special{da 0.075}%
    \special{pa 675 900}%
    \special{pa 975 900}%
    \special{da 0.075}%
    \special{pn 11}%
    \special{ar 488 900 300 150 0 6.28319}%
    \graphtemp=.5ex\advance\graphtemp by 0.000in
    \rlap{\kern 0.300in\lower\graphtemp\hbox to 0pt{\hss $v$\hss}}%
    \graphtemp=.5ex\advance\graphtemp by 0.525in
    \rlap{\kern 0.150in\lower\graphtemp\hbox to 0pt{\hss 1\hss}}%
    \graphtemp=.5ex\advance\graphtemp by 1.237in
    \rlap{\kern 0.300in\lower\graphtemp\hbox to 0pt{\hss $k=1$\hss}}%
    \graphtemp=.5ex\advance\graphtemp by 0.150in
    \rlap{\kern 2.362in\lower\graphtemp\hbox to 0pt{\hss $\bu$\hss}}%
    \graphtemp=.5ex\advance\graphtemp by 0.900in
    \rlap{\kern 1.800in\lower\graphtemp\hbox to 0pt{\hss $\bu$\hss}}%
    \graphtemp=.5ex\advance\graphtemp by 0.900in
    \rlap{\kern 2.175in\lower\graphtemp\hbox to 0pt{\hss $\bu$\hss}}%
    \graphtemp=.5ex\advance\graphtemp by 0.900in
    \rlap{\kern 2.550in\lower\graphtemp\hbox to 0pt{\hss $\bu$\hss}}%
    \graphtemp=.5ex\advance\graphtemp by 0.900in
    \rlap{\kern 2.925in\lower\graphtemp\hbox to 0pt{\hss $\bu$\hss}}%
    \graphtemp=.5ex\advance\graphtemp by 0.900in
    \rlap{\kern 3.300in\lower\graphtemp\hbox to 0pt{\hss $\bu$\hss}}%
    \special{pn 8}%
    \special{pa 2363 150}%
    \special{pa 2700 600}%
    \special{fp}%
    \special{sh 1.000}%
    \special{pa 2652 486}%
    \special{pa 2700 600}%
    \special{pa 2604 522}%
    \special{pa 2652 486}%
    \special{fp}%
    \special{pa 2700 600}%
    \special{pa 2925 900}%
    \special{fp}%
    \special{pa 2363 150}%
    \special{pa 2475 600}%
    \special{fp}%
    \special{sh 1.000}%
    \special{pa 2475 476}%
    \special{pa 2475 600}%
    \special{pa 2417 491}%
    \special{pa 2475 476}%
    \special{fp}%
    \special{pa 2475 600}%
    \special{pa 2550 900}%
    \special{fp}%
    \special{pa 2363 150}%
    \special{pa 2025 600}%
    \special{fp}%
    \special{sh 1.000}%
    \special{pa 2121 522}%
    \special{pa 2025 600}%
    \special{pa 2073 486}%
    \special{pa 2121 522}%
    \special{fp}%
    \special{pa 2025 600}%
    \special{pa 1800 900}%
    \special{fp}%
    \special{pa 2363 150}%
    \special{pa 2250 600}%
    \special{fp}%
    \special{sh 1.000}%
    \special{pa 2308 491}%
    \special{pa 2250 600}%
    \special{pa 2250 476}%
    \special{pa 2308 491}%
    \special{fp}%
    \special{pa 2250 600}%
    \special{pa 2175 900}%
    \special{fp}%
    \special{pn 28}%
    \special{pa 1800 900}%
    \special{pa 2925 900}%
    \special{fp}%
    \special{pn 11}%
    \special{pa 2925 900}%
    \special{pa 3300 900}%
    \special{fp}%
    \special{pn 8}%
    \special{pa 1800 900}%
    \special{pa 1500 900}%
    \special{da 0.075}%
    \special{pa 3300 900}%
    \special{pa 3600 900}%
    \special{da 0.075}%
    \special{pn 11}%
    \special{ar 2550 900 862 150 0 6.28319}%
    \graphtemp=.5ex\advance\graphtemp by 0.000in
    \rlap{\kern 2.362in\lower\graphtemp\hbox to 0pt{\hss $v$\hss}}%
    \graphtemp=.5ex\advance\graphtemp by 0.562in
    \rlap{\kern 2.362in\lower\graphtemp\hbox to 0pt{\hss $\FR13$\hss}}%
    \graphtemp=.5ex\advance\graphtemp by 1.237in
    \rlap{\kern 2.362in\lower\graphtemp\hbox to 0pt{\hss $k=4$\hss}}%
    \graphtemp=.5ex\advance\graphtemp by 0.562in
    \rlap{\kern 1.912in\lower\graphtemp\hbox to 0pt{\hss $\FR23$\hss}}%
    \graphtemp=.5ex\advance\graphtemp by 0.562in
    \rlap{\kern 2.812in\lower\graphtemp\hbox to 0pt{\hss $\FR23$\hss}}%
    \special{pa 1800 900}%
    \special{pa 1800 1013}%
    \special{fp}%
    \special{pa 2925 900}%
    \special{pa 2925 1013}%
    \special{fp}%
    \special{pa 2100 1013}%
    \special{pa 2175 900}%
    \special{fp}%
    \special{pa 2175 900}%
    \special{pa 2250 1013}%
    \special{fp}%
    \special{pa 2475 1013}%
    \special{pa 2550 900}%
    \special{fp}%
    \special{pa 2550 900}%
    \special{pa 2625 1013}%
    \special{fp}%
    \hbox{\vrule depth1.237in width0pt height 0pt}%
    \kern 3.600in
  }%
}%
}

\vspace{-.8pc}

\caption{Final case for Lemma~\ref{12vert}\label{figtwo12}}
\end{figure}

For an application, consider again the {\it arboricity} $\Upsilon(G)$.
Nash-Williams~\cite{NWi} famously proved that for every graph $G$, the
arboricity equals a trivial lower bound:
always $\Upsilon(G)={\max_{H\esub G}\FR{|E(H)|}{|V(H)|-1}}$.
Although there is a short general proof using matroids, for planar graphs the
existence of $5^-$-vertices permits an inductive decomposition into three
forests.  We include this in Exercise~\ref{planarb} to illustrate a technique
for reducibility with triangulations.

Lemma~\ref{12vert} was improved by Balogh, Kochol, Pluhar, and Yu~\cite{BKPY}
to guarantee a $5^-$vertex having at most two $11^+$-neighbors, proved by a
much longer discharging argument than in Lemma~\ref{12vert}.  The result is
sharp in that ``11'' cannot be replaced by ``10'' (the graph obtained by
adding a vertex of degree $3$ in each face of the icosahedron has three
$10$-neighbors for each $5^-$-vertex).  From this they proved that every
planar graph decomposes into three forests, with one having maximum degree
at most 8.  Lemma~\ref{12vert} allows the degree of the third forest to be
bounded by $9$ (Exercise~\ref{planarb}).  In a somewhat different direction,
a much more detailed strengthening that considers a larger set of light
subgraphs was proved by Borodin and Ivanova~\cite{BI13}; it extends or
strengthens several intermediate results.

The amount by which total charge is negative can be used to prove that actually
many light edges must occur.  Another instance of this technique occurs in
\cite{EHMS}.

\begin{theorem}[Borodin--Sanders~\cite{BS}]\label{manylight}

For any plane graph $G$ with $\delta(G)=5$,
$$2e_{5,5}+e_{5,6}+\FR27e_{5,7}\ge 60,$$
where $e_{i,j}$ is the number of edges with endpoints of degrees $i$ and $j$.
Furthermore, the coefficients in this inequality are sharp.
\end{theorem}
\begin{proof}
Add edges to obtain a triangulation $H$.  No vertex degree decreases, so
$\delta(H)=5$.  Also, since $\delta(G)=5$, no edges incident to vertices having
degree $5$ in $H$ are added.  Hence $e_{5,j}(H)\le e_{5,j}(G)$, and it
suffices to prove the desired lower bounds when $H$ is a triangulation with
minimum degree $5$.

Use vertex charging.  Each $5$-vertex takes $\FR15$ from each $8^+$-neighbor
and $\FR17$ from each $7$-neighbor.  Since $(4/5)j-6>0$ when $j\ge8$, every
$6^+$ vertex ends happy.  Negative charge remains only at $5$-vertices.  Each
$7^-$-neighbor corresponds to an edge along which a $5$-vertex fails to gain
$\FR15$, thereby leaving more negative charge there.  A $7$-neighbor sends
$\FR17$ instead of $\FR15$, thus contributing $\FR{-2}{35}$ to the negative
charge remaining at $5$-vertices.  Edges counted by $e_{5,5}$ affect both
endpoints.  Thus edges counted by $e_{5,5}$, $e_{5,6}$ and $e_{5,7}$ contribute
$\FR{-2}5$, $\FR{-1}5$, and $\FR{-2}{35}$, respectively, to the negative charge
remaining at $5$-vertices.  Since the total charge is $-12$ and there is no
negative charge elsewhere, in units of $\FR{-1}5$ we have
$2e_{5,5}+e_{5,6}+\FR27e_{5,7}\ge60$.  (Another proof sends all charge from
vertices to edges so that each vertex ends with charge $0$ and only light
edges have negative charge.)

Equality requires that no positive charge is left anywhere, since that would
require more negative charge left at $5$-vertices.  Hence a sharpness example
must be a triangulation with maximum degree at most $7$.  
We use different sharpness examples for different coefficients.

The $5$-regular icosahedron has $e_{5,5}=30$ and no other edges; thus the
coefficient on $e_{5,5}$ cannot be reduced.  The graph obtained from the
$3$-regular dodecahedron by inserting in each face a $5$-vertex adjacent to its
corners has $e_{5,6}=60$, with all other edges joining $6$-vertices; thus the
coefficient on $e_{5,6}$ cannot be reduced.


The graph in Figure~\ref{fige57} (three edges wrap around from left to right),
with $2e_{5,5}=e_{5,6}=28$ and $e_{5,7}=14$, shows that the coefficient on
$e_{5,7}$ cannot be reduced below $\FR27$.
\end{proof}

\begin{figure}[h]
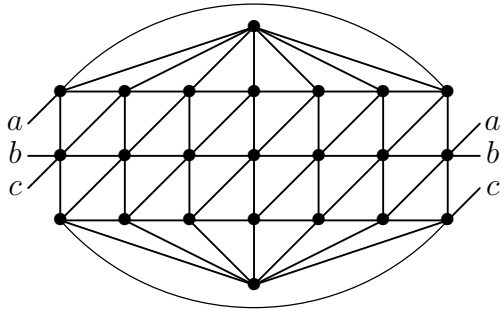

\gpic{
\expandafter\ifx\csname graph\endcsname\relax \csname newbox\endcsname\graph\fi
\expandafter\ifx\csname graphtemp\endcsname\relax \csname newdimen\endcsname\graphtemp\fi
\setbox\graph=\vtop{\vskip 0pt\hbox{%
    \graphtemp=.5ex\advance\graphtemp by 1.471in
    \rlap{\kern 1.250in\lower\graphtemp\hbox to 0pt{\hss $\bu$\hss}}%
    \graphtemp=.5ex\advance\graphtemp by 0.120in
    \rlap{\kern 1.250in\lower\graphtemp\hbox to 0pt{\hss $\bu$\hss}}%
    \graphtemp=.5ex\advance\graphtemp by 1.133in
    \rlap{\kern 0.236in\lower\graphtemp\hbox to 0pt{\hss $\bu$\hss}}%
    \graphtemp=.5ex\advance\graphtemp by 1.133in
    \rlap{\kern 0.574in\lower\graphtemp\hbox to 0pt{\hss $\bu$\hss}}%
    \graphtemp=.5ex\advance\graphtemp by 1.133in
    \rlap{\kern 0.912in\lower\graphtemp\hbox to 0pt{\hss $\bu$\hss}}%
    \graphtemp=.5ex\advance\graphtemp by 1.133in
    \rlap{\kern 1.250in\lower\graphtemp\hbox to 0pt{\hss $\bu$\hss}}%
    \graphtemp=.5ex\advance\graphtemp by 1.133in
    \rlap{\kern 1.588in\lower\graphtemp\hbox to 0pt{\hss $\bu$\hss}}%
    \graphtemp=.5ex\advance\graphtemp by 1.133in
    \rlap{\kern 1.926in\lower\graphtemp\hbox to 0pt{\hss $\bu$\hss}}%
    \graphtemp=.5ex\advance\graphtemp by 1.133in
    \rlap{\kern 2.264in\lower\graphtemp\hbox to 0pt{\hss $\bu$\hss}}%
    \special{pn 11}%
    \special{pa 1250 1471}%
    \special{pa 236 1133}%
    \special{fp}%
    \special{pa 1250 1471}%
    \special{pa 574 1133}%
    \special{fp}%
    \special{pa 1250 1471}%
    \special{pa 912 1133}%
    \special{fp}%
    \special{pa 1250 1471}%
    \special{pa 1250 1133}%
    \special{fp}%
    \special{pa 1250 1471}%
    \special{pa 1588 1133}%
    \special{fp}%
    \special{pa 1250 1471}%
    \special{pa 1926 1133}%
    \special{fp}%
    \special{pa 1250 1471}%
    \special{pa 2264 1133}%
    \special{fp}%
    \graphtemp=.5ex\advance\graphtemp by 0.795in
    \rlap{\kern 0.236in\lower\graphtemp\hbox to 0pt{\hss $\bu$\hss}}%
    \graphtemp=.5ex\advance\graphtemp by 0.795in
    \rlap{\kern 0.574in\lower\graphtemp\hbox to 0pt{\hss $\bu$\hss}}%
    \graphtemp=.5ex\advance\graphtemp by 0.795in
    \rlap{\kern 0.912in\lower\graphtemp\hbox to 0pt{\hss $\bu$\hss}}%
    \graphtemp=.5ex\advance\graphtemp by 0.795in
    \rlap{\kern 1.250in\lower\graphtemp\hbox to 0pt{\hss $\bu$\hss}}%
    \graphtemp=.5ex\advance\graphtemp by 0.795in
    \rlap{\kern 1.588in\lower\graphtemp\hbox to 0pt{\hss $\bu$\hss}}%
    \graphtemp=.5ex\advance\graphtemp by 0.795in
    \rlap{\kern 1.926in\lower\graphtemp\hbox to 0pt{\hss $\bu$\hss}}%
    \graphtemp=.5ex\advance\graphtemp by 0.795in
    \rlap{\kern 2.264in\lower\graphtemp\hbox to 0pt{\hss $\bu$\hss}}%
    \graphtemp=.5ex\advance\graphtemp by 0.458in
    \rlap{\kern 0.236in\lower\graphtemp\hbox to 0pt{\hss $\bu$\hss}}%
    \graphtemp=.5ex\advance\graphtemp by 0.458in
    \rlap{\kern 0.574in\lower\graphtemp\hbox to 0pt{\hss $\bu$\hss}}%
    \graphtemp=.5ex\advance\graphtemp by 0.458in
    \rlap{\kern 0.912in\lower\graphtemp\hbox to 0pt{\hss $\bu$\hss}}%
    \graphtemp=.5ex\advance\graphtemp by 0.458in
    \rlap{\kern 1.250in\lower\graphtemp\hbox to 0pt{\hss $\bu$\hss}}%
    \graphtemp=.5ex\advance\graphtemp by 0.458in
    \rlap{\kern 1.588in\lower\graphtemp\hbox to 0pt{\hss $\bu$\hss}}%
    \graphtemp=.5ex\advance\graphtemp by 0.458in
    \rlap{\kern 1.926in\lower\graphtemp\hbox to 0pt{\hss $\bu$\hss}}%
    \graphtemp=.5ex\advance\graphtemp by 0.458in
    \rlap{\kern 2.264in\lower\graphtemp\hbox to 0pt{\hss $\bu$\hss}}%
    \special{pa 1250 120}%
    \special{pa 236 458}%
    \special{fp}%
    \special{pa 1250 120}%
    \special{pa 574 458}%
    \special{fp}%
    \special{pa 1250 120}%
    \special{pa 912 458}%
    \special{fp}%
    \special{pa 1250 120}%
    \special{pa 1250 458}%
    \special{fp}%
    \special{pa 1250 120}%
    \special{pa 1588 458}%
    \special{fp}%
    \special{pa 1250 120}%
    \special{pa 1926 458}%
    \special{fp}%
    \special{pa 1250 120}%
    \special{pa 2264 458}%
    \special{fp}%
    \special{pa 236 1133}%
    \special{pa 2264 1133}%
    \special{fp}%
    \special{pa 68 795}%
    \special{pa 2432 795}%
    \special{fp}%
    \special{pa 236 458}%
    \special{pa 2264 458}%
    \special{fp}%
    \special{pa 236 458}%
    \special{pa 236 1133}%
    \special{fp}%
    \special{pa 236 1133}%
    \special{pa 912 458}%
    \special{fp}%
    \special{pa 912 458}%
    \special{pa 912 1133}%
    \special{fp}%
    \special{pa 912 1133}%
    \special{pa 1588 458}%
    \special{fp}%
    \special{pa 1588 458}%
    \special{pa 1588 1133}%
    \special{fp}%
    \special{pa 1588 1133}%
    \special{pa 2264 458}%
    \special{fp}%
    \special{pa 2264 458}%
    \special{pa 2264 1133}%
    \special{fp}%
    \special{pa 68 964}%
    \special{pa 574 458}%
    \special{fp}%
    \special{pa 574 458}%
    \special{pa 574 1133}%
    \special{fp}%
    \special{pa 574 1133}%
    \special{pa 1250 458}%
    \special{fp}%
    \special{pa 1250 458}%
    \special{pa 1250 1133}%
    \special{fp}%
    \special{pa 1250 1133}%
    \special{pa 1926 458}%
    \special{fp}%
    \special{pa 1926 458}%
    \special{pa 1926 1133}%
    \special{fp}%
    \special{pa 1926 1133}%
    \special{pa 2432 626}%
    \special{fp}%
    \special{pn 8}%
    \special{ar 1250 239 1351 1351 0.722734 2.418858}%
    \special{ar 1250 1351 1351 1351 -2.418858 -0.722734}%
    \special{pn 11}%
    \special{pa 236 458}%
    \special{pa 68 626}%
    \special{fp}%
    \special{pa 2264 1133}%
    \special{pa 2432 964}%
    \special{fp}%
    \graphtemp=.5ex\advance\graphtemp by 0.626in
    \rlap{\kern 0.000in\lower\graphtemp\hbox to 0pt{\hss $a$\hss}}%
    \graphtemp=.5ex\advance\graphtemp by 0.795in
    \rlap{\kern 0.000in\lower\graphtemp\hbox to 0pt{\hss $b$\hss}}%
    \graphtemp=.5ex\advance\graphtemp by 0.964in
    \rlap{\kern 0.000in\lower\graphtemp\hbox to 0pt{\hss $c$\hss}}%
    \graphtemp=.5ex\advance\graphtemp by 0.626in
    \rlap{\kern 2.500in\lower\graphtemp\hbox to 0pt{\hss $a$\hss}}%
    \graphtemp=.5ex\advance\graphtemp by 0.795in
    \rlap{\kern 2.500in\lower\graphtemp\hbox to 0pt{\hss $b$\hss}}%
    \graphtemp=.5ex\advance\graphtemp by 0.964in
    \rlap{\kern 2.500in\lower\graphtemp\hbox to 0pt{\hss $c$\hss}}%
    \hbox{\vrule depth1.591in width0pt height 0pt}%
    \kern 2.500in
  }%
}%
}
\caption{Sharpness example for Theorem~\ref{manylight}\label{fige57}}
\end{figure}

\bigskip
Let us now consider restricted families of planar graphs under girth
constraints.  We noted in the introduction that planar graphs with girth at
least $g$ satisfy $\mad(G)<\FR{2g}{g-2}$.  In some cases, planarity permits a
stronger result, meaning that obtaining the same conclusion using only a bound
on $\mad(G)$ requires $\mad(G)<b$ for some $b$ smaller than $\FR{2g}{g-2}$.

For example, consider a special case of Remark~\ref{2bound}: every graph $G$
with $\mad(G)<\FR83$ and $\delta(G)\ge 2$ has a $2$-vertex with a
$3^-$-neighbor.  Planar graphs with girth at least $8$ satisfy $\mad(G)<\FR83$.
The result in terms of $\mad(G)$ is sharp, since subdividing every edge of a
$4$-regular graph yields a graph $G$ with $\Mad(G)=\FR83$ having no such
$2$-vertex.  However, the conclusion holds for planar graphs with girth $7$,
which allow $\Mad(G)$ to be larger.  The argument illustrates typical
difficulties that may arise when discovering discharging arguments.  

\begin{lemma}\label{girth7}
Every planar graph $G$ with girth at least $7$ and $\delta(G)\ge2$ has a
$2$-vertex with a $3^-$-neighbor.
\end{lemma}
\begin{proof}
Assume that $G$ has no such configuration and use face charging.  With
initial charges $2d(v)-6$ and $\ell(f)-6$, when $G$ has girth at least $7$ the
only objects with negative initial charge are $2$-vertices.  Let each
$2$-vertex take $\FR12$ from each neighbor and each incident face.  To complete
the proof, we check that all vertices and faces end with nonnegative charge.

The discharging rule ensures that $2$-vertices end with charge $0$.
Since $3$-vertices have no $2$-neighbors, their charge remains $0$.  For
$j\ge4$, a $j$-vertex may lose $\FR12$ along each edge and ends with charge at
least $2j-6-\FR j2$, which is nonnegative.

A $j$-face has at most $\FL{\FR j2}$ incident $2$-vertices, since $2$-vertices
are not adjacent.  Hence a $j$-face has final charge at least
$j-6-\FR12\FL{\FR j2}$, which is nonnegative for $j\ge8$.
To help the $7$-faces, we add another discharging rule.  When adjacent
$4^+$-vertices form an edge $e$, direct the charge $\FR12$ that each could send
to a $2$-neighbor so that instead the two faces bounded by $e$ each receive
$\FR12$.  Now when a $7$-face gives away $\FR32$ to three $2$-vertices, it
recovers $\FR12$ from the two adjacent $4^+$-vertices on its boundary and ends
with charge $0$.
\end{proof}

\begin{figure}[h]
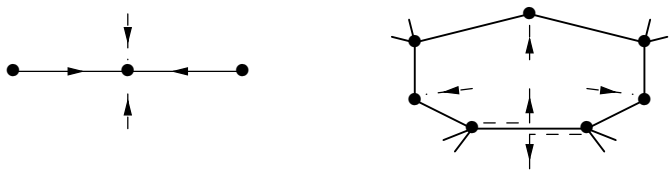

\gpic{
\expandafter\ifx\csname graph\endcsname\relax \csname newbox\endcsname\graph\fi
\expandafter\ifx\csname graphtemp\endcsname\relax \csname newdimen\endcsname\graphtemp\fi
\setbox\graph=\vtop{\vskip 0pt\hbox{%
    \graphtemp=.5ex\advance\graphtemp by 0.376in
    \rlap{\kern 0.676in\lower\graphtemp\hbox to 0pt{\hss $\bu$\hss}}%
    \graphtemp=.5ex\advance\graphtemp by 0.526in
    \rlap{\kern 2.178in\lower\graphtemp\hbox to 0pt{\hss $\bu$\hss}}%
    \graphtemp=.5ex\advance\graphtemp by 0.676in
    \rlap{\kern 2.479in\lower\graphtemp\hbox to 0pt{\hss $\bu$\hss}}%
    \graphtemp=.5ex\advance\graphtemp by 0.676in
    \rlap{\kern 3.079in\lower\graphtemp\hbox to 0pt{\hss $\bu$\hss}}%
    \graphtemp=.5ex\advance\graphtemp by 0.526in
    \rlap{\kern 3.380in\lower\graphtemp\hbox to 0pt{\hss $\bu$\hss}}%
    \graphtemp=.5ex\advance\graphtemp by 0.225in
    \rlap{\kern 3.380in\lower\graphtemp\hbox to 0pt{\hss $\bu$\hss}}%
    \graphtemp=.5ex\advance\graphtemp by 0.075in
    \rlap{\kern 2.779in\lower\graphtemp\hbox to 0pt{\hss $\bu$\hss}}%
    \graphtemp=.5ex\advance\graphtemp by 0.225in
    \rlap{\kern 2.178in\lower\graphtemp\hbox to 0pt{\hss $\bu$\hss}}%
    \special{pn 11}%
    \special{pa 2178 526}%
    \special{pa 2479 676}%
    \special{fp}%
    \special{pa 2479 676}%
    \special{pa 3079 676}%
    \special{fp}%
    \special{pa 3079 676}%
    \special{pa 3380 526}%
    \special{fp}%
    \special{pa 3380 526}%
    \special{pa 3380 225}%
    \special{fp}%
    \special{pa 3380 225}%
    \special{pa 2779 75}%
    \special{fp}%
    \special{pa 2779 75}%
    \special{pa 2178 225}%
    \special{fp}%
    \special{pa 2178 225}%
    \special{pa 2178 526}%
    \special{fp}%
    \special{pn 8}%
    \special{pa 75 376}%
    \special{pa 436 376}%
    \special{fp}%
    \special{sh 1.000}%
    \special{pa 364 358}%
    \special{pa 436 376}%
    \special{pa 364 394}%
    \special{pa 364 358}%
    \special{fp}%
    \special{pa 436 376}%
    \special{pa 676 376}%
    \special{fp}%
    \special{pa 1277 376}%
    \special{pa 916 376}%
    \special{fp}%
    \special{sh 1.000}%
    \special{pa 988 394}%
    \special{pa 916 376}%
    \special{pa 988 358}%
    \special{pa 988 394}%
    \special{fp}%
    \special{pa 916 376}%
    \special{pa 676 376}%
    \special{fp}%
    \special{pa 676 75}%
    \special{pa 676 315}%
    \special{da 0.060}%
    \special{pa 676 171}%
    \special{pa 676 219}%
    \special{fp}%
    \special{sh 1.000}%
    \special{pa 694 147}%
    \special{pa 676 219}%
    \special{pa 658 147}%
    \special{pa 694 147}%
    \special{fp}%
    \special{pa 676 676}%
    \special{pa 676 436}%
    \special{da 0.060}%
    \special{pa 676 580}%
    \special{pa 676 532}%
    \special{fp}%
    \special{sh 1.000}%
    \special{pa 658 604}%
    \special{pa 676 532}%
    \special{pa 694 604}%
    \special{pa 658 604}%
    \special{fp}%
    \graphtemp=.5ex\advance\graphtemp by 0.376in
    \rlap{\kern 0.075in\lower\graphtemp\hbox to 0pt{\hss $\bu$\hss}}%
    \graphtemp=.5ex\advance\graphtemp by 0.376in
    \rlap{\kern 1.277in\lower\graphtemp\hbox to 0pt{\hss $\bu$\hss}}%
    \special{pn 11}%
    \special{pa 3410 105}%
    \special{pa 3380 225}%
    \special{fp}%
    \special{pa 3380 225}%
    \special{pa 3500 195}%
    \special{fp}%
    \special{pa 2148 105}%
    \special{pa 2178 225}%
    \special{fp}%
    \special{pa 2178 225}%
    \special{pa 2058 195}%
    \special{fp}%
    \special{pa 2328 736}%
    \special{pa 2479 676}%
    \special{fp}%
    \special{pa 2479 676}%
    \special{pa 2388 796}%
    \special{fp}%
    \special{pa 3230 736}%
    \special{pa 3079 676}%
    \special{fp}%
    \special{pa 3079 676}%
    \special{pa 3170 796}%
    \special{fp}%
    \special{pn 8}%
    \special{pa 2539 646}%
    \special{pa 2779 646}%
    \special{da 0.060}%
    \special{pa 3049 706}%
    \special{pa 2779 706}%
    \special{da 0.060}%
    \special{pa 2779 646}%
    \special{pa 2779 436}%
    \special{da 0.060}%
    \special{pa 2779 562}%
    \special{pa 2779 520}%
    \special{fp}%
    \special{sh 1.000}%
    \special{pa 2761 592}%
    \special{pa 2779 520}%
    \special{pa 2797 592}%
    \special{pa 2761 592}%
    \special{fp}%
    \special{pa 2779 706}%
    \special{pa 2779 916}%
    \special{da 0.060}%
    \special{pa 2779 790}%
    \special{pa 2779 832}%
    \special{fp}%
    \special{sh 1.000}%
    \special{pa 2797 760}%
    \special{pa 2779 832}%
    \special{pa 2761 760}%
    \special{pa 2797 760}%
    \special{fp}%
    \special{pa 2479 466}%
    \special{pa 2238 496}%
    \special{da 0.060}%
    \special{pa 2382 478}%
    \special{pa 2334 484}%
    \special{fp}%
    \special{sh 1.000}%
    \special{pa 2408 493}%
    \special{pa 2334 484}%
    \special{pa 2404 457}%
    \special{pa 2408 493}%
    \special{fp}%
    \special{pa 3079 466}%
    \special{pa 3320 496}%
    \special{da 0.060}%
    \special{pa 3176 478}%
    \special{pa 3224 484}%
    \special{fp}%
    \special{sh 1.000}%
    \special{pa 3154 457}%
    \special{pa 3224 484}%
    \special{pa 3150 493}%
    \special{pa 3154 457}%
    \special{fp}%
    \special{pa 2779 315}%
    \special{pa 2779 135}%
    \special{da 0.060}%
    \special{pa 2779 243}%
    \special{pa 2779 207}%
    \special{fp}%
    \special{sh 1.000}%
    \special{pa 2761 279}%
    \special{pa 2779 207}%
    \special{pa 2797 279}%
    \special{pa 2761 279}%
    \special{fp}%
    \hbox{\vrule depth0.916in width0pt height 0pt}%
    \kern 3.500in
  }%
}%
}
\caption{Discharging for Lemma~\ref{girth7}\label{figredir}}
\end{figure}

This proof illustrates both ``redirection'' of transmitted charge and the
phenomenon of designing discharging rules initially to make deficient vertices
or faces happy but discovering later that additional rules are then needed to
repair others that lost too much.  It turns out that balanced charging, where
$2$-vertices and $3$-vertices both need charge, yields a simpler discharging
proof of this result; see Exercise~\ref{edgir7}.

Lemma~\ref{girth7} yields a stronger result for planar graphs with girth $7$
than is possible for the corresponding bound on $\mad(G)$.  A graph $G$ is
\emph{dynamically $k$-choosable} if for every $k$-uniform list assignment $L$,
there is a dynamic $L$-coloring of $G$, meaning a proper $L$-coloring with the
additional property that the neighbors of a vertex cannot all have the same
color if the vertex has degree at least $2$.  By showing that the configuration
in Lemma~\ref{girth7} (a $2$-vertex with a $3^-$-neighbor is reducible
(Exercise~\ref{4dynam}), Kim and Park~\cite{KP} showed that every planar graph
with girth at least $7$ is dynamically $4$-choosable.  This result is sharp,
since it is well known that there are planar graphs that are not $4$-choosable
(Voigt~\cite{Voigt}), and subdividing every edge of such a planar graph yields
a planar graph with girth $6$ that is not dynamically $4$-choosable.

Lemma~\ref{girth7} also has an application to acyclic coloring.
Configurations consisting of a $1^-$-vertex or a $2$-vertex with a
$3^-$-neighbor are reducible for acyclic $4$-choosability, so this holds
for planar graphs with girth at least $7$.  Gr\"unbaum~\cite{Gru} conjectured
that all planar graphs are acyclically $5$-colorable.  This was proved by
Borodin~\cite{Bacy} after successive improvements to Gr\"unbaum's initial
upper bound of $9$.  The bound of $5$ is sharp, even among bipartite planar
graphs~\cite{KMe} (Exercise~\ref{noac4}).  Borodin's proof used discharging
with some $450$ reducible configurations but no computers, an enormous effort.

Borodin et al.~\cite{BFKRS} conjectured the stronger statement that all planar
graphs are in fact acyclically $5$-choosable.  Toward the conjecture, Borodin
and Ivanova~\cite{BI4} proved that planar graphs without $4$-cycles are
acyclically $5$-choosable.  In~\cite{BI3}, only special $4$-cycles are
forbidden.  Montassier, Raspaud, and Wang~\cite{MRW06} conjectured that planar
graphs without $4$-cycles are acyclically $4$-choosable and proved this in some
cases; it holds when both $4$-cycles and $5$-cycles are
forbidden~\cite{BI5,CR13}.

Larger girth (or smaller $\Mad(G)$) makes acyclic coloring easier.  We have
already observed that planar graphs with girth at least $7$ are acyclically
$4$-choosable; Montassier~\cite{Mon} proved that girth at least $5$ suffices,
while Borodin et al.~\cite{BCIR} proved that girth at least $7$ yields acyclic
$3$-choosability.  We saw in Theorem~\ref{acyc6} that $\mad(G)<3$ yields
acyclic $6$-choosability.  The condition holds for planar graphs with girth at
least $6$, but using planarity allows us to relax the girth restriction.  For
planar graphs with girth at least $5$, we prove a structure theorem that yields
acyclic $6$-choosability and has other applications.  Note that for planar
graphs with girth $5$ and minimum degree $3$, it guarantees an edge of weight
$6$.

\begin{lemma}[Cranston--Yu~\cite{CY}]\label{gir5dis}
If $G$ is a planar graph with girth at least $5$ and $\delta(G)\ge2$, then $G$
has a $2$-vertex with a $5^-$-neighbor or a $5$-face whose incident vertices
are four $3$-vertices and a $5^-$-vertex.
\end{lemma}
\begin{proof}
(sketch)
Let $G$ be such a graph containing none of the specified configurations.
Assign charges by balanced charging; discharging will make all vertices and
faces happy when the specified configurations do not occur.

\medskip\noindent
(R1) Each $3^-$-vertex $v$ takes $\FR{4-d(v)}{d(v)}$ from each incident face.

\noindent
(R2) Each $6^+$-vertex $v$ gives $\FR{d(v)-4}{d(v)}$ to each incident face.

\medskip
The rules immediately make each vertex happy ($5$-vertices end positive),
and it remains only to check that each face ends happy.  The configurations in
Figure~\ref{figgir5} show faces that end with charge $0$; Exercise~\ref{gir5}
requests the verification that other faces end happy.
\end{proof}

\begin{figure}[hbt]
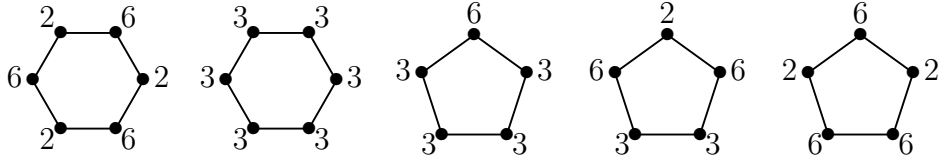

\gpic{
\expandafter\ifx\csname graph\endcsname\relax \csname newbox\endcsname\graph\fi
\expandafter\ifx\csname graphtemp\endcsname\relax \csname newdimen\endcsname\graphtemp\fi
\setbox\graph=\vtop{\vskip 0pt\hbox{%
    \graphtemp=.5ex\advance\graphtemp by 0.096in
    \rlap{\kern 0.530in\lower\graphtemp\hbox to 0pt{\hss $\bu$\hss}}%
    \graphtemp=.5ex\advance\graphtemp by 0.346in
    \rlap{\kern 0.674in\lower\graphtemp\hbox to 0pt{\hss $\bu$\hss}}%
    \graphtemp=.5ex\advance\graphtemp by 0.597in
    \rlap{\kern 0.530in\lower\graphtemp\hbox to 0pt{\hss $\bu$\hss}}%
    \graphtemp=.5ex\advance\graphtemp by 0.597in
    \rlap{\kern 0.241in\lower\graphtemp\hbox to 0pt{\hss $\bu$\hss}}%
    \graphtemp=.5ex\advance\graphtemp by 0.346in
    \rlap{\kern 0.096in\lower\graphtemp\hbox to 0pt{\hss $\bu$\hss}}%
    \graphtemp=.5ex\advance\graphtemp by 0.096in
    \rlap{\kern 0.241in\lower\graphtemp\hbox to 0pt{\hss $\bu$\hss}}%
    \special{pn 11}%
    \special{pa 530 96}%
    \special{pa 674 346}%
    \special{fp}%
    \special{pa 674 346}%
    \special{pa 530 597}%
    \special{fp}%
    \special{pa 530 597}%
    \special{pa 241 597}%
    \special{fp}%
    \special{pa 241 597}%
    \special{pa 96 346}%
    \special{fp}%
    \special{pa 96 346}%
    \special{pa 241 96}%
    \special{fp}%
    \special{pa 241 96}%
    \special{pa 530 96}%
    \special{fp}%
    \graphtemp=.5ex\advance\graphtemp by 0.028in
    \rlap{\kern 0.598in\lower\graphtemp\hbox to 0pt{\hss 6\hss}}%
    \graphtemp=.5ex\advance\graphtemp by 0.346in
    \rlap{\kern 0.770in\lower\graphtemp\hbox to 0pt{\hss 2\hss}}%
    \graphtemp=.5ex\advance\graphtemp by 0.665in
    \rlap{\kern 0.598in\lower\graphtemp\hbox to 0pt{\hss 6\hss}}%
    \graphtemp=.5ex\advance\graphtemp by 0.665in
    \rlap{\kern 0.173in\lower\graphtemp\hbox to 0pt{\hss 2\hss}}%
    \graphtemp=.5ex\advance\graphtemp by 0.346in
    \rlap{\kern 0.000in\lower\graphtemp\hbox to 0pt{\hss 6\hss}}%
    \graphtemp=.5ex\advance\graphtemp by 0.028in
    \rlap{\kern 0.173in\lower\graphtemp\hbox to 0pt{\hss 2\hss}}%
    \graphtemp=.5ex\advance\graphtemp by 0.096in
    \rlap{\kern 1.541in\lower\graphtemp\hbox to 0pt{\hss $\bu$\hss}}%
    \graphtemp=.5ex\advance\graphtemp by 0.346in
    \rlap{\kern 1.685in\lower\graphtemp\hbox to 0pt{\hss $\bu$\hss}}%
    \graphtemp=.5ex\advance\graphtemp by 0.597in
    \rlap{\kern 1.541in\lower\graphtemp\hbox to 0pt{\hss $\bu$\hss}}%
    \graphtemp=.5ex\advance\graphtemp by 0.597in
    \rlap{\kern 1.252in\lower\graphtemp\hbox to 0pt{\hss $\bu$\hss}}%
    \graphtemp=.5ex\advance\graphtemp by 0.346in
    \rlap{\kern 1.107in\lower\graphtemp\hbox to 0pt{\hss $\bu$\hss}}%
    \graphtemp=.5ex\advance\graphtemp by 0.096in
    \rlap{\kern 1.252in\lower\graphtemp\hbox to 0pt{\hss $\bu$\hss}}%
    \special{pa 1541 96}%
    \special{pa 1685 346}%
    \special{fp}%
    \special{pa 1685 346}%
    \special{pa 1541 597}%
    \special{fp}%
    \special{pa 1541 597}%
    \special{pa 1252 597}%
    \special{fp}%
    \special{pa 1252 597}%
    \special{pa 1107 346}%
    \special{fp}%
    \special{pa 1107 346}%
    \special{pa 1252 96}%
    \special{fp}%
    \special{pa 1252 96}%
    \special{pa 1541 96}%
    \special{fp}%
    \graphtemp=.5ex\advance\graphtemp by 0.028in
    \rlap{\kern 1.609in\lower\graphtemp\hbox to 0pt{\hss 3\hss}}%
    \graphtemp=.5ex\advance\graphtemp by 0.346in
    \rlap{\kern 1.781in\lower\graphtemp\hbox to 0pt{\hss 3\hss}}%
    \graphtemp=.5ex\advance\graphtemp by 0.665in
    \rlap{\kern 1.609in\lower\graphtemp\hbox to 0pt{\hss 3\hss}}%
    \graphtemp=.5ex\advance\graphtemp by 0.665in
    \rlap{\kern 1.184in\lower\graphtemp\hbox to 0pt{\hss 3\hss}}%
    \graphtemp=.5ex\advance\graphtemp by 0.346in
    \rlap{\kern 1.011in\lower\graphtemp\hbox to 0pt{\hss 3\hss}}%
    \graphtemp=.5ex\advance\graphtemp by 0.028in
    \rlap{\kern 1.184in\lower\graphtemp\hbox to 0pt{\hss 3\hss}}%
    \graphtemp=.5ex\advance\graphtemp by 0.106in
    \rlap{\kern 2.407in\lower\graphtemp\hbox to 0pt{\hss $\bu$\hss}}%
    \graphtemp=.5ex\advance\graphtemp by 0.305in
    \rlap{\kern 2.682in\lower\graphtemp\hbox to 0pt{\hss $\bu$\hss}}%
    \graphtemp=.5ex\advance\graphtemp by 0.628in
    \rlap{\kern 2.577in\lower\graphtemp\hbox to 0pt{\hss $\bu$\hss}}%
    \graphtemp=.5ex\advance\graphtemp by 0.628in
    \rlap{\kern 2.237in\lower\graphtemp\hbox to 0pt{\hss $\bu$\hss}}%
    \graphtemp=.5ex\advance\graphtemp by 0.305in
    \rlap{\kern 2.132in\lower\graphtemp\hbox to 0pt{\hss $\bu$\hss}}%
    \special{pa 2407 106}%
    \special{pa 2682 305}%
    \special{fp}%
    \special{pa 2682 305}%
    \special{pa 2577 628}%
    \special{fp}%
    \special{pa 2577 628}%
    \special{pa 2237 628}%
    \special{fp}%
    \special{pa 2237 628}%
    \special{pa 2132 305}%
    \special{fp}%
    \special{pa 2132 305}%
    \special{pa 2407 106}%
    \special{fp}%
    \graphtemp=.5ex\advance\graphtemp by 0.009in
    \rlap{\kern 2.407in\lower\graphtemp\hbox to 0pt{\hss 6\hss}}%
    \graphtemp=.5ex\advance\graphtemp by 0.305in
    \rlap{\kern 2.778in\lower\graphtemp\hbox to 0pt{\hss 3\hss}}%
    \graphtemp=.5ex\advance\graphtemp by 0.696in
    \rlap{\kern 2.645in\lower\graphtemp\hbox to 0pt{\hss 3\hss}}%
    \graphtemp=.5ex\advance\graphtemp by 0.696in
    \rlap{\kern 2.169in\lower\graphtemp\hbox to 0pt{\hss 3\hss}}%
    \graphtemp=.5ex\advance\graphtemp by 0.305in
    \rlap{\kern 2.036in\lower\graphtemp\hbox to 0pt{\hss 3\hss}}%
    \graphtemp=.5ex\advance\graphtemp by 0.106in
    \rlap{\kern 3.418in\lower\graphtemp\hbox to 0pt{\hss $\bu$\hss}}%
    \graphtemp=.5ex\advance\graphtemp by 0.305in
    \rlap{\kern 3.693in\lower\graphtemp\hbox to 0pt{\hss $\bu$\hss}}%
    \graphtemp=.5ex\advance\graphtemp by 0.628in
    \rlap{\kern 3.588in\lower\graphtemp\hbox to 0pt{\hss $\bu$\hss}}%
    \graphtemp=.5ex\advance\graphtemp by 0.628in
    \rlap{\kern 3.248in\lower\graphtemp\hbox to 0pt{\hss $\bu$\hss}}%
    \graphtemp=.5ex\advance\graphtemp by 0.305in
    \rlap{\kern 3.143in\lower\graphtemp\hbox to 0pt{\hss $\bu$\hss}}%
    \special{pa 3418 106}%
    \special{pa 3693 305}%
    \special{fp}%
    \special{pa 3693 305}%
    \special{pa 3588 628}%
    \special{fp}%
    \special{pa 3588 628}%
    \special{pa 3248 628}%
    \special{fp}%
    \special{pa 3248 628}%
    \special{pa 3143 305}%
    \special{fp}%
    \special{pa 3143 305}%
    \special{pa 3418 106}%
    \special{fp}%
    \graphtemp=.5ex\advance\graphtemp by 0.009in
    \rlap{\kern 3.418in\lower\graphtemp\hbox to 0pt{\hss 2\hss}}%
    \graphtemp=.5ex\advance\graphtemp by 0.305in
    \rlap{\kern 3.789in\lower\graphtemp\hbox to 0pt{\hss 6\hss}}%
    \graphtemp=.5ex\advance\graphtemp by 0.696in
    \rlap{\kern 3.656in\lower\graphtemp\hbox to 0pt{\hss 3\hss}}%
    \graphtemp=.5ex\advance\graphtemp by 0.696in
    \rlap{\kern 3.180in\lower\graphtemp\hbox to 0pt{\hss 3\hss}}%
    \graphtemp=.5ex\advance\graphtemp by 0.305in
    \rlap{\kern 3.047in\lower\graphtemp\hbox to 0pt{\hss 6\hss}}%
    \graphtemp=.5ex\advance\graphtemp by 0.106in
    \rlap{\kern 4.429in\lower\graphtemp\hbox to 0pt{\hss $\bu$\hss}}%
    \graphtemp=.5ex\advance\graphtemp by 0.305in
    \rlap{\kern 4.704in\lower\graphtemp\hbox to 0pt{\hss $\bu$\hss}}%
    \graphtemp=.5ex\advance\graphtemp by 0.628in
    \rlap{\kern 4.599in\lower\graphtemp\hbox to 0pt{\hss $\bu$\hss}}%
    \graphtemp=.5ex\advance\graphtemp by 0.628in
    \rlap{\kern 4.259in\lower\graphtemp\hbox to 0pt{\hss $\bu$\hss}}%
    \graphtemp=.5ex\advance\graphtemp by 0.305in
    \rlap{\kern 4.154in\lower\graphtemp\hbox to 0pt{\hss $\bu$\hss}}%
    \special{pa 4429 106}%
    \special{pa 4704 305}%
    \special{fp}%
    \special{pa 4704 305}%
    \special{pa 4599 628}%
    \special{fp}%
    \special{pa 4599 628}%
    \special{pa 4259 628}%
    \special{fp}%
    \special{pa 4259 628}%
    \special{pa 4154 305}%
    \special{fp}%
    \special{pa 4154 305}%
    \special{pa 4429 106}%
    \special{fp}%
    \graphtemp=.5ex\advance\graphtemp by 0.009in
    \rlap{\kern 4.429in\lower\graphtemp\hbox to 0pt{\hss 6\hss}}%
    \graphtemp=.5ex\advance\graphtemp by 0.305in
    \rlap{\kern 4.800in\lower\graphtemp\hbox to 0pt{\hss 2\hss}}%
    \graphtemp=.5ex\advance\graphtemp by 0.696in
    \rlap{\kern 4.667in\lower\graphtemp\hbox to 0pt{\hss 6\hss}}%
    \graphtemp=.5ex\advance\graphtemp by 0.696in
    \rlap{\kern 4.191in\lower\graphtemp\hbox to 0pt{\hss 6\hss}}%
    \graphtemp=.5ex\advance\graphtemp by 0.305in
    \rlap{\kern 4.058in\lower\graphtemp\hbox to 0pt{\hss 2\hss}}%
    \hbox{\vrule depth0.725in width0pt height 0pt}%
    \kern 4.800in
  }%
}%
}
\caption{Sharp configurations for Lemma~\ref{gir5dis}\label{figgir5}}
\end{figure}

\begin{theorem}\label{acyclic6}
If $G$ is a planar graph with girth at least $5$, then $G$ is 
acyclically $6$-choosable.
\end{theorem}
\begin{proof}
Since a $1^-$-vertex lies in no cycle, its color need only avoid that of its
(possible) neighbor.  Hence a $1^-$-vertex is reducible for $\chial(G)\le6$,
and we may assume $\delta(G)\ge2$.  It therefore suffices to show that the
configurations in Lemma~\ref{gir5dis} are reducible for $\chial(G)\le6$.
Let $L$ be a $6$-uniform list assignment for $G$.

First consider a $2$-vertex $v$ with a $5^-$-neighbor $u$.  Let $\phi$ be an
acyclic $L$-coloring of $G-v$.  If the colors on $N_G(v)$ are distinct, then
color $v$ with a color in $L(v)$ other than those.  If the colors on $N(v)$ are
equal, then color $v$ with a color not used on $N_G(v)\cup N_{G-v}(u)$.  Since
$|N_{G-v}(u)|\le4$, this forbids at most five colors, and a color remains
available in $L(v)$.  Now there are no $2$-colored cycles in $G-v$ and none
through $v$.

For the remaining configuration, let $v_1,v_2,v_3,v_4,w$ be the vertices on a
$5$-face, with each $v_i$ of degree $3$ and $d(w)\le5$.  Let $x_i$ be the
neighbor of $v_i$ outside the $5$-cycle (see Figure~\ref{figacy6}).

\begin{figure}[hbt]
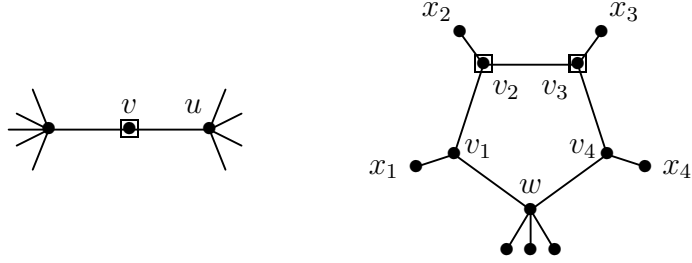

\gpic{
\expandafter\ifx\csname graph\endcsname\relax \csname newbox\endcsname\graph\fi
\expandafter\ifx\csname graphtemp\endcsname\relax \csname newdimen\endcsname\graphtemp\fi
\setbox\graph=\vtop{\vskip 0pt\hbox{%
    \graphtemp=.5ex\advance\graphtemp by 0.678in
    \rlap{\kern 0.210in\lower\graphtemp\hbox to 0pt{\hss $\bu$\hss}}%
    \graphtemp=.5ex\advance\graphtemp by 0.678in
    \rlap{\kern 0.631in\lower\graphtemp\hbox to 0pt{\hss $\bu$\hss}}%
    \graphtemp=.5ex\advance\graphtemp by 0.678in
    \rlap{\kern 1.051in\lower\graphtemp\hbox to 0pt{\hss $\bu$\hss}}%
    \special{pn 11}%
    \special{pa 210 678}%
    \special{pa 1051 678}%
    \special{fp}%
    \graphtemp=.5ex\advance\graphtemp by 0.552in
    \rlap{\kern 0.631in\lower\graphtemp\hbox to 0pt{\hss $v$\hss}}%
    \graphtemp=.5ex\advance\graphtemp by 0.552in
    \rlap{\kern 0.967in\lower\graphtemp\hbox to 0pt{\hss $u$\hss}}%
    \special{pa 1051 678}%
    \special{pa 1135 468}%
    \special{fp}%
    \special{pa 1051 678}%
    \special{pa 1219 594}%
    \special{fp}%
    \special{pa 1051 678}%
    \special{pa 1219 762}%
    \special{fp}%
    \special{pa 1051 678}%
    \special{pa 1135 888}%
    \special{fp}%
    \special{pa 210 678}%
    \special{pa 126 468}%
    \special{fp}%
    \special{pa 210 678}%
    \special{pa 42 594}%
    \special{fp}%
    \special{pa 210 678}%
    \special{pa 0 678}%
    \special{fp}%
    \special{pa 210 678}%
    \special{pa 42 762}%
    \special{fp}%
    \special{pa 210 678}%
    \special{pa 126 888}%
    \special{fp}%
    \graphtemp=.5ex\advance\graphtemp by 0.678in
    \rlap{\kern 0.631in\lower\graphtemp\hbox to 0pt{\hss $\marker$\hss}}%
    \graphtemp=.5ex\advance\graphtemp by 1.099in
    \rlap{\kern 2.732in\lower\graphtemp\hbox to 0pt{\hss $\bu$\hss}}%
    \graphtemp=.5ex\advance\graphtemp by 0.808in
    \rlap{\kern 2.332in\lower\graphtemp\hbox to 0pt{\hss $\bu$\hss}}%
    \graphtemp=.5ex\advance\graphtemp by 0.338in
    \rlap{\kern 2.485in\lower\graphtemp\hbox to 0pt{\hss $\bu$\hss}}%
    \graphtemp=.5ex\advance\graphtemp by 0.338in
    \rlap{\kern 2.979in\lower\graphtemp\hbox to 0pt{\hss $\bu$\hss}}%
    \graphtemp=.5ex\advance\graphtemp by 0.808in
    \rlap{\kern 3.132in\lower\graphtemp\hbox to 0pt{\hss $\bu$\hss}}%
    \special{pa 2732 1099}%
    \special{pa 2332 808}%
    \special{fp}%
    \special{pa 2332 808}%
    \special{pa 2485 338}%
    \special{fp}%
    \special{pa 2485 338}%
    \special{pa 2979 338}%
    \special{fp}%
    \special{pa 2979 338}%
    \special{pa 3132 808}%
    \special{fp}%
    \special{pa 3132 808}%
    \special{pa 2732 1099}%
    \special{fp}%
    \graphtemp=.5ex\advance\graphtemp by 0.972in
    \rlap{\kern 2.732in\lower\graphtemp\hbox to 0pt{\hss $w$\hss}}%
    \graphtemp=.5ex\advance\graphtemp by 0.766in
    \rlap{\kern 2.459in\lower\graphtemp\hbox to 0pt{\hss $v_1$\hss}}%
    \graphtemp=.5ex\advance\graphtemp by 0.457in
    \rlap{\kern 2.604in\lower\graphtemp\hbox to 0pt{\hss $v_2$\hss}}%
    \graphtemp=.5ex\advance\graphtemp by 0.457in
    \rlap{\kern 2.861in\lower\graphtemp\hbox to 0pt{\hss $v_3$\hss}}%
    \graphtemp=.5ex\advance\graphtemp by 0.766in
    \rlap{\kern 3.006in\lower\graphtemp\hbox to 0pt{\hss $v_4$\hss}}%
    \graphtemp=.5ex\advance\graphtemp by 0.338in
    \rlap{\kern 2.485in\lower\graphtemp\hbox to 0pt{\hss $\marker$\hss}}%
    \graphtemp=.5ex\advance\graphtemp by 0.338in
    \rlap{\kern 2.979in\lower\graphtemp\hbox to 0pt{\hss $\marker$\hss}}%
    \graphtemp=.5ex\advance\graphtemp by 1.309in
    \rlap{\kern 2.732in\lower\graphtemp\hbox to 0pt{\hss $\bu$\hss}}%
    \graphtemp=.5ex\advance\graphtemp by 0.873in
    \rlap{\kern 2.133in\lower\graphtemp\hbox to 0pt{\hss $\bu$\hss}}%
    \graphtemp=.5ex\advance\graphtemp by 0.168in
    \rlap{\kern 2.361in\lower\graphtemp\hbox to 0pt{\hss $\bu$\hss}}%
    \graphtemp=.5ex\advance\graphtemp by 0.168in
    \rlap{\kern 3.103in\lower\graphtemp\hbox to 0pt{\hss $\bu$\hss}}%
    \graphtemp=.5ex\advance\graphtemp by 0.873in
    \rlap{\kern 3.332in\lower\graphtemp\hbox to 0pt{\hss $\bu$\hss}}%
    \special{pa 2332 808}%
    \special{pa 2133 873}%
    \special{fp}%
    \special{pa 2485 338}%
    \special{pa 2361 168}%
    \special{fp}%
    \special{pa 2979 338}%
    \special{pa 3103 168}%
    \special{fp}%
    \special{pa 3132 808}%
    \special{pa 3332 873}%
    \special{fp}%
    \special{pa 2732 1099}%
    \special{pa 2732 1309}%
    \special{fp}%
    \special{pa 2732 1099}%
    \special{pa 2606 1309}%
    \special{fp}%
    \special{pa 2732 1099}%
    \special{pa 2858 1309}%
    \special{fp}%
    \graphtemp=.5ex\advance\graphtemp by 1.309in
    \rlap{\kern 2.606in\lower\graphtemp\hbox to 0pt{\hss $\bu$\hss}}%
    \graphtemp=.5ex\advance\graphtemp by 1.309in
    \rlap{\kern 2.858in\lower\graphtemp\hbox to 0pt{\hss $\bu$\hss}}%
    \graphtemp=.5ex\advance\graphtemp by 0.873in
    \rlap{\kern 1.964in\lower\graphtemp\hbox to 0pt{\hss $x_1$\hss}}%
    \graphtemp=.5ex\advance\graphtemp by 0.049in
    \rlap{\kern 2.243in\lower\graphtemp\hbox to 0pt{\hss $x_2$\hss}}%
    \graphtemp=.5ex\advance\graphtemp by 0.049in
    \rlap{\kern 3.222in\lower\graphtemp\hbox to 0pt{\hss $x_3$\hss}}%
    \graphtemp=.5ex\advance\graphtemp by 0.873in
    \rlap{\kern 3.500in\lower\graphtemp\hbox to 0pt{\hss $x_4$\hss}}%
    \hbox{\vrule depth1.414in width0pt height 0pt}%
    \kern 3.500in
  }%
}%
}

\vspace{-1pc}

\caption{Reducible configurations for Theorem~\ref{acyclic6}\label{figacy6}}
\end{figure}

Let $\phi$ be an acyclic $L$-coloring of $G-\{v_2,v_3\}$.  We consider three
cases, depending on whether $\phi$ uses one color or two colors on $N_G(v_2)$
and on $N_G(v_3)$.
(a)  If $\phi(v_1) \ne \phi(x_2)$ and $\phi(v_4) \ne \phi(x_3)$, then choose
$\phi(v_2)$ and $\phi(v_3)$ distinct and outside
$\{\phi(v_1), \phi(x_2), \phi(x_3), \phi(v_4)\}$.
(b) If $\phi(v_1) = \phi(x_2)$ but $\phi(v_4) \ne \phi(x_3)$, then choose
$\phi(v_2) \notin \{\phi(w), \phi(x_1), \phi(x_2)\}$ and
$\phi(v_3)\!\notin\!\{\phi(v_4), \phi(x_3), \phi(v_2), \phi(v_1)\}$.
(c) If $\phi(v_1)\!=\!\phi(x_2)$ and $\phi(v_4) = \phi(x_3)$, then choose
$\phi(v_2) \notin \{\phi(w), \phi(v_1), \phi(x_1), \phi(v_4)\}$ and
$\phi(v_3) \notin \{\phi(v_2), \phi(v_4), \phi(w), \phi(x_4)\}$.  In each
case, the coloring is proper, and the new vertices lie in no $2$-colored cycle.
\end{proof}


\bigskip

We will not discuss the proof of the Four Color Theorem here.
It is well known that after a hundred years of failed attempts, Appel and Haken
(working with Koch) found ``an unavoidable set of reducible configurations''
using the discharging method.  The discharging rules and reducibility arguments
were far more complicated than anything we present here.  The initial proof
involved 1936 reducible configurations.  The unavoidable set was generated by
hand, but reducibility was checked by computer.  The publication comprised
nearly 140 pages in two papers~\cite{AH771,AHK} plus over 400 pages of
microfiche
that became a 741-page book~\cite{AH89}. 

Some people objected to the use of computers, but the proof is now generally
accepted.  Robertson, Sanders, Seymour, and Thomas~\cite{RSST} looked
for a simpler proof but eventually used the same approach.  Their unavoidable
set had only 633 configurations and 32 discharging rules, but they still needed
a computer.  With the increases in computing power and simpler arguments, their
proof ran in only 20 minutes instead of the original 1200 hours.

With the Four Color Theorem proved, attention has turned to making use of it
(a notable example is Robertson, Seymour, and Thomas~\cite{RSTh} using it
to prove the case $k=6$ of Hadwiger's Conjecture, for which they won the 1994
Fulkerson Prize) and to understanding which planar graphs are $3$-colorable.
Computationally, testing $3$-colorability of a planar graph is
NP-hard~\cite{Stock}, but many sufficient conditions are known.

The most natural condition is to increase the girth; already
Gr\"otzsch~\cite{Gro} proved that planar graphs with girth at least $4$
are $3$-colorable.  There have been a number of proofs of this
(\cite{DKT,KY,ThomG1,ThomG3}), all using discharging at some point.
Thomassen~\cite{ThomG3} showed that girth at least $5$ suffices for
$3$-choosability.

Steinberg~\cite{St} conjectured that every planar graph without $4$-cycles or
$5$-cycles is $3$-colorable.  Eventually, Cohen-Addad et al.~\cite{CHKLS} found
counterexamples.  Results on this family can be compared with the
family where $\mad(G)<4$; see Exercise~\ref{no45}.

During the 40 years between~\cite{St} and~\cite{CHKLS}, many papers used
discharging to prove $3$-colorability under various conditions excluding sets
of cycle lengths.  For example, Borodin et al.~\cite{BGRS} proved that
excluding cycles of lengths $4$ through $7$ suffices.  Earlier,
Borodin~\cite{B1996} and Sanders and Zhao~\cite{SZ0} proved that excluding
$4$-cycles and faces of lengths $5$ through $9$ is sufficient.  The traditional
proof (presented in the survey~\cite{Bsurv}) uses balanced charging, but face
charging yields a somewhat simpler proof.

\begin{lemma}[\cite{B1996}]\label{10facelem}
Every plane graph $G$ with $\delta(G)\ge3$ has two 3-faces with a common edge,
or a $j$-face with $4\le j\le 9$, or a $10$-face whose vertices all have degree
$3$.
\end{lemma}
\begin{proof}
Let $G$ be a plane graph with $\delta(G)\ge3$ having none of the listed
configurations.  Use face charging: assign charge $2d(v)-6$ to each vertex $v$
and charge $\ell(f)-6$ to each face $f$.  The total charge is $-12$.

Since no faces have lengths $4$ through $9$, the only objects with initial
negative charge are triangles; they begin with charge $-3$.  Each triangle
takes $1$ from each neighboring face.  To repair faces that may lose too much,
each face $f$ takes 1 from each incident $4^+$-vertex lying on at least one
triangle sharing an edge with $f$ (see Figure~\ref{fig10face}).

We have made $3$-faces happy, and $3$-vertices remain at charge $0$.  Other
vertices remain happy because $3$-faces do not share edges.  For $j\ge4$, a
$j$-vertex loses charge at most $\FL{\FR{2j}3}$ and ends with at least
$\CL{\FR {4j}3}-6$, which is nonnegative for $j\ge4$.

Now consider a $j$-face $f$ for $j\ge10$.  It loses $1$ for every path along
its boundary such that the neighboring faces are triangles and the endpoints
have degree $3$; $f$ gives $1$ to each of those faces but regains $1$ from each
intervening vertex.  If an endpoint of a maximal such path has degree at least
$4$, then there is no net loss.  Hence the net loss for $f$ is at most
$\FL{\FR j2}$, and the final charge is at least $\CL{\FR j2}-6$, which is
nonnegative when $j\ge 11$.

Hence negative charge can occur only at $10$-faces.  A $10$-face $f$ must
lose more than $4$ to become negative.  This requires five paths through which
$f$ loses $1$.  The paths must be single edges sharing no vertices, and all
the vertices incident to $f$ must have degree $3$.
\end{proof}

\begin{figure}[h]
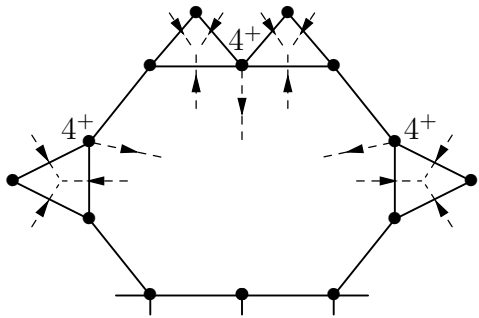

\gpic{
\expandafter\ifx\csname graph\endcsname\relax \csname newbox\endcsname\graph\fi
\expandafter\ifx\csname graphtemp\endcsname\relax \csname newdimen\endcsname\graphtemp\fi
\setbox\graph=\vtop{\vskip 0pt\hbox{%
    \graphtemp=.5ex\advance\graphtemp by 1.530in
    \rlap{\kern 1.250in\lower\graphtemp\hbox to 0pt{\hss $\bu$\hss}}%
    \graphtemp=.5ex\advance\graphtemp by 1.530in
    \rlap{\kern 1.730in\lower\graphtemp\hbox to 0pt{\hss $\bu$\hss}}%
    \graphtemp=.5ex\advance\graphtemp by 1.130in
    \rlap{\kern 2.050in\lower\graphtemp\hbox to 0pt{\hss $\bu$\hss}}%
    \graphtemp=.5ex\advance\graphtemp by 0.930in
    \rlap{\kern 2.450in\lower\graphtemp\hbox to 0pt{\hss $\bu$\hss}}%
    \graphtemp=.5ex\advance\graphtemp by 0.730in
    \rlap{\kern 2.050in\lower\graphtemp\hbox to 0pt{\hss $\bu$\hss}}%
    \graphtemp=.5ex\advance\graphtemp by 0.330in
    \rlap{\kern 1.730in\lower\graphtemp\hbox to 0pt{\hss $\bu$\hss}}%
    \graphtemp=.5ex\advance\graphtemp by 0.050in
    \rlap{\kern 1.490in\lower\graphtemp\hbox to 0pt{\hss $\bu$\hss}}%
    \graphtemp=.5ex\advance\graphtemp by 0.330in
    \rlap{\kern 1.250in\lower\graphtemp\hbox to 0pt{\hss $\bu$\hss}}%
    \special{pn 11}%
    \special{pa 1250 1530}%
    \special{pa 1730 1530}%
    \special{fp}%
    \special{pa 1730 1530}%
    \special{pa 2050 1130}%
    \special{fp}%
    \special{pa 2050 1130}%
    \special{pa 2450 930}%
    \special{fp}%
    \special{pa 2450 930}%
    \special{pa 2050 730}%
    \special{fp}%
    \special{pa 2050 730}%
    \special{pa 1730 330}%
    \special{fp}%
    \special{pa 1730 330}%
    \special{pa 1490 50}%
    \special{fp}%
    \special{pa 1490 50}%
    \special{pa 1250 330}%
    \special{fp}%
    \graphtemp=.5ex\advance\graphtemp by 1.530in
    \rlap{\kern 1.250in\lower\graphtemp\hbox to 0pt{\hss $\bu$\hss}}%
    \graphtemp=.5ex\advance\graphtemp by 1.530in
    \rlap{\kern 0.770in\lower\graphtemp\hbox to 0pt{\hss $\bu$\hss}}%
    \graphtemp=.5ex\advance\graphtemp by 1.130in
    \rlap{\kern 0.450in\lower\graphtemp\hbox to 0pt{\hss $\bu$\hss}}%
    \graphtemp=.5ex\advance\graphtemp by 0.930in
    \rlap{\kern 0.050in\lower\graphtemp\hbox to 0pt{\hss $\bu$\hss}}%
    \graphtemp=.5ex\advance\graphtemp by 0.730in
    \rlap{\kern 0.450in\lower\graphtemp\hbox to 0pt{\hss $\bu$\hss}}%
    \graphtemp=.5ex\advance\graphtemp by 0.330in
    \rlap{\kern 0.770in\lower\graphtemp\hbox to 0pt{\hss $\bu$\hss}}%
    \graphtemp=.5ex\advance\graphtemp by 0.050in
    \rlap{\kern 1.010in\lower\graphtemp\hbox to 0pt{\hss $\bu$\hss}}%
    \graphtemp=.5ex\advance\graphtemp by 0.330in
    \rlap{\kern 1.250in\lower\graphtemp\hbox to 0pt{\hss $\bu$\hss}}%
    \special{pa 1250 1530}%
    \special{pa 770 1530}%
    \special{fp}%
    \special{pa 770 1530}%
    \special{pa 450 1130}%
    \special{fp}%
    \special{pa 450 1130}%
    \special{pa 50 930}%
    \special{fp}%
    \special{pa 50 930}%
    \special{pa 450 730}%
    \special{fp}%
    \special{pa 450 730}%
    \special{pa 770 330}%
    \special{fp}%
    \special{pa 770 330}%
    \special{pa 1010 50}%
    \special{fp}%
    \special{pa 1010 50}%
    \special{pa 1250 330}%
    \special{fp}%
    \special{pa 450 1130}%
    \special{pa 450 730}%
    \special{fp}%
    \special{pa 2050 1130}%
    \special{pa 2050 730}%
    \special{fp}%
    \special{pa 770 330}%
    \special{pa 1730 330}%
    \special{fp}%
    \special{pa 1250 1530}%
    \special{pa 1250 1630}%
    \special{fp}%
    \special{pa 590 1530}%
    \special{pa 770 1530}%
    \special{fp}%
    \special{pa 770 1530}%
    \special{pa 770 1630}%
    \special{fp}%
    \special{pa 1910 1530}%
    \special{pa 1730 1530}%
    \special{fp}%
    \special{pa 1730 1530}%
    \special{pa 1730 1630}%
    \special{fp}%
    \graphtemp=.5ex\advance\graphtemp by 0.673in
    \rlap{\kern 0.393in\lower\graphtemp\hbox to 0pt{\hss $4^+$\hss}}%
    \graphtemp=.5ex\advance\graphtemp by 0.210in
    \rlap{\kern 1.270in\lower\graphtemp\hbox to 0pt{\hss $4^+$\hss}}%
    \graphtemp=.5ex\advance\graphtemp by 0.673in
    \rlap{\kern 2.187in\lower\graphtemp\hbox to 0pt{\hss $4^+$\hss}}%
    \special{pn 8}%
    \special{pa 650 930}%
    \special{pa 310 930}%
    \special{da 0.040}%
    \special{pa 514 930}%
    \special{pa 446 930}%
    \special{fp}%
    \special{sh 1.000}%
    \special{pa 526 950}%
    \special{pa 446 930}%
    \special{pa 526 910}%
    \special{pa 526 950}%
    \special{fp}%
    \special{pa 150 690}%
    \special{pa 290 910}%
    \special{da 0.040}%
    \special{pa 206 778}%
    \special{pa 234 822}%
    \special{fp}%
    \special{sh 1.000}%
    \special{pa 208 744}%
    \special{pa 234 822}%
    \special{pa 174 765}%
    \special{pa 208 744}%
    \special{fp}%
    \special{pa 150 1170}%
    \special{pa 290 950}%
    \special{da 0.040}%
    \special{pa 206 1082}%
    \special{pa 234 1038}%
    \special{fp}%
    \special{sh 1.000}%
    \special{pa 174 1095}%
    \special{pa 234 1038}%
    \special{pa 208 1116}%
    \special{pa 174 1095}%
    \special{fp}%
    \special{pa 1850 930}%
    \special{pa 2190 930}%
    \special{da 0.040}%
    \special{pa 1986 930}%
    \special{pa 2054 930}%
    \special{fp}%
    \special{sh 1.000}%
    \special{pa 1974 910}%
    \special{pa 2054 930}%
    \special{pa 1974 950}%
    \special{pa 1974 910}%
    \special{fp}%
    \special{pa 2350 690}%
    \special{pa 2210 910}%
    \special{da 0.040}%
    \special{pa 2294 778}%
    \special{pa 2266 822}%
    \special{fp}%
    \special{sh 1.000}%
    \special{pa 2326 765}%
    \special{pa 2266 822}%
    \special{pa 2292 744}%
    \special{pa 2326 765}%
    \special{fp}%
    \special{pa 2350 1170}%
    \special{pa 2210 950}%
    \special{da 0.040}%
    \special{pa 2294 1082}%
    \special{pa 2266 1038}%
    \special{fp}%
    \special{sh 1.000}%
    \special{pa 2292 1116}%
    \special{pa 2266 1038}%
    \special{pa 2326 1095}%
    \special{pa 2292 1116}%
    \special{fp}%
    \special{pa 1010 550}%
    \special{pa 1010 270}%
    \special{da 0.040}%
    \special{pa 1010 438}%
    \special{pa 1010 382}%
    \special{fp}%
    \special{sh 1.000}%
    \special{pa 990 462}%
    \special{pa 1010 382}%
    \special{pa 1030 462}%
    \special{pa 990 462}%
    \special{fp}%
    \special{pa 1150 50}%
    \special{pa 1030 230}%
    \special{da 0.040}%
    \special{pa 1102 122}%
    \special{pa 1078 158}%
    \special{fp}%
    \special{sh 1.000}%
    \special{pa 1139 103}%
    \special{pa 1078 158}%
    \special{pa 1106 80}%
    \special{pa 1139 103}%
    \special{fp}%
    \special{pa 870 50}%
    \special{pa 990 230}%
    \special{da 0.040}%
    \special{pa 918 122}%
    \special{pa 942 158}%
    \special{fp}%
    \special{sh 1.000}%
    \special{pa 914 80}%
    \special{pa 942 158}%
    \special{pa 881 103}%
    \special{pa 914 80}%
    \special{fp}%
    \special{pa 1490 550}%
    \special{pa 1490 270}%
    \special{da 0.040}%
    \special{pa 1490 438}%
    \special{pa 1490 382}%
    \special{fp}%
    \special{sh 1.000}%
    \special{pa 1470 462}%
    \special{pa 1490 382}%
    \special{pa 1510 462}%
    \special{pa 1470 462}%
    \special{fp}%
    \special{pa 1630 50}%
    \special{pa 1510 230}%
    \special{da 0.040}%
    \special{pa 1582 122}%
    \special{pa 1558 158}%
    \special{fp}%
    \special{sh 1.000}%
    \special{pa 1619 103}%
    \special{pa 1558 158}%
    \special{pa 1586 80}%
    \special{pa 1619 103}%
    \special{fp}%
    \special{pa 1350 50}%
    \special{pa 1470 230}%
    \special{da 0.040}%
    \special{pa 1398 122}%
    \special{pa 1422 158}%
    \special{fp}%
    \special{sh 1.000}%
    \special{pa 1394 80}%
    \special{pa 1422 158}%
    \special{pa 1361 103}%
    \special{pa 1394 80}%
    \special{fp}%
    \special{pa 470 730}%
    \special{pa 850 810}%
    \special{da 0.040}%
    \special{pa 622 762}%
    \special{pa 698 778}%
    \special{fp}%
    \special{sh 1.000}%
    \special{pa 624 742}%
    \special{pa 698 778}%
    \special{pa 616 781}%
    \special{pa 624 742}%
    \special{fp}%
    \special{pa 2030 730}%
    \special{pa 1650 810}%
    \special{da 0.040}%
    \special{pa 1878 762}%
    \special{pa 1802 778}%
    \special{fp}%
    \special{sh 1.000}%
    \special{pa 1884 781}%
    \special{pa 1802 778}%
    \special{pa 1876 742}%
    \special{pa 1884 781}%
    \special{fp}%
    \special{pa 1250 350}%
    \special{pa 1250 730}%
    \special{da 0.040}%
    \special{pa 1250 502}%
    \special{pa 1250 578}%
    \special{fp}%
    \special{sh 1.000}%
    \special{pa 1270 498}%
    \special{pa 1250 578}%
    \special{pa 1230 498}%
    \special{pa 1270 498}%
    \special{fp}%
    \hbox{\vrule depth1.630in width0pt height 0pt}%
    \kern 2.500in
  }%
}%
}
\caption{Discharging for Lemma~\ref{10facelem}\label{fig10face}}
\end{figure}

\vspace{-1pc}

\begin{theorem}[\cite{B1996,SZ0}]\label{10face}
Every plane graph having no 4-cycle and no $j$-face with $5\le j\le 9$ is
3-colorable.
\end{theorem}
\begin{proof}
A smallest counterexample $G$ must be 4-critical, and hence it has minimum
degree at least 3 and is 2-connected.  Since there is no $4$-cycle, no two
3-faces share an edge.
By Lemma~\ref{10facelem}, we may thus assume that $G$ is embedded with at least
one 10-face $C$, whose vertices all have degree 3.  Let $f$ be a proper
3-coloring of $G-V(C)$.  Since each vertex on $C$ has exactly one neighbor
outside $C$, two colors remain available at each vertex of $C$.  Since even
cycles are 2-choosable, the coloring can be completed.
\end{proof}

\medskip

{\small

\begin{exercise}
Let $G$ be a simple plane graph with $\delta(G)\ge 3$.  Prove that $G$ has a
$3$-vertex on a $5^-$-face or a $5^-$-vertex on a triangle.
\end{exercise}

\begin{exercise}
(Lebesgue~\cite{Le})
Strengthen the previous exercise by proving that every plane graph $G$ with
$\delta(G)\ge3$ contains a $3$-vertex on a $5^-$-face, a $4$-vertex on a
$3$-face, or a $5$-vertex with four incident $3$-faces.  (Comment: Lebesgue
phrased the proof only for $3$-connected plane graphs.)
\end{exercise}

\begin{exercise}\label{kotzsharp}
Construct planar graphs to show that the bounds in Lemma~\ref{kotzig} are sharp.
That is, none of the values $10$, $7$, $6$ can be reduced in the statement that
normal plane maps have a $3$-vertex with a $10^-$-neighbor, a $4$-vertex with a
$7^-$-neighbor, or a $5$-vertex with a $6^-$-neighbor.
\end{exercise}

\begin{exercise}\label{Franklin}
Prove that planarity is needed in Lemma~\ref{kotzig} by showing
that a graph $G$ with $\Mad(G)<6$ and $\delta(G)=5$ need not have a $5$-vertex
with any $6^-$-neighbor.
\end{exercise}

\begin{exercise}\label{trisharp}
Prove that requiring minimum degree $5$ in Theorem~\ref{lighttri} is necessary,
by constructing for each $k\in\NN$ a planar graph with minimum degree $4$
having no triangle with weight at most $k$.
\end{exercise}

\begin{exercise}
(Cranston~\cite{C1})
Let $G$ be a plane graph with $\Delta(G)\ge 7$.  Prove that $G$ has either two
$3$-faces with a common edge or an edge with weight at most $\Delta(G)+2$.
Conclude that if $G$ is a plane graph with $\Delta(G)\ge 7$ and no two
$3$-faces sharing an edge, then $G$ is $(\Delta(G)+1)$-edge-choosable.
(Hint: Use balanced charging.  Comment: Cranston proved that the same
conditions are also sufficient when $\Delta(G)\ge6$, which implies several
earlier results.)
\end{exercise}

\begin{exercise}\label{edgir1}
The argument in Remark~\ref{2bound} shows that $\FR{12}5$ is the largest $b$
such that $\Mad(G)<b$ guarantees adjacent $2$-vertices in $G$.  Note that
$\Mad(G)<\FR{12}5$ when $G$ is planar with girth at least $12$.  Prove that
planar graphs with girth at least $11$ have adjacent $2$-vertices, and provide
a construction to show that the conclusion fails for some planar graph with
girth $10$.
\end{exercise}

\begin{exercise}
(Gr\"unbaum~\cite{Gru})
Prove that if a planar graph has no edges joining $5$-vertices, then it has
at least $60$ edges whose endpoints have degrees $5$ and $6$.
\end{exercise}

\begin{exercise}
Planar graphs with girth at least $6$ satisfy $\Mad(G)<3$, so by
Lemma~\ref{avd3-25} each such graph has a $2$-vertex with a $5^-$-neighbor.
Show that this is sharp even for planar graphs by constructing a planar graph
with girth $6$ having no edge of weight at most $6$.
\end{exercise}


\begin{exercise}\label{edgirn}
Determine whether a planar graph with girth at least $4$ and minimum degree
$3$ must have a $3$-vertex with a $4^-$-neighbor.  Construct a planar graph
$G_k$ with girth $4$ and minimum degree $3$ in which the distance between
$3$-vertices is at least $k$.  Construct a planar graph $H_k$ with minimum
degree $5$ in which the distance between $5$-vertices is at least $k$.
\end{exercise}

\begin{exercise}
Let $G$ be a graph with $\delta(G)=3$ and $\mad(G)<\FR{10}3$.  Prove that
$G$ has a $3$-vertex whose neighbors have degree-sum at most $10$.  Prove that
this result is sharp even in the family of planar graphs with girth at least
$5$ by constructing such a graph in which no $3$-vertex has three $3$-neighbors.
(Comment: G. Tardos constructed such a graph with $98$ vertices.)
\end{exercise}

\begin{exercise}\label{exfirst}
(Borodin~\cite{B1996e})
Prove that every planar graph with minimum degree $5$ contains two $3$-faces
sharing an edge with weight at most $11$.
(Hint: Use vertex charging, with $5$-vertices taking $\FR12$ from incident
$4^+$-faces and the remaining needed charge from $7^+$-neighbors.)
\end{exercise}

\begin{exercise}\label{albertson}
Prove that every plane triangulation with minimum degree 5 has two $3$-faces
sharing an edge such that the non-shared vertices have degree-sum at most $11$.
(Hint: Use vertex charging; $6$-vertices that give charge to $5$-neighbors will
need charge from $7^+$-neighbors.
Comment: Albertson~\cite{A} used this configuration in a proof that
$\al(G)\ge \FR{2n}9$ when $G$ is an $n$-vertex planar graph with no separating
triangle, without using the Four Color Theorem or the language of discharging.)
\end{exercise}


\begin{exercise}\label{planarb}
Prove inductively that every planar graph decomposes into three forests.
(Hint: Reduce to triangulations, and then apply the induction hypothesis
to a smaller graph obtained by deleting a light vertex and triangulating
the resulting face.  There are a number of cases when the deleted vertex has
degree $5$, depending on the usage of the two added edges.)  Use
Lemma~\ref{12vert} and more detailed analysis to prove that the third
forest can be guaranteed to have maximum degree at most $9$.  (Comment:
The second part of this exercise is long.  See Balogh et al.~\cite{BKPY} for
maximum degree at most $8$.)
\end{exercise}

\begin{exercise}\label{e55}
Let $G$ be a planar graph with $\delta(G)=5$.  With $e_{i,j}$ denoting the
numbers of edges with endpoints of degrees $i$ and $j$, prove
$\FR{26}{11}e_{5,5}+e_{5,6}\ge60$.  (Comment: Borodin and Sanders~\cite{BS}
proved the stronger result $\FR73e_{5,5}+e_{5,6}\ge60$; the coefficients
are sharp.)
\end{exercise}

\begin{exercise}\label{edgir7}
Reprove Lemma~\ref{girth7} by using balanced charging to prove that every
planar graph with girth at least $7$ and minimum degree at least $2$ has a
$2$-vertex adjacent to a $3^-$-vertex.  Prove that the conclusion does not
always hold when $\Mad(G)<\FR{14}5$ (thus planarity is needed).  Show that the
conclusion does not hold for all planar graphs with girth $6$.
\end{exercise}

\begin{exercise}\label{4dynam}
(Kim--Park~\cite{KP})
Prove that among the planar graphs with girth at least $7$, a minimal graph
that is not dynamically $4$-choosable cannot contain a $2$-vertex with a
$3^-$-neighbor.  (Comment: With Lemma~\ref{girth7}, this proves that every
planar graph with girth at least $7$ is dynamically $4$-choosable.  Note that
$\Mad(G)<\FR{14}5$ when $G$ is planar with girth at least $7$, but
$\Mad(G)<\FR{14}5$ is not sufficient for dynamic $4$-choosability.)
\end{exercise}

\begin{exercise}\label{noac4}
(Gr\"unbaum~\cite{Gru}, Kostochka--Mel'nikov~\cite{KMe})
Prove that the two graphs in Figure~\ref{fignoac} are not acylically
$4$-colorable.  The half-edges leaving the figure on the right lead to an
additional vertex having the same neighborhood as the central vertex.
\end{exercise}

\begin{figure}[hbt]
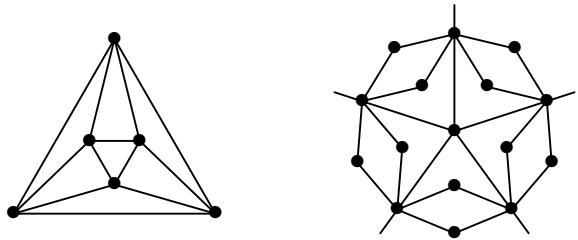

\gpic{
\expandafter\ifx\csname graph\endcsname\relax \csname newbox\endcsname\graph\fi
\expandafter\ifx\csname graphtemp\endcsname\relax \csname newdimen\endcsname\graphtemp\fi
\setbox\graph=\vtop{\vskip 0pt\hbox{%
    \graphtemp=.5ex\advance\graphtemp by 0.178in
    \rlap{\kern 0.592in\lower\graphtemp\hbox to 0pt{\hss $\bu$\hss}}%
    \graphtemp=.5ex\advance\graphtemp by 1.093in
    \rlap{\kern 1.120in\lower\graphtemp\hbox to 0pt{\hss $\bu$\hss}}%
    \graphtemp=.5ex\advance\graphtemp by 1.093in
    \rlap{\kern 0.064in\lower\graphtemp\hbox to 0pt{\hss $\bu$\hss}}%
    \special{pn 11}%
    \special{pa 592 178}%
    \special{pa 1120 1093}%
    \special{fp}%
    \special{pa 1120 1093}%
    \special{pa 64 1093}%
    \special{fp}%
    \special{pa 64 1093}%
    \special{pa 592 178}%
    \special{fp}%
    \graphtemp=.5ex\advance\graphtemp by 0.941in
    \rlap{\kern 0.592in\lower\graphtemp\hbox to 0pt{\hss $\bu$\hss}}%
    \graphtemp=.5ex\advance\graphtemp by 0.712in
    \rlap{\kern 0.460in\lower\graphtemp\hbox to 0pt{\hss $\bu$\hss}}%
    \graphtemp=.5ex\advance\graphtemp by 0.712in
    \rlap{\kern 0.724in\lower\graphtemp\hbox to 0pt{\hss $\bu$\hss}}%
    \special{pa 592 941}%
    \special{pa 460 712}%
    \special{fp}%
    \special{pa 460 712}%
    \special{pa 724 712}%
    \special{fp}%
    \special{pa 724 712}%
    \special{pa 592 941}%
    \special{fp}%
    \special{pa 592 941}%
    \special{pa 64 1093}%
    \special{fp}%
    \special{pa 64 1093}%
    \special{pa 460 712}%
    \special{fp}%
    \special{pa 460 712}%
    \special{pa 592 178}%
    \special{fp}%
    \special{pa 592 178}%
    \special{pa 724 712}%
    \special{fp}%
    \special{pa 724 712}%
    \special{pa 1120 1093}%
    \special{fp}%
    \special{pa 1120 1093}%
    \special{pa 592 941}%
    \special{fp}%
    \graphtemp=.5ex\advance\graphtemp by 0.153in
    \rlap{\kern 2.371in\lower\graphtemp\hbox to 0pt{\hss $\bu$\hss}}%
    \graphtemp=.5ex\advance\graphtemp by 0.504in
    \rlap{\kern 2.855in\lower\graphtemp\hbox to 0pt{\hss $\bu$\hss}}%
    \graphtemp=.5ex\advance\graphtemp by 1.072in
    \rlap{\kern 2.670in\lower\graphtemp\hbox to 0pt{\hss $\bu$\hss}}%
    \graphtemp=.5ex\advance\graphtemp by 1.072in
    \rlap{\kern 2.072in\lower\graphtemp\hbox to 0pt{\hss $\bu$\hss}}%
    \graphtemp=.5ex\advance\graphtemp by 0.504in
    \rlap{\kern 1.888in\lower\graphtemp\hbox to 0pt{\hss $\bu$\hss}}%
    \graphtemp=.5ex\advance\graphtemp by 0.661in
    \rlap{\kern 2.371in\lower\graphtemp\hbox to 0pt{\hss $\bu$\hss}}%
    \special{pa 2371 661}%
    \special{pa 2371 0}%
    \special{fp}%
    \special{pa 2371 661}%
    \special{pa 3000 457}%
    \special{fp}%
    \special{pa 2371 661}%
    \special{pa 2760 1196}%
    \special{fp}%
    \special{pa 2371 661}%
    \special{pa 1983 1196}%
    \special{fp}%
    \special{pa 2371 661}%
    \special{pa 1743 457}%
    \special{fp}%
    \special{pa 2371 153}%
    \special{pa 2541 428}%
    \special{fp}%
    \special{pa 2541 428}%
    \special{pa 2855 504}%
    \special{fp}%
    \special{pa 2855 504}%
    \special{pa 2686 228}%
    \special{fp}%
    \special{pa 2686 228}%
    \special{pa 2371 153}%
    \special{fp}%
    \special{pa 2855 504}%
    \special{pa 2645 750}%
    \special{fp}%
    \special{pa 2645 750}%
    \special{pa 2670 1072}%
    \special{fp}%
    \special{pa 2670 1072}%
    \special{pa 2880 826}%
    \special{fp}%
    \special{pa 2880 826}%
    \special{pa 2855 504}%
    \special{fp}%
    \special{pa 2670 1072}%
    \special{pa 2371 949}%
    \special{fp}%
    \special{pa 2371 949}%
    \special{pa 2072 1072}%
    \special{fp}%
    \special{pa 2072 1072}%
    \special{pa 2371 1196}%
    \special{fp}%
    \special{pa 2371 1196}%
    \special{pa 2670 1072}%
    \special{fp}%
    \special{pa 2072 1072}%
    \special{pa 2098 750}%
    \special{fp}%
    \special{pa 2098 750}%
    \special{pa 1888 504}%
    \special{fp}%
    \special{pa 1888 504}%
    \special{pa 1863 826}%
    \special{fp}%
    \special{pa 1863 826}%
    \special{pa 2072 1072}%
    \special{fp}%
    \special{pa 1888 504}%
    \special{pa 2202 428}%
    \special{fp}%
    \special{pa 2202 428}%
    \special{pa 2371 153}%
    \special{fp}%
    \special{pa 2371 153}%
    \special{pa 2057 228}%
    \special{fp}%
    \special{pa 2057 228}%
    \special{pa 1888 504}%
    \special{fp}%
    \graphtemp=.5ex\advance\graphtemp by 0.428in
    \rlap{\kern 2.541in\lower\graphtemp\hbox to 0pt{\hss $\bu$\hss}}%
    \graphtemp=.5ex\advance\graphtemp by 0.228in
    \rlap{\kern 2.686in\lower\graphtemp\hbox to 0pt{\hss $\bu$\hss}}%
    \graphtemp=.5ex\advance\graphtemp by 0.750in
    \rlap{\kern 2.645in\lower\graphtemp\hbox to 0pt{\hss $\bu$\hss}}%
    \graphtemp=.5ex\advance\graphtemp by 0.826in
    \rlap{\kern 2.880in\lower\graphtemp\hbox to 0pt{\hss $\bu$\hss}}%
    \graphtemp=.5ex\advance\graphtemp by 0.949in
    \rlap{\kern 2.371in\lower\graphtemp\hbox to 0pt{\hss $\bu$\hss}}%
    \graphtemp=.5ex\advance\graphtemp by 1.196in
    \rlap{\kern 2.371in\lower\graphtemp\hbox to 0pt{\hss $\bu$\hss}}%
    \graphtemp=.5ex\advance\graphtemp by 0.750in
    \rlap{\kern 2.098in\lower\graphtemp\hbox to 0pt{\hss $\bu$\hss}}%
    \graphtemp=.5ex\advance\graphtemp by 0.826in
    \rlap{\kern 1.863in\lower\graphtemp\hbox to 0pt{\hss $\bu$\hss}}%
    \graphtemp=.5ex\advance\graphtemp by 0.428in
    \rlap{\kern 2.202in\lower\graphtemp\hbox to 0pt{\hss $\bu$\hss}}%
    \graphtemp=.5ex\advance\graphtemp by 0.228in
    \rlap{\kern 2.057in\lower\graphtemp\hbox to 0pt{\hss $\bu$\hss}}%
    \hbox{\vrule depth1.259in width0pt height 0pt}%
    \kern 3.000in
  }%
}%
}
\caption{Graphs that are not acyclically $4$-colorable\label{fignoac}}
\end{figure}

\begin{exercise}\label{gir5}
Complete the proof of Lemma~\ref{gir5dis}.
\end{exercise}

\begin{exercise}
(Dvo\v{r}ak--Kawarabayashi--Thomas~\cite{DKT})
Let $C$ be the outer boundary in a 2-connected triangle-free plane graph $G$
that is not a cycle.  If $C$ has length at most $6$, and every vertex not on
$C$ has degree at least $3$, then $G$ contains a bounded $4$-face or a proper
$5$-face, where a $5$-face is {\it proper} if (at least) four of its vertices
have degree $3$ and are not on $C$.  (Comment: This result was used
in~\cite{DKT} to give a new proof of Gr\"otzsch's Theorem~\cite{Gro} that
triangle-free planar graphs are $3$-colorable.  The proof in~\cite{DKT} used
vertex charging, but using face charging is simpler.)
\end{exercise}

\begin{exercise}\label{no45}
Let $G$ be a plane graph having no $4$-cycle and no face with length in
$\{4,\ldots,k\}$.  Use discharging to prove that the average face length in $G$
is at least $6-\FR{18}{k+4}$.  Conclude that $\mad(G)<3+\FR9{2k-1}$.  In
particular, $\mad(G)<4$ when $G$ is a plane graph having no $4$-face or
$5$-face.
\end{exercise}

\begin{exercise}\label{wu}
(Wu~\cite{jlw}) Strengthen Lemma~\ref{altcyc} to show that every planar graph
contains an edge of weight at most $15$ or a $2$-alternating cycle such that
some high-degree vertex on the cycle has an additional $2$-neighbor outside the
cycle.  Conclude that $\larb(G)=\cdgh$ for every planar graph $G$ with
$\Delta(G)\ge 13$.
\end{exercise}

\begin{exercise}
Construct a planar graph with no $5$-cycle and a planar graph with no $4$-cycle
that are not $3$-colorable.
\end{exercise}

}

\section{List Coloring}\label{seclist}

List coloring (Definition~\ref{Dlist}) was invented in the 1970s by
Vizing~\cite{V4} and by Erd\H{o}s, Rubin, and Taylor~\cite{ERT}.  As we noted,
reducibility arguments for a coloring property often extend to reducibility for
the corresponding list coloring property, especially when made just by choosing
colors for vertices in a particular order.  For example, the famous result of
Brooks~\cite{Bro} that $\chi(G)\le\Delta(G)$ when $G$ is a connected graph
that is not a complete graph or odd cycle was strengthened to
$\chil(G)\le\Delta(G)$ for such graphs in~\cite{V4} and in~\cite{ERT}
(without discharging).  For planar graphs, the beautiful result by
Thomassen~\cite{T94} that planar graphs are $5$-choosable (sharp
by~\cite{Voigt}) also does not use discharging.  Thomassen~\cite{T3cho}  also
provedthat planar graphs with girth at least $5$ are $3$-choosable.

In this section we take a closer look at several problems involving list
coloring in order to develop further techniques for discharging arguments.
We begin with a useful lemma.


\begin{lemma}[\cite{ERT}]\label{evencyc}
Even cycles are $2$-choosable.
\end{lemma}
\begin{proof}
We show that $C_{2t}$ is $L$-colorable when every list has size $2$.  If the
lists are identical, then choose the colors to alternate.  Otherwise, there are
adjacent vertices $x$ and $y$ such that $L(x)$ contains a color $c$ not in
$L(y)$.  Use $c$ on $x$, and then follow the path $C_{2t}-x$ from $x$ to $y$
to color the vertices other than $x$: at each new vertex, choose a color from
its list that was not chosen for the previous vertex.  Such a choice is always
available, and the chosen colors satisfy every edge because the colors chosen
on $x$ and $y$ differ.
\end{proof}

Coloring and list-coloring have been studied extensively for squares of graphs.
Given a graph $G$, let $G^2$ be the graph obtained from $G$ by adding edges to
join vertices that are distance $2$ apart in $G$.  The neighbors of a vertex $v$
in $G$ form a clique with $v$ in $G^2$, so always $\chi(G^2)\ge \Delta(G)+1$.
Proper coloring of $G^2$ has also been called {\it $2$-distance coloring} of
$G$, since vertices with the same color must be separated by distance at least
$2$.

Kostochka and Woodall~\cite{KW} conjectured that always $\chil(G^2)=\chi(G^2)$.
This was proved in special cases, but Kim and Park~\cite{KP0} provided
counterexamples.  They used orthogonal families of Latin squares to construct a
graph $G$ for prime $p$ such that $G^2$ is the complete $(2p-1)$-partite graph
$K_{p,\ldots,p}$; on such graphs, $\chil-\chi$ is unbounded.

Thus sufficient conditions for $\chil(G^2)=\Delta(G)+1$ hold only on special
classes but establish a strong property.  We present such a result to show how
a discharging proof is discovered.  The discharging method often begins with
configurations that are easy to show reducible.  A discharging proof of
unavoidability of a set of such configurations starts by forbidding them.
When discharging, we may encounter a situation that does not guarantee the
desired final charge on some vertices.  Instead of trying to adjust the
discharging rules, we may try to add this configuration to the unavoidable set,
allowing us to assume that it does not occur.  This approach succeeds if we can
show that the new configuration is reducible.

We use $N_G(v)$ for the neighborhood of a vertex $v$ in a graph $G$,
with $N_G[v]=N_G(v)\cup\{v\}$.  When $L$ is a list assignment for $G$,
an $L$-coloring of a subgraph $G'$ of $G$ is with respect to the restriction
of $L$ to the vertices of $G'$.

\begin{lemma}[Borodin--Ivanova--Neustroeva~\cite{BIN}]\label{sqred}
Fix $k\ge 4$.  Among graphs $G$ with $\Delta(G)\le k$, the following
configurations are reducible for the property $\chil(G^2)\le k+1$.

\noindent
(A) a $1^-$-vertex,

\noindent
(B) a $2$-thread joining a $(k-1)^-$-vertex and a $(k-2)^-$-vertex,

\noindent
(C) a cycle of length divisible by $4$ composed of $3$-threads whose endpoints
have degree $k$.

\end{lemma}
\begin{proof}
Let $L$ be a $(k+1)$-uniform assignment on $G$; Figure~\ref{figsqred}
shows (B) and (C).

If (A) occurs at a $1^-$-vertex $v$, then let $G'=G-v$.  An $L$-coloring of
$G'^2$ extends to an $L$-coloring of $G^2$, because at most $k$ colors need to
be avoided at $v$.

If (B) occurs, then $G$ has a path $\la x,u,v,y\ra$ such that $d(u)=d(v)=2$,
$d(x)\le k-1$, and $d(y)\le k-2$.  With distance $3$ between $x$ and $y$, we
have $(G-\{u,v\})^2=G^2-\{u,v\}$.  Let $G'=G-\{u,v\}$.  By minimality, $G'^2$
has an $L$-coloring $\phi$.  In $G$, the color on $u$ must avoid the colors on
$\{x,y\}\cup N_{G'}(x)$.  Since $d(x)\le k-1$ and $|L(u)|=k+1$, a color is
available for $u$.  Now the color on $v$ must avoid those on
$\{x,y,u\}\cup N_{G'}(y)$.  Since $d(y)\le k-2$ and $|L(u)|=k+1$, a color is
available for $v$.

\begin{figure}[h]
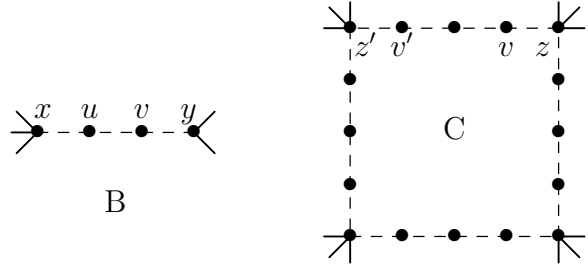

\gpic{
\expandafter\ifx\csname graph\endcsname\relax \csname newbox\endcsname\graph\fi
\expandafter\ifx\csname graphtemp\endcsname\relax \csname newdimen\endcsname\graphtemp\fi
\setbox\graph=\vtop{\vskip 0pt\hbox{%
    \graphtemp=.5ex\advance\graphtemp by 1.227in
    \rlap{\kern 1.773in\lower\graphtemp\hbox to 0pt{\hss $\bu$\hss}}%
    \graphtemp=.5ex\advance\graphtemp by 1.227in
    \rlap{\kern 2.045in\lower\graphtemp\hbox to 0pt{\hss $\bu$\hss}}%
    \graphtemp=.5ex\advance\graphtemp by 1.227in
    \rlap{\kern 2.318in\lower\graphtemp\hbox to 0pt{\hss $\bu$\hss}}%
    \graphtemp=.5ex\advance\graphtemp by 1.227in
    \rlap{\kern 2.591in\lower\graphtemp\hbox to 0pt{\hss $\bu$\hss}}%
    \graphtemp=.5ex\advance\graphtemp by 1.227in
    \rlap{\kern 2.864in\lower\graphtemp\hbox to 0pt{\hss $\bu$\hss}}%
    \graphtemp=.5ex\advance\graphtemp by 0.955in
    \rlap{\kern 2.864in\lower\graphtemp\hbox to 0pt{\hss $\bu$\hss}}%
    \graphtemp=.5ex\advance\graphtemp by 0.682in
    \rlap{\kern 2.864in\lower\graphtemp\hbox to 0pt{\hss $\bu$\hss}}%
    \graphtemp=.5ex\advance\graphtemp by 0.409in
    \rlap{\kern 2.864in\lower\graphtemp\hbox to 0pt{\hss $\bu$\hss}}%
    \graphtemp=.5ex\advance\graphtemp by 0.955in
    \rlap{\kern 1.773in\lower\graphtemp\hbox to 0pt{\hss $\bu$\hss}}%
    \graphtemp=.5ex\advance\graphtemp by 0.682in
    \rlap{\kern 1.773in\lower\graphtemp\hbox to 0pt{\hss $\bu$\hss}}%
    \graphtemp=.5ex\advance\graphtemp by 0.409in
    \rlap{\kern 1.773in\lower\graphtemp\hbox to 0pt{\hss $\bu$\hss}}%
    \graphtemp=.5ex\advance\graphtemp by 0.136in
    \rlap{\kern 1.773in\lower\graphtemp\hbox to 0pt{\hss $\bu$\hss}}%
    \graphtemp=.5ex\advance\graphtemp by 0.136in
    \rlap{\kern 2.045in\lower\graphtemp\hbox to 0pt{\hss $\bu$\hss}}%
    \graphtemp=.5ex\advance\graphtemp by 0.136in
    \rlap{\kern 2.318in\lower\graphtemp\hbox to 0pt{\hss $\bu$\hss}}%
    \graphtemp=.5ex\advance\graphtemp by 0.136in
    \rlap{\kern 2.591in\lower\graphtemp\hbox to 0pt{\hss $\bu$\hss}}%
    \graphtemp=.5ex\advance\graphtemp by 0.136in
    \rlap{\kern 2.864in\lower\graphtemp\hbox to 0pt{\hss $\bu$\hss}}%
    \special{pn 8}%
    \special{pa 1773 1227}%
    \special{pa 2864 1227}%
    \special{pa 2864 136}%
    \special{pa 1773 136}%
    \special{pa 1773 1227}%
    \special{da 0.055}%
    \special{pn 11}%
    \special{pa 1773 1227}%
    \special{pa 1636 1227}%
    \special{fp}%
    \special{pa 1773 1227}%
    \special{pa 1773 1364}%
    \special{fp}%
    \special{pa 1773 1227}%
    \special{pa 1664 1336}%
    \special{fp}%
    \special{pa 1773 136}%
    \special{pa 1636 136}%
    \special{fp}%
    \special{pa 1773 136}%
    \special{pa 1773 0}%
    \special{fp}%
    \special{pa 1773 136}%
    \special{pa 1664 27}%
    \special{fp}%
    \special{pa 2864 1227}%
    \special{pa 3000 1227}%
    \special{fp}%
    \special{pa 2864 1227}%
    \special{pa 2864 1364}%
    \special{fp}%
    \special{pa 2864 1227}%
    \special{pa 2973 1336}%
    \special{fp}%
    \special{pa 2864 136}%
    \special{pa 3000 136}%
    \special{fp}%
    \special{pa 2864 136}%
    \special{pa 2864 0}%
    \special{fp}%
    \special{pa 2864 136}%
    \special{pa 2973 27}%
    \special{fp}%
    \graphtemp=.5ex\advance\graphtemp by 0.245in
    \rlap{\kern 2.591in\lower\graphtemp\hbox to 0pt{\hss $v$\hss}}%
    \graphtemp=.5ex\advance\graphtemp by 0.245in
    \rlap{\kern 2.782in\lower\graphtemp\hbox to 0pt{\hss $z$\hss}}%
    \graphtemp=.5ex\advance\graphtemp by 0.245in
    \rlap{\kern 2.045in\lower\graphtemp\hbox to 0pt{\hss $v'$\hss}}%
    \graphtemp=.5ex\advance\graphtemp by 0.245in
    \rlap{\kern 1.855in\lower\graphtemp\hbox to 0pt{\hss $z'$\hss}}%
    \graphtemp=.5ex\advance\graphtemp by 0.682in
    \rlap{\kern 0.136in\lower\graphtemp\hbox to 0pt{\hss $\bu$\hss}}%
    \graphtemp=.5ex\advance\graphtemp by 0.682in
    \rlap{\kern 0.409in\lower\graphtemp\hbox to 0pt{\hss $\bu$\hss}}%
    \graphtemp=.5ex\advance\graphtemp by 0.682in
    \rlap{\kern 0.682in\lower\graphtemp\hbox to 0pt{\hss $\bu$\hss}}%
    \graphtemp=.5ex\advance\graphtemp by 0.682in
    \rlap{\kern 0.955in\lower\graphtemp\hbox to 0pt{\hss $\bu$\hss}}%
    \special{pn 8}%
    \special{pa 136 682}%
    \special{pa 955 682}%
    \special{da 0.055}%
    \special{pn 11}%
    \special{pa 1064 573}%
    \special{pa 955 682}%
    \special{fp}%
    \special{pa 955 682}%
    \special{pa 1064 791}%
    \special{fp}%
    \special{pa 136 682}%
    \special{pa 0 682}%
    \special{fp}%
    \special{pa 136 682}%
    \special{pa 27 573}%
    \special{fp}%
    \special{pa 136 682}%
    \special{pa 27 791}%
    \special{fp}%
    \graphtemp=.5ex\advance\graphtemp by 0.573in
    \rlap{\kern 0.164in\lower\graphtemp\hbox to 0pt{\hss $x$\hss}}%
    \graphtemp=.5ex\advance\graphtemp by 0.573in
    \rlap{\kern 0.409in\lower\graphtemp\hbox to 0pt{\hss $u$\hss}}%
    \graphtemp=.5ex\advance\graphtemp by 0.573in
    \rlap{\kern 0.682in\lower\graphtemp\hbox to 0pt{\hss $v$\hss}}%
    \graphtemp=.5ex\advance\graphtemp by 0.573in
    \rlap{\kern 0.927in\lower\graphtemp\hbox to 0pt{\hss $y$\hss}}%
    \graphtemp=.5ex\advance\graphtemp by 1.064in
    \rlap{\kern 0.545in\lower\graphtemp\hbox to 0pt{\hss B\hss}}%
    \graphtemp=.5ex\advance\graphtemp by 0.682in
    \rlap{\kern 2.318in\lower\graphtemp\hbox to 0pt{\hss C\hss}}%
    \hbox{\vrule depth1.364in width0pt height 0pt}%
    \kern 3.000in
  }%
}%
}
\caption{Reducible configurations in $G$ for $\chil(G^2)\le k+1$ (with $k=5$)
\label{figsqred}}
\end{figure}

If (C) occurs, then obtain $G'$ from $G$ by deleting the $2$-vertices on
the given cycle $C$.  Again $G'^2$ is the subgraph of $G^2$ induced by $V(G')$.
Let $v$ be a deleted vertex having a $k$-neighbor $z$ in $G$.  The color on
$v$ must avoid those on $z$ and all $k-2$ neighbors of $z$ in $G'$.  Since
$|L(v)|=k+1$, at least two colors are available for $v$.  These neighbors of
$k$-vertices on $C$ induce an even cycle in $G^2$.  By Lemma~\ref{evencyc}, we
can extend the coloring of $G'$ to these vertices.  Finally, the $2$-vertices
at the centers of the $3$-threads have only four neighbors in $G^2$, all of
which are now colored.  Since $k\ge4$, a color remains available at each such
vertex.
\end{proof}

The discharging argument to guarantee these reducible configurations is our
first encounter with a global notion of discharging.  We introduce a {\it pot}
of charge.  By allowing vertices to contribute charge to the pot or draw 
charge from it, we permit charge to move long distances in the graph.
The pot starts with charge $0$ and must end with nonnegative charge.  This
prevents the pot from supplying charge to the graph, so making all the 
vertices happy still contradicts the initial hypothesis on average degree.

We will also see that the configurations originally found to be 
reducible may not suffice.

\begin{theorem}[\cite{BLP2,CS}]\label{listsq}
If $\Delta(G)\le6$ and $\mad(G)< \FR52$, then $\chil(G^2)\le7$.
\end{theorem}
\begin{proof}
Let $G$ be a minimal counterexample. Let $k=6$.  By Lemma~\ref{sqred}(A),
we may assume $\delta(G)\ge 2$.  By Lemma~\ref{sqred}(B), $G$ has no $4$-thread
(or longer), and $3$-threads have $k$-vertices at both ends.  By
Lemma~\ref{sqred}(C), the union of the $3$-threads is an acyclic subgraph $H$.
Hence the number of $6$-vertices is greater than the number of $3$-threads.


We now seek discharging rules to prove that if $\mad(G)<\FR52$ and
$\delta(G)\ge2$, then some configuration of type (B) or (C) in
Lemma~\ref{sqred} must occur.  This will not quite work; we will need to add
more configurations to the set, but they will be reducible.

\medskip\noindent
(R1) Vertices with degree $5$ or $6$ give $\FR12$ to each neighbor.

\noindent
(R2) A $2$-vertex with one neighbor of degree $2$ and one of degree
$3$ or $4$ takes $\FR12$ from the higher-degree neighbor.

\noindent
(R3) A $2$-vertex whose neighbors both have degree $3$ or $4$ takes $\FR14$
from each neighbor.

\noindent
(R4) Each $6$-vertex contributes $\FR12$ to the pot, and each $2$-vertex 
at the center of a $3$-thread takes $\FR12$ from the pot.

\medskip
Since there are more $6$-vertices than $3$-threads, the pot ends with
positive charge.  By the discharging rules, each $2$-vertex explicitly gains
charge $\FR12$ and ends happy.  A $5$-vertex can afford to give $\FR52$, and a
$6$-vertex can afford to give $\FR62$ to its neighbors plus $\FR12$ to the pot.

A $4$-vertex is unhappy if it loses more than $\FR32$ without having a
$5^+$-neighbor.  A $3$-vertex is unhappy if it loses more than $\FR12$
without having a $5^+$-neighbor.  Fortunately, the configurations in which
vertices can become unhappy are reducible for $\chil(G^2)\le7$, so their
occurrence causes no difficulty.  See Figure~\ref{fig433}.

\begin{figure}[h]
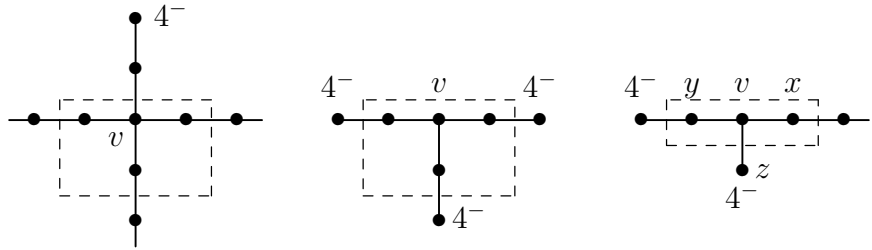

\gpic{
\expandafter\ifx\csname graph\endcsname\relax \csname newbox\endcsname\graph\fi
\expandafter\ifx\csname graphtemp\endcsname\relax \csname newdimen\endcsname\graphtemp\fi
\setbox\graph=\vtop{\vskip 0pt\hbox{%
    \graphtemp=.5ex\advance\graphtemp by 1.165in
    \rlap{\kern 0.662in\lower\graphtemp\hbox to 0pt{\hss $\bu$\hss}}%
    \graphtemp=.5ex\advance\graphtemp by 0.900in
    \rlap{\kern 0.662in\lower\graphtemp\hbox to 0pt{\hss $\bu$\hss}}%
    \graphtemp=.5ex\advance\graphtemp by 0.635in
    \rlap{\kern 0.132in\lower\graphtemp\hbox to 0pt{\hss $\bu$\hss}}%
    \graphtemp=.5ex\advance\graphtemp by 0.635in
    \rlap{\kern 0.397in\lower\graphtemp\hbox to 0pt{\hss $\bu$\hss}}%
    \graphtemp=.5ex\advance\graphtemp by 0.635in
    \rlap{\kern 0.662in\lower\graphtemp\hbox to 0pt{\hss $\bu$\hss}}%
    \graphtemp=.5ex\advance\graphtemp by 0.371in
    \rlap{\kern 0.662in\lower\graphtemp\hbox to 0pt{\hss $\bu$\hss}}%
    \graphtemp=.5ex\advance\graphtemp by 0.635in
    \rlap{\kern 0.926in\lower\graphtemp\hbox to 0pt{\hss $\bu$\hss}}%
    \graphtemp=.5ex\advance\graphtemp by 0.635in
    \rlap{\kern 1.191in\lower\graphtemp\hbox to 0pt{\hss $\bu$\hss}}%
    \graphtemp=.5ex\advance\graphtemp by 0.106in
    \rlap{\kern 0.662in\lower\graphtemp\hbox to 0pt{\hss $\bu$\hss}}%
    \special{pn 11}%
    \special{pa 662 1297}%
    \special{pa 662 106}%
    \special{fp}%
    \special{pa 0 635}%
    \special{pa 1324 635}%
    \special{fp}%
    \graphtemp=.5ex\advance\graphtemp by 0.106in
    \rlap{\kern 0.847in\lower\graphtemp\hbox to 0pt{\hss $4^-$\hss}}%
    \graphtemp=.5ex\advance\graphtemp by 0.737in
    \rlap{\kern 0.560in\lower\graphtemp\hbox to 0pt{\hss $v$\hss}}%
    \special{pn 8}%
    \special{pa 265 1032}%
    \special{pa 1059 1032}%
    \special{pa 1059 529}%
    \special{pa 265 529}%
    \special{pa 265 1032}%
    \special{da 0.053}%
    \graphtemp=.5ex\advance\graphtemp by 0.900in
    \rlap{\kern 2.250in\lower\graphtemp\hbox to 0pt{\hss $\bu$\hss}}%
    \graphtemp=.5ex\advance\graphtemp by 0.635in
    \rlap{\kern 1.985in\lower\graphtemp\hbox to 0pt{\hss $\bu$\hss}}%
    \graphtemp=.5ex\advance\graphtemp by 0.635in
    \rlap{\kern 2.250in\lower\graphtemp\hbox to 0pt{\hss $\bu$\hss}}%
    \graphtemp=.5ex\advance\graphtemp by 0.635in
    \rlap{\kern 2.515in\lower\graphtemp\hbox to 0pt{\hss $\bu$\hss}}%
    \graphtemp=.5ex\advance\graphtemp by 1.165in
    \rlap{\kern 2.250in\lower\graphtemp\hbox to 0pt{\hss $\bu$\hss}}%
    \graphtemp=.5ex\advance\graphtemp by 0.635in
    \rlap{\kern 1.721in\lower\graphtemp\hbox to 0pt{\hss $\bu$\hss}}%
    \graphtemp=.5ex\advance\graphtemp by 0.635in
    \rlap{\kern 2.779in\lower\graphtemp\hbox to 0pt{\hss $\bu$\hss}}%
    \special{pn 11}%
    \special{pa 2250 1165}%
    \special{pa 2250 635}%
    \special{fp}%
    \special{pa 1721 635}%
    \special{pa 2779 635}%
    \special{fp}%
    \graphtemp=.5ex\advance\graphtemp by 0.450in
    \rlap{\kern 2.250in\lower\graphtemp\hbox to 0pt{\hss $v$\hss}}%
    \special{pn 8}%
    \special{pa 1853 1032}%
    \special{pa 2647 1032}%
    \special{pa 2647 529}%
    \special{pa 1853 529}%
    \special{pa 1853 1032}%
    \special{da 0.053}%
    \graphtemp=.5ex\advance\graphtemp by 1.165in
    \rlap{\kern 2.409in\lower\graphtemp\hbox to 0pt{\hss $4^-$\hss}}%
    \graphtemp=.5ex\advance\graphtemp by 0.476in
    \rlap{\kern 1.721in\lower\graphtemp\hbox to 0pt{\hss $4^-$\hss}}%
    \graphtemp=.5ex\advance\graphtemp by 0.476in
    \rlap{\kern 2.779in\lower\graphtemp\hbox to 0pt{\hss $4^-$\hss}}%
    \graphtemp=.5ex\advance\graphtemp by 0.900in
    \rlap{\kern 3.838in\lower\graphtemp\hbox to 0pt{\hss $\bu$\hss}}%
    \graphtemp=.5ex\advance\graphtemp by 0.635in
    \rlap{\kern 3.574in\lower\graphtemp\hbox to 0pt{\hss $\bu$\hss}}%
    \graphtemp=.5ex\advance\graphtemp by 0.635in
    \rlap{\kern 3.838in\lower\graphtemp\hbox to 0pt{\hss $\bu$\hss}}%
    \graphtemp=.5ex\advance\graphtemp by 0.635in
    \rlap{\kern 4.103in\lower\graphtemp\hbox to 0pt{\hss $\bu$\hss}}%
    \graphtemp=.5ex\advance\graphtemp by 0.635in
    \rlap{\kern 4.368in\lower\graphtemp\hbox to 0pt{\hss $\bu$\hss}}%
    \graphtemp=.5ex\advance\graphtemp by 0.635in
    \rlap{\kern 3.309in\lower\graphtemp\hbox to 0pt{\hss $\bu$\hss}}%
    \special{pn 11}%
    \special{pa 3838 900}%
    \special{pa 3838 635}%
    \special{fp}%
    \special{pa 3309 635}%
    \special{pa 4500 635}%
    \special{fp}%
    \graphtemp=.5ex\advance\graphtemp by 1.059in
    \rlap{\kern 3.838in\lower\graphtemp\hbox to 0pt{\hss $4^-$\hss}}%
    \graphtemp=.5ex\advance\graphtemp by 0.476in
    \rlap{\kern 3.309in\lower\graphtemp\hbox to 0pt{\hss $4^-$\hss}}%
    \graphtemp=.5ex\advance\graphtemp by 0.450in
    \rlap{\kern 3.838in\lower\graphtemp\hbox to 0pt{\hss $v$\hss}}%
    \graphtemp=.5ex\advance\graphtemp by 0.900in
    \rlap{\kern 3.944in\lower\graphtemp\hbox to 0pt{\hss $z$\hss}}%
    \graphtemp=.5ex\advance\graphtemp by 0.450in
    \rlap{\kern 3.574in\lower\graphtemp\hbox to 0pt{\hss $y$\hss}}%
    \graphtemp=.5ex\advance\graphtemp by 0.450in
    \rlap{\kern 4.103in\lower\graphtemp\hbox to 0pt{\hss $x$\hss}}%
    \special{pn 8}%
    \special{pa 3441 768}%
    \special{pa 4235 768}%
    \special{pa 4235 529}%
    \special{pa 3441 529}%
    \special{pa 3441 768}%
    \special{da 0.053}%
    \hbox{\vrule depth1.297in width0pt height 0pt}%
    \kern 4.500in
  }%
}%
}

\vspace{-1pc}
\caption{Additional reducible configurations for
Theorem~\ref{listsq}\label{fig433}}
\end{figure}

For a $4$-vertex $v$ to lose more than $\FR32$ all its neighbors must be
$2$-vertices and at least three of the incident threads must be $2$-threads.
We show that this configuration is reducible.  Define $G'$ from $G$ by deleting
$v$ and its neighbors on three incident $2$-threads; note that $G'^2$ is the
subgraph of $G^2$ induced by $V(G')$.  Since $|N_{G^2}(v)\cap V(G')|=5$, we can
extend an $L$-coloring of $G'^2$ to $v$.  When we restore the deleted
$2$-neighbors of $v$, the numbers of vertices whose colors they must avoid are
$4,5,6$, respectively, so at each step a color is available.

For a $3$-vertex to end unhappy by losing more than $\FR12$, it must have
no $5^+$-neighbor (since it gives at most $\FR12$ to each neighbor).
It may give at least $\FR14$ to each of three neighbors or give $\FR12$ to one
neighbor and at least $\FR14$ to another.

In the first case, let $N_G(v)=\{x_1,x_2,x_3\}$, and let $G'=G-N_G[v]$.
The neighbor of $x_i$ other than $v$ has degree at most $4$.  As we restore
$N_G(v)$, the number of vertices whose colors they must avoid are $4,5,6$,
so at each step a color is available.  We can then replace $v$; it must avoid
the colors on six vertices.

In the second case, $v$ gives charge to exactly two $2$-neighbors, $x$ and
$y$, where $x$ lies on a $2$-thread and takes $\FR12$ from $v$, and the 
other neighbor of $y$ is a $4^-$-vertex.  Let $z$ be the third neighbor of $v$;
note that $d(z)\le4$.  With $S=\{v,x,y\}$, let $G'=G-S$; again $G'^2=G^2-S$.
Restore $v$, then $y$, then $x$.  As each is restored, its color is chosen from
its list to avoid the colors on at most six other vertices.
\end{proof}

Cranston and \v{S}krekovski~\cite{CS} proved more generally that if
$\Delta(G)\ge6$ and $\mad(G)<2+\FR{4\Delta(G)-8}{5\Delta(G)+2}$, then 
$\chil(G^2)=\Delta(G)+1$.  Thus when $\Mad(G)$ is sufficiently small compared
to $\Delta(G)$, the trivial lower bounds on $\chi(G^2)$ and $\chil(G^2)$ are
tight.  With a similar but shorter proof, Bonamy, L\'ev\^eque, and
Pinlou~\cite{BLP2} proved the less precise statement that for each positive
$\epsilon$, there exists $k_\epsilon$ such that $\chil(G^2)=\Delta(G)+1$ for
$\Delta(G)\ge k_\epsilon$ when $\mad(G)<\FR{14}5-\epsilon$.
In~\cite{BLP3} they extended this to $\mad(G)<3-\epsilon$.

Even for planar graphs and ordinary coloring, $\Mad(G)<4$ does not yield
$\chi(G)\le\Delta(G^2)+c$ for any constant $c$.  Note that girth $4$
implies $\Mad(G)<4$ when $G$ is planar.  Consider the $3$-vertex multigraph in
which each pair has multiplicity $k$; this is sometimes called the
{\it fat triangle}.
Subdividing each edge once yields a planar graph with girth $4$ and maximum
degree $2k$ whose square has chromatic number $3k$ (see Figure~\ref{figfat}).
Nevertheless,~\cite{BLP2} obtained a function $c$ such that if
$\Mad(G)<4-\epsilon$, then $\chil(G^2)\le\Delta(G)+c(\epsilon)$.
Yancey~\cite{Yan} refined this for large $\Delta(G)$, proving for $c\ge3$
that if $\mad(G)<4-\FR4{c+1}$ and $\Delta(G)$ is sufficiently large, then
$\chil(G^2)\le\Delta(G)+c$.

\begin{figure}[h]
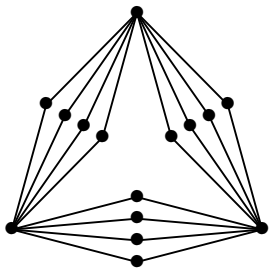

\gpic{
\expandafter\ifx\csname graph\endcsname\relax \csname newbox\endcsname\graph\fi
\expandafter\ifx\csname graphtemp\endcsname\relax \csname newdimen\endcsname\graphtemp\fi
\setbox\graph=\vtop{\vskip 0pt\hbox{%
    \graphtemp=.5ex\advance\graphtemp by 0.095in
    \rlap{\kern 0.750in\lower\graphtemp\hbox to 0pt{\hss $\bu$\hss}}%
    \graphtemp=.5ex\advance\graphtemp by 1.230in
    \rlap{\kern 1.405in\lower\graphtemp\hbox to 0pt{\hss $\bu$\hss}}%
    \graphtemp=.5ex\advance\graphtemp by 1.230in
    \rlap{\kern 0.095in\lower\graphtemp\hbox to 0pt{\hss $\bu$\hss}}%
    \graphtemp=.5ex\advance\graphtemp by 0.747in
    \rlap{\kern 0.930in\lower\graphtemp\hbox to 0pt{\hss $\bu$\hss}}%
    \graphtemp=.5ex\advance\graphtemp by 0.691in
    \rlap{\kern 1.029in\lower\graphtemp\hbox to 0pt{\hss $\bu$\hss}}%
    \graphtemp=.5ex\advance\graphtemp by 0.634in
    \rlap{\kern 1.127in\lower\graphtemp\hbox to 0pt{\hss $\bu$\hss}}%
    \graphtemp=.5ex\advance\graphtemp by 0.577in
    \rlap{\kern 1.225in\lower\graphtemp\hbox to 0pt{\hss $\bu$\hss}}%
    \special{pn 11}%
    \special{pa 750 95}%
    \special{pa 930 747}%
    \special{fp}%
    \special{pa 930 747}%
    \special{pa 1405 1230}%
    \special{fp}%
    \special{pa 1405 1230}%
    \special{pa 1029 691}%
    \special{fp}%
    \special{pa 1029 691}%
    \special{pa 750 95}%
    \special{fp}%
    \special{pa 750 95}%
    \special{pa 1127 634}%
    \special{fp}%
    \special{pa 1127 634}%
    \special{pa 1405 1230}%
    \special{fp}%
    \special{pa 1405 1230}%
    \special{pa 1225 577}%
    \special{fp}%
    \special{pa 1225 577}%
    \special{pa 750 95}%
    \special{fp}%
    \graphtemp=.5ex\advance\graphtemp by 1.060in
    \rlap{\kern 0.750in\lower\graphtemp\hbox to 0pt{\hss $\bu$\hss}}%
    \graphtemp=.5ex\advance\graphtemp by 1.173in
    \rlap{\kern 0.750in\lower\graphtemp\hbox to 0pt{\hss $\bu$\hss}}%
    \graphtemp=.5ex\advance\graphtemp by 1.287in
    \rlap{\kern 0.750in\lower\graphtemp\hbox to 0pt{\hss $\bu$\hss}}%
    \graphtemp=.5ex\advance\graphtemp by 1.400in
    \rlap{\kern 0.750in\lower\graphtemp\hbox to 0pt{\hss $\bu$\hss}}%
    \special{pa 1405 1230}%
    \special{pa 750 1060}%
    \special{fp}%
    \special{pa 750 1060}%
    \special{pa 95 1230}%
    \special{fp}%
    \special{pa 95 1230}%
    \special{pa 750 1173}%
    \special{fp}%
    \special{pa 750 1173}%
    \special{pa 1405 1230}%
    \special{fp}%
    \special{pa 1405 1230}%
    \special{pa 750 1287}%
    \special{fp}%
    \special{pa 750 1287}%
    \special{pa 95 1230}%
    \special{fp}%
    \special{pa 95 1230}%
    \special{pa 750 1400}%
    \special{fp}%
    \special{pa 750 1400}%
    \special{pa 1405 1230}%
    \special{fp}%
    \graphtemp=.5ex\advance\graphtemp by 0.747in
    \rlap{\kern 0.570in\lower\graphtemp\hbox to 0pt{\hss $\bu$\hss}}%
    \graphtemp=.5ex\advance\graphtemp by 0.691in
    \rlap{\kern 0.471in\lower\graphtemp\hbox to 0pt{\hss $\bu$\hss}}%
    \graphtemp=.5ex\advance\graphtemp by 0.634in
    \rlap{\kern 0.373in\lower\graphtemp\hbox to 0pt{\hss $\bu$\hss}}%
    \graphtemp=.5ex\advance\graphtemp by 0.577in
    \rlap{\kern 0.275in\lower\graphtemp\hbox to 0pt{\hss $\bu$\hss}}%
    \special{pa 95 1230}%
    \special{pa 570 747}%
    \special{fp}%
    \special{pa 570 747}%
    \special{pa 750 95}%
    \special{fp}%
    \special{pa 750 95}%
    \special{pa 471 691}%
    \special{fp}%
    \special{pa 471 691}%
    \special{pa 95 1230}%
    \special{fp}%
    \special{pa 95 1230}%
    \special{pa 373 634}%
    \special{fp}%
    \special{pa 373 634}%
    \special{pa 750 95}%
    \special{fp}%
    \special{pa 750 95}%
    \special{pa 275 577}%
    \special{fp}%
    \special{pa 275 577}%
    \special{pa 95 1230}%
    \special{fp}%
    \hbox{\vrule depth1.495in width0pt height 0pt}%
    \kern 1.500in
  }%
}%
}

\vspace{-1pc}

\caption{Construction with girth $4$ and $\chi(G^2)=3k$ (here $k=4$)
\label{figfat}}
\end{figure}
 
When $G$ is planar, larger girth restricts $\Mad(G)$ more tightly.  Motivated
by the subdivided fat triangle, Wang and Lih~\cite{WL} conjectured that for
$g\ge 5$, there exists $k_g$ such that $\Delta(G)\ge k_g$ implies
$\chi(G^2)=\Delta(G)+1$ when $G$ is a planar graph with girth at least $g$.
The conjecture is false for $g\in\{5,6\}$; \cite{BGINT} and \cite{DKNS} both
contain infinite sequences of planar graphs with girth 6, growing maximum
degree, and $\chi(G^2)=\Delta(G)+2$.

However, the Wang--Lih Conjecture holds and can be strengthened to list
coloring when $g\ge7$.  Ivanova~\cite{I} proved $\chil(G^2)=\Delta(G)+1$ for
planar $G$ having girth at least $7$ and $\Delta(G)\ge16$ (improving on
$\Delta(G)\ge30$ from~\cite{BGINT}), and she also showed that the thresholds
$10,6,5$ on $\Delta(G)$ are sufficient when $G$ has girth at least $8,10,12$,
respectively.

For girth $6$, Dvo\v{r}\'ak, Kr\'al', Nejedl\'y, and
\v{S}krekovski~\cite{DKNS} proved $\chi(G^2)\le \Delta(G)+2$ for planar $G$
with $\Delta(G)\ge 8821$ (they also conjectured $\chi(G^2)\le\Delta(G)+2$
for girth $5$ when $\Delta(G)$ is large enough).  For girth $6$,
Borodin and Ivanova~\cite{BI1} improved $\Delta(G)\ge 8821$ to $\Delta(G)\ge18$;
in \cite{BI2} and then \cite{BI2+} they showed that $\Delta(G)\ge 36$ and then
$\Delta(G)\ge24$ yields $\chil(G^2)\le \Delta(G)+2$~\cite{BI2}.

Bonamy, L\'ev\^eque, and Pinlou~\cite{BLP1} proved $\chil(G^2)\le \Delta(G)+2$
when $\Delta(G)\ge 17$ and $\Mad(G)<3$, regardless of planarity.  As we have
noted, the hypothesis ``$\mad(G)<3$'' in place of ``planar with girth at least
$6$'' yields a stronger result.

Now consider again the result of Cranston and \v{S}krekovski~\cite{CS}.
Reducing the bound on $\Mad(G)$ from $3$ to $2+\FR{4\Delta(G)-8}{5\Delta(G)+2}$
yields $\chil(G^2)=\Delta(G)+1$ rather than $\chil(G^2)\le\Delta(G)+2$, even
for the larger family where $\Delta(G)\ge6$.  Furthermore, as $\Delta(G)$
grows, the needed bound on $\Mad(G)$ tends to $\FR{14}5$, which is the bound
guaranteed for planar graphs with girth at least $7$.  Hence it seems plausible
that $\chil(G^2)=\Delta(G)+1$ for planar graphs with girth at least $7$ even
when $\Delta(G)\ge6$.  For fuller exploration, we suggest an open question.

\begin{question}\label{G2prob}
Among the family of graphs such that $\Delta(G)\ge k$, what is the largest
value $b_{j,k}$ such that $\Mad(G)<b_{j,k}$ implies $\chil(G^2)\le \Delta(G)+j$?
\end{question}

Next we weaken the requirements.  A coloring where vertices at distance $2$
have distinct colors but adjacent vertices need not is an {\it injective
coloring} (the coloring is injective on each vertex neighborhood).  For 
motivation, consider a network of transmitters that broadcast on fixed
frequencies; frequencies in a neighborhood must differ so that a receiver can
know which neighbor is sending the message.  The {\it injective chromatic
number}, written $\chi^i(G)$, is the minimum number of colors needed, and the
{\it injective choice number}, $\chil^i(G)$, is the least $k$ such that $G$ has
an injective $L$-coloring when $L$ is any $k$-uniform list assignment.

From the definition, always $\chi^i(G)\le \chi(G^2)$ and
$\chil^i(G)\le\chil(G^2)$.  The trivial lower bound on $\chi^i(G)$ is
$\Delta(G)$ rather than $\Delta(G)+1$.  We seek results like those above, with
a bound on $\chi^i(G)$ or $\chil^i(G)$ that is one less than the corresponding
bound for $\chi(G^2)$ or $\chil(G^2)$.  Again when $\Mad(G)$ is small relative
to $\Delta(G)$, the value is close to the lower bound.  In~\cite{BLP1}, for
example, it is noted that the proof there also yields
$\chil^i(G)\le\Delta(G)+1$ when $\Delta(G)\ge 17$ and $\Mad(G)<3$ (\cite{BLP3}
and~\cite{CS} also translate to injective coloring).

Nevertheless, the analogue of Problem~\ref{G2prob} for injective coloring
remains largely open.  When $j=0$, rather tight bounds on $\mad(G)$ suffice.
Cranston, Kim, and Yu~\cite{CKY1} proved that $\chi^i(G)=\Delta(G)$ when
$\Mad(G)<\FR{42}{19}$ and $\Delta(G)\ge3$.  Sharpness is not known, even
for $\Delta(G)=3$.  Subdividing one edge of $K_4$ yields a graph $H$ such that
$\chi'(H)>\Delta(H)$, and then subdividing every edge of $H$ yields a bipartite
graph $G$ such that $\chi^i(G)>\Delta(G)$ and $\Mad(G)=\FR73$.  The largest $b$
such that $\Mad(G)<b$ implies $\chi^i(G)=\Delta(G)$ when $\Delta(G)=3$ is not
known; it is at least $\FR{42}{19}$ and at most $\FR73$.

To yield $\chil^i(G)\le\Delta(G)+1$, it suffices to have $\Mad(G)\le\FR52$ when
$\Delta(G)\ge3$~\cite{CKY1}.  For $\Delta(G)\ge4$ this is fairly easy (it uses
Exercise~\ref{52CKY}); for $\Delta(G)\ge6$ it follows from~\cite{CS}.

To yield $\chil^i(G)\le\Delta(G)+2$, it suffices to have $\mad(G)<\FR{36}{13}$
when $\Delta(G)=3$~\cite{CKY2}; we will see that this is sharp.  For
$\Delta(G)\ge4$, it suffices to have $\Mad(G)<\FR{14}5$~\cite{CKY2}; the cases
$\Delta(G)\in\{4,5\}$ are difficult, and sharpness is not known.  Note that
when $\Delta(G)\ge4$ the allowed values of $\Mad(G)$ are larger than when
$\Delta(G)=3$; the loosest condition on $\Mad(G)$ that suffices for a given
bound on $\chi^i(G)-\Delta(G)$ should grow (somewhat) as $\Delta(G)$ grows.

We use one of these results to further explore how discharging arguments are
found.  In the discharging process, charge may travel distance $2$.

\begin{theorem}\label{36.13dis}
{\rm(\cite{CKY2})} If $\Delta(G)\le3$ and $\Mad(G)<\FR{36}{13}$, then
$\chil^i(G)\le 5$.
\vspace{-.5pc}
\end{theorem}
\begin{proof}
We present the discharging argument and leave the reducibility of the
configurations in the resulting unavoidable set to Exercise~\ref{36.13red}.  We
claim that every graph $G$ with $\Delta(G)=3$ and $\avd(G)<\FR{36}{13}$
contains one of the following configurations: a $1^-$-vertex, adjacent
$2$-vertices, a $3$-vertex with two $2$-neighbors, or adjacent $3$-vertices
each having a $2$-neighbor.

If none of these configurations occurs, then $\delta(G)\ge2$.  With initial
charge equal to degree, only $2$-vertices need charge; all other vertices are
$3$-vertices.  A way to allow $2$-vertices to reach charge $\FR{36}{13}$
without taking too much from $3$-vertices is as follows:

\medskip\noindent
(R1) Every $2$-vertex takes $\FR3{13}$ from each neighbor.

\noindent
(R2) Every $2$-vertex takes $\FR1{13}$ via each path of length $2$ from a
$3$-vertex.
\medskip

\begin{figure}[h]
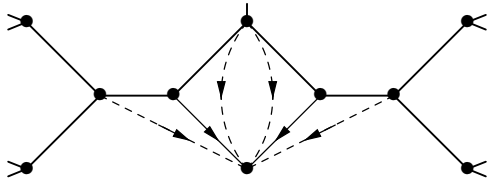

\gpic{
\expandafter\ifx\csname graph\endcsname\relax \csname newbox\endcsname\graph\fi
\expandafter\ifx\csname graphtemp\endcsname\relax \csname newdimen\endcsname\graphtemp\fi
\setbox\graph=\vtop{\vskip 0pt\hbox{%
    \graphtemp=.5ex\advance\graphtemp by 0.865in
    \rlap{\kern 1.250in\lower\graphtemp\hbox to 0pt{\hss $\bu$\hss}}%
    \graphtemp=.5ex\advance\graphtemp by 0.865in
    \rlap{\kern 0.096in\lower\graphtemp\hbox to 0pt{\hss $\bu$\hss}}%
    \graphtemp=.5ex\advance\graphtemp by 0.096in
    \rlap{\kern 0.096in\lower\graphtemp\hbox to 0pt{\hss $\bu$\hss}}%
    \graphtemp=.5ex\advance\graphtemp by 0.481in
    \rlap{\kern 0.481in\lower\graphtemp\hbox to 0pt{\hss $\bu$\hss}}%
    \graphtemp=.5ex\advance\graphtemp by 0.481in
    \rlap{\kern 0.865in\lower\graphtemp\hbox to 0pt{\hss $\bu$\hss}}%
    \graphtemp=.5ex\advance\graphtemp by 0.096in
    \rlap{\kern 1.250in\lower\graphtemp\hbox to 0pt{\hss $\bu$\hss}}%
    \graphtemp=.5ex\advance\graphtemp by 0.481in
    \rlap{\kern 1.635in\lower\graphtemp\hbox to 0pt{\hss $\bu$\hss}}%
    \graphtemp=.5ex\advance\graphtemp by 0.481in
    \rlap{\kern 2.019in\lower\graphtemp\hbox to 0pt{\hss $\bu$\hss}}%
    \graphtemp=.5ex\advance\graphtemp by 0.865in
    \rlap{\kern 2.404in\lower\graphtemp\hbox to 0pt{\hss $\bu$\hss}}%
    \graphtemp=.5ex\advance\graphtemp by 0.096in
    \rlap{\kern 2.404in\lower\graphtemp\hbox to 0pt{\hss $\bu$\hss}}%
    \special{pn 11}%
    \special{pa 481 481}%
    \special{pa 96 865}%
    \special{fp}%
    \special{pa 481 481}%
    \special{pa 96 96}%
    \special{fp}%
    \special{pa 481 481}%
    \special{pa 865 481}%
    \special{fp}%
    \special{pa 2019 481}%
    \special{pa 1635 481}%
    \special{fp}%
    \special{pa 2019 481}%
    \special{pa 2404 865}%
    \special{fp}%
    \special{pa 2019 481}%
    \special{pa 2404 96}%
    \special{fp}%
    \special{pa 1250 96}%
    \special{pa 865 481}%
    \special{fp}%
    \special{pa 1250 96}%
    \special{pa 1635 481}%
    \special{fp}%
    \special{pa 1250 96}%
    \special{pa 1250 0}%
    \special{fp}%
    \special{pa 2500 58}%
    \special{pa 2404 96}%
    \special{fp}%
    \special{pa 2404 96}%
    \special{pa 2500 135}%
    \special{fp}%
    \special{pa 2500 827}%
    \special{pa 2404 865}%
    \special{fp}%
    \special{pa 2404 865}%
    \special{pa 2500 904}%
    \special{fp}%
    \special{pa 0 827}%
    \special{pa 96 865}%
    \special{fp}%
    \special{pa 96 865}%
    \special{pa 0 904}%
    \special{fp}%
    \special{pa 0 58}%
    \special{pa 96 96}%
    \special{fp}%
    \special{pa 96 96}%
    \special{pa 0 135}%
    \special{fp}%
    \special{pn 8}%
    \special{pa 865 481}%
    \special{pa 1096 712}%
    \special{fp}%
    \special{sh 1.000}%
    \special{pa 1055 644}%
    \special{pa 1096 712}%
    \special{pa 1028 671}%
    \special{pa 1055 644}%
    \special{fp}%
    \special{pa 1096 712}%
    \special{pa 1250 865}%
    \special{fp}%
    \special{pa 1635 481}%
    \special{pa 1404 712}%
    \special{fp}%
    \special{sh 1.000}%
    \special{pa 1472 671}%
    \special{pa 1404 712}%
    \special{pa 1445 644}%
    \special{pa 1472 671}%
    \special{fp}%
    \special{pa 1404 712}%
    \special{pa 1250 865}%
    \special{fp}%
    \special{pa 481 481}%
    \special{pa 1250 865}%
    \special{da 0.038}%
    \special{pa 788 635}%
    \special{pa 942 712}%
    \special{fp}%
    \special{sh 1.000}%
    \special{pa 882 660}%
    \special{pa 942 712}%
    \special{pa 865 694}%
    \special{pa 882 660}%
    \special{fp}%
    \special{pa 2019 481}%
    \special{pa 1250 865}%
    \special{da 0.038}%
    \special{pa 1712 635}%
    \special{pa 1558 712}%
    \special{fp}%
    \special{sh 1.000}%
    \special{pa 1635 694}%
    \special{pa 1558 712}%
    \special{pa 1618 660}%
    \special{pa 1635 694}%
    \special{fp}%
    \special{ar 770 481 615 615 0.612632 0.675132}%
    \special{ar 770 481 615 615 0.483855 0.546355}%
    \special{ar 770 481 615 615 0.355079 0.417579}%
    \special{ar 770 481 615 615 0.226303 0.288803}%
    \special{ar 770 481 615 615 0.097526 0.160026}%
    \special{ar 770 481 615 615 -0.031250 0.031250}%
    \special{ar 770 481 615 615 -0.160026 -0.097526}%
    \special{ar 770 481 615 615 -0.288803 -0.226303}%
    \special{ar 770 481 615 615 -0.417579 -0.355079}%
    \special{ar 770 481 615 615 -0.546355 -0.483855}%
    \special{ar 770 481 615 615 -0.675132 -0.612632}%
    \special{ar 1730 481 615 615 -2.528961 -2.466461}%
    \special{ar 1730 481 615 615 -2.657737 -2.595237}%
    \special{ar 1730 481 615 615 -2.786514 -2.724014}%
    \special{ar 1730 481 615 615 -2.915290 -2.852790}%
    \special{ar 1730 481 615 615 -3.044066 -2.981566}%
    \special{ar 1730 481 615 615 -3.172843 -3.110343}%
    \special{ar 1730 481 615 615 -3.301619 -3.239119}%
    \special{ar 1730 481 615 615 -3.430395 -3.367895}%
    \special{ar 1730 481 615 615 -3.559172 -3.496672}%
    \special{ar 1730 481 615 615 -3.687948 -3.625448}%
    \special{ar 1730 481 615 615 -3.816724 -3.754224}%
    \special{pa 1115 462}%
    \special{pa 1115 485}%
    \special{fp}%
    \special{sh 1.000}%
    \special{pa 1135 408}%
    \special{pa 1115 485}%
    \special{pa 1096 408}%
    \special{pa 1135 408}%
    \special{fp}%
    \special{pa 1115 485}%
    \special{pa 1115 500}%
    \special{fp}%
    \special{pa 1385 462}%
    \special{pa 1385 485}%
    \special{fp}%
    \special{sh 1.000}%
    \special{pa 1404 408}%
    \special{pa 1385 485}%
    \special{pa 1365 408}%
    \special{pa 1404 408}%
    \special{fp}%
    \special{pa 1385 485}%
    \special{pa 1385 500}%
    \special{fp}%
    \hbox{\vrule depth0.913in width0pt height 0pt}%
    \kern 2.500in
  }%
}%
}

\vspace{-1pc}

\caption{Discharging rules for Theorem~\ref{36.13dis}; dashes move $\FR1{13}$
\label{fig3613}}
\end{figure}

Each $3$-vertex $v$ having a $2$-neighbor gives it $\FR3{13}$.  Since no two
$2$-vertices are adjacent, and adjacent $3$-vertices cannot both have
$2$-neighbors, $v$ loses no other charge.  Each $3$-vertex $w$ having no
$2$-neighbor loses at most $\FR1{13}$ along each incident edge, because its
$3$-neighbors do not have two $2$-neighbors.  (Under (R2), a $3$-vertex
opposite a $2$-vertex $x$ on a $4$-cycle gives $\FR2{13}$ to $x$.)  Thus every
$3$-vertex ends with charge at least $\FR{36}{13}$.

A $2$-vertex gains $\FR3{13}$ from each neighbor, and it also gains $\FR1{13}$
along each of the two other edges incident to each neighbor (see
Figure~\ref{fig3613}).  Hence it gains $\FR{10}{13}$ and reaches charge
$\FR{36}{13}$.  (With no adjacent $3$-vertices having $2$-neighbors, no
$2$-vertex lies on a triangle.)

We have shown that $\avd(G)\ge\FR{36}{13}$ when the specified configurations
do not occur.
\end{proof}

\begin{remark}
The proof of Theorem~\ref{36.13dis} allows every vertex to end with charge
exactly $\FR{36}{13}$.  This can happen, making the structure theorem sharp.
In fact, here also the coloring result is sharp.  Deleting one vertex from the
Heawood graph (the incidence graph of the Fano plane) yields a graph $H$ with
$\avd(H)=\FR{36}{13}$, $\Delta(H)=3$, and $\chi^i(G)=6$.  

The discharging rules in Theorem~\ref{36.13dis} follow naturally from the bound
on $\Mad(G)$ and the forbidden configurations, but how are those found?  To
discover the structure theorem, first study the coloring problem to find
reducible configurations.  A $1^-$-vertex and two adjacent $2$-vertices are
easy to show reducible.  With a bit more thought, a $3$-vertex with two
$2$-neighbors is reducible.  These configurations form an unavoidable set for
$\Mad(G)<\FR83$, using the discharging rule that each $2$-vertex takes $\FR13$
from each neighbor.  That yields the desired conclusion when $\Mad(G)<\FR83$,
but we can do better.

After adding the reducible configuration consisting of two adjacent
$3$-vertices having $2$-neighbors, we seek the loosest bound on $\Mad(G)$ under
which this larger set is unavoidable.  It will exceed $\FR83$.  The
$2$-vertices can take charge only from $3$-vertices, but when $\Mad(G)>\FR83$
their neighbors cannot afford to give enough to satisfy them.  When two
adjacent $3$-vertices with $2$-neighbors are forbidden, the $2$-vertices can
also gain charge along paths of length $2$.

Now we have the ``avenues'' of discharging.  Let each $2$-vertex take $a$
from each neighbor and $b$ along each path of length $2$.  Now $2$-vertices
end with $2+2a+4b$, $3$-vertices having $2$-neighbors end with $3-a$, and
$3$-vertices without $2$-neighbors end with as little as $3-3b$.  We seek $a$
and $b$ to maximize the minimum of $\{2+2a+4b,3-a,3-3b\}$.  If $3-a$ and $3-3b$
are not equal, then the value can be improved, so take $a=3b$.  Now
$\min\{2+10b,3-3b\}$ is maximized when $2+10b=3-3b$, or $b=\FR1{13}$.  Hence
the proof works when $\Mad(G)<\FR{36}{13}$ and fails for any larger bound (as
also implied by the sharpness example).
\end{remark}

\medskip
We also apply Lemma~\ref{12vert} to the problem of coloring the square of a
planar graph, where there is another well-known conjecture (the original 
conjecture was more general).

\begin{conjecture}[Wegner's Conjecture~\cite{Weg}]\label{wegner}
If $G$ is planar, then $\chi(G^2)\le \FL{\FR32\Delta(G)}+1$ for
$\Delta(G)\ge8$; also $\chi(G^2)\le \Delta(G)+5$ for $4\le \Delta(G)\le 7$
and $\chi(G^2)\le 7$ for $\Delta(G)\le3$.
\end{conjecture}

The case $\Delta(G)=3$ was recently proved by Hartke, Jahanbekam, and
Thomas~\cite{HJT} using discharging and computerized checking of reducibility.
Wegner gave sharpness constructions; fixing $\Delta(G)$, these are planar
graphs of diameter $2$ (so $\chi(G^2)=|V(G)|$) with the most vertices.  The
general situation uses graphs studied by Erd\H{o}s and R\'enyi~\cite{ER}, shown
on the left in Figure~\ref{figweg}.  The other graphs there with maximum degree
$k$ have diameter $2$ with $k+5$ vertices for $4\le k\le 7$, taken
from~\cite{HeS} (the half-edges in the graph for $k=6$ meet at the eleventh
vertex).  Wegner found the three leftmost graphs.

\begin{figure}[hbt]
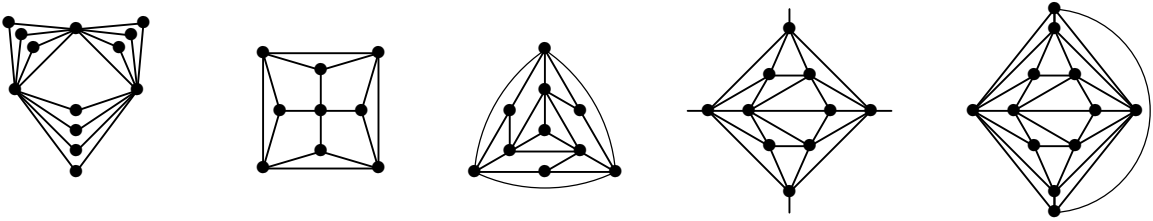

\gpic{
\expandafter\ifx\csname graph\endcsname\relax \csname newbox\endcsname\graph\fi
\expandafter\ifx\csname graphtemp\endcsname\relax \csname newdimen\endcsname\graphtemp\fi
\setbox\graph=\vtop{\vskip 0pt\hbox{%
    \graphtemp=.5ex\advance\graphtemp by 0.133in
    \rlap{\kern 0.379in\lower\graphtemp\hbox to 0pt{\hss $\bu$\hss}}%
    \graphtemp=.5ex\advance\graphtemp by 0.101in
    \rlap{\kern 0.027in\lower\graphtemp\hbox to 0pt{\hss $\bu$\hss}}%
    \graphtemp=.5ex\advance\graphtemp by 0.165in
    \rlap{\kern 0.091in\lower\graphtemp\hbox to 0pt{\hss $\bu$\hss}}%
    \graphtemp=.5ex\advance\graphtemp by 0.229in
    \rlap{\kern 0.155in\lower\graphtemp\hbox to 0pt{\hss $\bu$\hss}}%
    \graphtemp=.5ex\advance\graphtemp by 0.229in
    \rlap{\kern 0.603in\lower\graphtemp\hbox to 0pt{\hss $\bu$\hss}}%
    \graphtemp=.5ex\advance\graphtemp by 0.165in
    \rlap{\kern 0.667in\lower\graphtemp\hbox to 0pt{\hss $\bu$\hss}}%
    \graphtemp=.5ex\advance\graphtemp by 0.101in
    \rlap{\kern 0.731in\lower\graphtemp\hbox to 0pt{\hss $\bu$\hss}}%
    \graphtemp=.5ex\advance\graphtemp by 0.453in
    \rlap{\kern 0.059in\lower\graphtemp\hbox to 0pt{\hss $\bu$\hss}}%
    \graphtemp=.5ex\advance\graphtemp by 0.453in
    \rlap{\kern 0.699in\lower\graphtemp\hbox to 0pt{\hss $\bu$\hss}}%
    \special{pn 11}%
    \special{pa 379 133}%
    \special{pa 27 101}%
    \special{fp}%
    \special{pa 379 133}%
    \special{pa 91 165}%
    \special{fp}%
    \special{pa 379 133}%
    \special{pa 155 229}%
    \special{fp}%
    \special{pa 379 133}%
    \special{pa 603 229}%
    \special{fp}%
    \special{pa 379 133}%
    \special{pa 667 165}%
    \special{fp}%
    \special{pa 379 133}%
    \special{pa 731 101}%
    \special{fp}%
    \special{pa 379 133}%
    \special{pa 59 453}%
    \special{fp}%
    \special{pa 379 133}%
    \special{pa 699 453}%
    \special{fp}%
    \special{pa 59 453}%
    \special{pa 27 101}%
    \special{fp}%
    \special{pa 59 453}%
    \special{pa 91 165}%
    \special{fp}%
    \special{pa 59 453}%
    \special{pa 155 229}%
    \special{fp}%
    \special{pa 699 453}%
    \special{pa 603 229}%
    \special{fp}%
    \special{pa 699 453}%
    \special{pa 667 165}%
    \special{fp}%
    \special{pa 699 453}%
    \special{pa 731 101}%
    \special{fp}%
    \graphtemp=.5ex\advance\graphtemp by 0.773in
    \rlap{\kern 0.379in\lower\graphtemp\hbox to 0pt{\hss $\bu$\hss}}%
    \graphtemp=.5ex\advance\graphtemp by 0.667in
    \rlap{\kern 0.379in\lower\graphtemp\hbox to 0pt{\hss $\bu$\hss}}%
    \graphtemp=.5ex\advance\graphtemp by 0.560in
    \rlap{\kern 0.379in\lower\graphtemp\hbox to 0pt{\hss $\bu$\hss}}%
    \graphtemp=.5ex\advance\graphtemp by 0.880in
    \rlap{\kern 0.379in\lower\graphtemp\hbox to 0pt{\hss $\bu$\hss}}%
    \special{pa 59 453}%
    \special{pa 379 773}%
    \special{fp}%
    \special{pa 59 453}%
    \special{pa 379 667}%
    \special{fp}%
    \special{pa 59 453}%
    \special{pa 379 560}%
    \special{fp}%
    \special{pa 59 453}%
    \special{pa 379 880}%
    \special{fp}%
    \special{pa 699 453}%
    \special{pa 379 773}%
    \special{fp}%
    \special{pa 699 453}%
    \special{pa 379 667}%
    \special{fp}%
    \special{pa 699 453}%
    \special{pa 379 560}%
    \special{fp}%
    \special{pa 699 453}%
    \special{pa 379 880}%
    \special{fp}%
    \graphtemp=.5ex\advance\graphtemp by 0.258in
    \rlap{\kern 1.961in\lower\graphtemp\hbox to 0pt{\hss $\bu$\hss}}%
    \graphtemp=.5ex\advance\graphtemp by 0.862in
    \rlap{\kern 1.961in\lower\graphtemp\hbox to 0pt{\hss $\bu$\hss}}%
    \graphtemp=.5ex\advance\graphtemp by 0.862in
    \rlap{\kern 1.357in\lower\graphtemp\hbox to 0pt{\hss $\bu$\hss}}%
    \graphtemp=.5ex\advance\graphtemp by 0.258in
    \rlap{\kern 1.357in\lower\graphtemp\hbox to 0pt{\hss $\bu$\hss}}%
    \special{pa 1961 258}%
    \special{pa 1961 862}%
    \special{fp}%
    \special{pa 1961 862}%
    \special{pa 1357 862}%
    \special{fp}%
    \special{pa 1357 862}%
    \special{pa 1357 258}%
    \special{fp}%
    \special{pa 1357 258}%
    \special{pa 1961 258}%
    \special{fp}%
    \graphtemp=.5ex\advance\graphtemp by 0.347in
    \rlap{\kern 1.659in\lower\graphtemp\hbox to 0pt{\hss $\bu$\hss}}%
    \graphtemp=.5ex\advance\graphtemp by 0.560in
    \rlap{\kern 1.872in\lower\graphtemp\hbox to 0pt{\hss $\bu$\hss}}%
    \graphtemp=.5ex\advance\graphtemp by 0.773in
    \rlap{\kern 1.659in\lower\graphtemp\hbox to 0pt{\hss $\bu$\hss}}%
    \graphtemp=.5ex\advance\graphtemp by 0.560in
    \rlap{\kern 1.445in\lower\graphtemp\hbox to 0pt{\hss $\bu$\hss}}%
    \special{pa 1659 560}%
    \special{pa 1659 347}%
    \special{fp}%
    \special{pa 1659 560}%
    \special{pa 1872 560}%
    \special{fp}%
    \special{pa 1659 560}%
    \special{pa 1659 773}%
    \special{fp}%
    \special{pa 1659 560}%
    \special{pa 1445 560}%
    \special{fp}%
    \graphtemp=.5ex\advance\graphtemp by 0.560in
    \rlap{\kern 1.659in\lower\graphtemp\hbox to 0pt{\hss $\bu$\hss}}%
    \special{pa 1659 347}%
    \special{pa 1961 258}%
    \special{fp}%
    \special{pa 1961 258}%
    \special{pa 1872 560}%
    \special{fp}%
    \special{pa 1872 560}%
    \special{pa 1961 862}%
    \special{fp}%
    \special{pa 1961 862}%
    \special{pa 1659 773}%
    \special{fp}%
    \special{pa 1659 773}%
    \special{pa 1357 862}%
    \special{fp}%
    \special{pa 1357 862}%
    \special{pa 1445 560}%
    \special{fp}%
    \special{pa 1445 560}%
    \special{pa 1357 258}%
    \special{fp}%
    \special{pa 1357 258}%
    \special{pa 1659 347}%
    \special{fp}%
    \graphtemp=.5ex\advance\graphtemp by 0.667in
    \rlap{\kern 2.832in\lower\graphtemp\hbox to 0pt{\hss $\bu$\hss}}%
    \graphtemp=.5ex\advance\graphtemp by 0.453in
    \rlap{\kern 2.832in\lower\graphtemp\hbox to 0pt{\hss $\bu$\hss}}%
    \graphtemp=.5ex\advance\graphtemp by 0.560in
    \rlap{\kern 3.017in\lower\graphtemp\hbox to 0pt{\hss $\bu$\hss}}%
    \graphtemp=.5ex\advance\graphtemp by 0.773in
    \rlap{\kern 3.017in\lower\graphtemp\hbox to 0pt{\hss $\bu$\hss}}%
    \graphtemp=.5ex\advance\graphtemp by 0.880in
    \rlap{\kern 2.832in\lower\graphtemp\hbox to 0pt{\hss $\bu$\hss}}%
    \graphtemp=.5ex\advance\graphtemp by 0.773in
    \rlap{\kern 2.648in\lower\graphtemp\hbox to 0pt{\hss $\bu$\hss}}%
    \graphtemp=.5ex\advance\graphtemp by 0.560in
    \rlap{\kern 2.648in\lower\graphtemp\hbox to 0pt{\hss $\bu$\hss}}%
    \special{pa 2832 453}%
    \special{pa 3017 773}%
    \special{fp}%
    \special{pa 3017 773}%
    \special{pa 2648 773}%
    \special{fp}%
    \special{pa 2648 773}%
    \special{pa 2832 453}%
    \special{fp}%
    \special{pa 2832 453}%
    \special{pa 3017 560}%
    \special{fp}%
    \special{pa 3017 773}%
    \special{pa 2832 880}%
    \special{fp}%
    \special{pa 2648 773}%
    \special{pa 2648 560}%
    \special{fp}%
    \graphtemp=.5ex\advance\graphtemp by 0.240in
    \rlap{\kern 2.832in\lower\graphtemp\hbox to 0pt{\hss $\bu$\hss}}%
    \graphtemp=.5ex\advance\graphtemp by 0.880in
    \rlap{\kern 3.202in\lower\graphtemp\hbox to 0pt{\hss $\bu$\hss}}%
    \graphtemp=.5ex\advance\graphtemp by 0.880in
    \rlap{\kern 2.463in\lower\graphtemp\hbox to 0pt{\hss $\bu$\hss}}%
    \special{pa 2832 667}%
    \special{pa 2832 240}%
    \special{fp}%
    \special{pa 2832 240}%
    \special{pa 3202 880}%
    \special{fp}%
    \special{pa 3202 880}%
    \special{pa 2463 880}%
    \special{fp}%
    \special{pa 2832 240}%
    \special{pa 2463 880}%
    \special{fp}%
    \special{pa 2463 880}%
    \special{pa 2832 667}%
    \special{fp}%
    \special{pa 2832 667}%
    \special{pa 3202 880}%
    \special{fp}%
    \special{pn 8}%
    \special{ar 3314 945 853 853 -3.065835 -2.170178}%
    \special{ar 2832 111 853 853 1.122978 2.018615}%
    \special{ar 2351 945 853 853 -0.971415 -0.075757}%
    \graphtemp=.5ex\advance\graphtemp by 0.375in
    \rlap{\kern 4.219in\lower\graphtemp\hbox to 0pt{\hss $\bu$\hss}}%
    \graphtemp=.5ex\advance\graphtemp by 0.560in
    \rlap{\kern 4.326in\lower\graphtemp\hbox to 0pt{\hss $\bu$\hss}}%
    \graphtemp=.5ex\advance\graphtemp by 0.745in
    \rlap{\kern 4.219in\lower\graphtemp\hbox to 0pt{\hss $\bu$\hss}}%
    \graphtemp=.5ex\advance\graphtemp by 0.745in
    \rlap{\kern 4.006in\lower\graphtemp\hbox to 0pt{\hss $\bu$\hss}}%
    \graphtemp=.5ex\advance\graphtemp by 0.560in
    \rlap{\kern 3.899in\lower\graphtemp\hbox to 0pt{\hss $\bu$\hss}}%
    \graphtemp=.5ex\advance\graphtemp by 0.375in
    \rlap{\kern 4.006in\lower\graphtemp\hbox to 0pt{\hss $\bu$\hss}}%
    \special{pn 11}%
    \special{pa 4219 375}%
    \special{pa 4326 560}%
    \special{fp}%
    \special{pa 4326 560}%
    \special{pa 4219 745}%
    \special{fp}%
    \special{pa 4219 745}%
    \special{pa 4006 745}%
    \special{fp}%
    \special{pa 4006 745}%
    \special{pa 3899 560}%
    \special{fp}%
    \special{pa 3899 560}%
    \special{pa 4006 375}%
    \special{fp}%
    \special{pa 4006 375}%
    \special{pa 4219 375}%
    \special{fp}%
    \special{pa 4219 375}%
    \special{pa 3899 560}%
    \special{fp}%
    \special{pa 3899 560}%
    \special{pa 4219 745}%
    \special{fp}%
    \graphtemp=.5ex\advance\graphtemp by 0.133in
    \rlap{\kern 4.112in\lower\graphtemp\hbox to 0pt{\hss $\bu$\hss}}%
    \graphtemp=.5ex\advance\graphtemp by 0.560in
    \rlap{\kern 4.539in\lower\graphtemp\hbox to 0pt{\hss $\bu$\hss}}%
    \graphtemp=.5ex\advance\graphtemp by 0.987in
    \rlap{\kern 4.112in\lower\graphtemp\hbox to 0pt{\hss $\bu$\hss}}%
    \graphtemp=.5ex\advance\graphtemp by 0.560in
    \rlap{\kern 3.686in\lower\graphtemp\hbox to 0pt{\hss $\bu$\hss}}%
    \special{pa 4112 133}%
    \special{pa 4539 560}%
    \special{fp}%
    \special{pa 4539 560}%
    \special{pa 4112 987}%
    \special{fp}%
    \special{pa 4112 987}%
    \special{pa 3686 560}%
    \special{fp}%
    \special{pa 3686 560}%
    \special{pa 4112 133}%
    \special{fp}%
    \special{pa 4112 133}%
    \special{pa 4219 375}%
    \special{fp}%
    \special{pa 4219 375}%
    \special{pa 4539 560}%
    \special{fp}%
    \special{pa 4539 560}%
    \special{pa 4219 745}%
    \special{fp}%
    \special{pa 4219 745}%
    \special{pa 4112 987}%
    \special{fp}%
    \special{pa 4112 987}%
    \special{pa 4006 745}%
    \special{fp}%
    \special{pa 4006 745}%
    \special{pa 3686 560}%
    \special{fp}%
    \special{pa 3686 560}%
    \special{pa 4006 375}%
    \special{fp}%
    \special{pa 4006 375}%
    \special{pa 4112 133}%
    \special{fp}%
    \special{pa 4112 133}%
    \special{pa 4112 27}%
    \special{fp}%
    \special{pa 4646 560}%
    \special{pa 3579 560}%
    \special{fp}%
    \special{pa 4112 987}%
    \special{pa 4112 1093}%
    \special{fp}%
    \graphtemp=.5ex\advance\graphtemp by 0.375in
    \rlap{\kern 5.606in\lower\graphtemp\hbox to 0pt{\hss $\bu$\hss}}%
    \graphtemp=.5ex\advance\graphtemp by 0.560in
    \rlap{\kern 5.713in\lower\graphtemp\hbox to 0pt{\hss $\bu$\hss}}%
    \graphtemp=.5ex\advance\graphtemp by 0.745in
    \rlap{\kern 5.606in\lower\graphtemp\hbox to 0pt{\hss $\bu$\hss}}%
    \graphtemp=.5ex\advance\graphtemp by 0.745in
    \rlap{\kern 5.393in\lower\graphtemp\hbox to 0pt{\hss $\bu$\hss}}%
    \graphtemp=.5ex\advance\graphtemp by 0.560in
    \rlap{\kern 5.286in\lower\graphtemp\hbox to 0pt{\hss $\bu$\hss}}%
    \graphtemp=.5ex\advance\graphtemp by 0.375in
    \rlap{\kern 5.393in\lower\graphtemp\hbox to 0pt{\hss $\bu$\hss}}%
    \special{pa 5606 375}%
    \special{pa 5713 560}%
    \special{fp}%
    \special{pa 5713 560}%
    \special{pa 5606 745}%
    \special{fp}%
    \special{pa 5606 745}%
    \special{pa 5393 745}%
    \special{fp}%
    \special{pa 5393 745}%
    \special{pa 5286 560}%
    \special{fp}%
    \special{pa 5286 560}%
    \special{pa 5393 375}%
    \special{fp}%
    \special{pa 5393 375}%
    \special{pa 5606 375}%
    \special{fp}%
    \special{pa 5606 375}%
    \special{pa 5286 560}%
    \special{fp}%
    \special{pa 5286 560}%
    \special{pa 5606 745}%
    \special{fp}%
    \graphtemp=.5ex\advance\graphtemp by 0.133in
    \rlap{\kern 5.499in\lower\graphtemp\hbox to 0pt{\hss $\bu$\hss}}%
    \graphtemp=.5ex\advance\graphtemp by 0.560in
    \rlap{\kern 5.926in\lower\graphtemp\hbox to 0pt{\hss $\bu$\hss}}%
    \graphtemp=.5ex\advance\graphtemp by 0.987in
    \rlap{\kern 5.499in\lower\graphtemp\hbox to 0pt{\hss $\bu$\hss}}%
    \graphtemp=.5ex\advance\graphtemp by 0.560in
    \rlap{\kern 5.073in\lower\graphtemp\hbox to 0pt{\hss $\bu$\hss}}%
    \special{pa 5499 133}%
    \special{pa 5926 560}%
    \special{fp}%
    \special{pa 5926 560}%
    \special{pa 5499 987}%
    \special{fp}%
    \special{pa 5499 987}%
    \special{pa 5073 560}%
    \special{fp}%
    \special{pa 5073 560}%
    \special{pa 5499 133}%
    \special{fp}%
    \special{pa 5499 133}%
    \special{pa 5606 375}%
    \special{fp}%
    \special{pa 5606 375}%
    \special{pa 5926 560}%
    \special{fp}%
    \special{pa 5926 560}%
    \special{pa 5606 745}%
    \special{fp}%
    \special{pa 5606 745}%
    \special{pa 5499 987}%
    \special{fp}%
    \special{pa 5499 987}%
    \special{pa 5393 745}%
    \special{fp}%
    \special{pa 5393 745}%
    \special{pa 5073 560}%
    \special{fp}%
    \special{pa 5073 560}%
    \special{pa 5393 375}%
    \special{fp}%
    \special{pa 5393 375}%
    \special{pa 5499 133}%
    \special{fp}%
    \special{pa 5499 133}%
    \special{pa 5499 27}%
    \special{fp}%
    \special{pa 5926 560}%
    \special{pa 5073 560}%
    \special{fp}%
    \special{pa 5499 987}%
    \special{pa 5499 1093}%
    \special{fp}%
    \graphtemp=.5ex\advance\graphtemp by 0.027in
    \rlap{\kern 5.499in\lower\graphtemp\hbox to 0pt{\hss $\bu$\hss}}%
    \graphtemp=.5ex\advance\graphtemp by 1.093in
    \rlap{\kern 5.499in\lower\graphtemp\hbox to 0pt{\hss $\bu$\hss}}%
    \special{pa 5499 27}%
    \special{pa 5073 560}%
    \special{fp}%
    \special{pa 5499 27}%
    \special{pa 5499 133}%
    \special{fp}%
    \special{pa 5499 27}%
    \special{pa 5926 560}%
    \special{fp}%
    \special{pa 5499 1093}%
    \special{pa 5073 560}%
    \special{fp}%
    \special{pa 5499 1093}%
    \special{pa 5499 987}%
    \special{fp}%
    \special{pa 5499 1093}%
    \special{pa 5926 560}%
    \special{fp}%
    \special{pn 8}%
    \special{ar 5466 560 534 534 -1.507603 1.507603}%
    \hbox{\vrule depth1.120in width0pt height 0pt}%
    \kern 6.000in
  }%
}%
}
\caption{Large graphs with diameter $2$ and fixed maximum degree\label{figweg}}
\end{figure}

For upper bounds, van den Heuvel and McGuinness~\cite{HM} proved
$\chi(G^2)\le 2\Delta(G)+25$, but their argument also yields
$\chil(G^2)\le 2\Delta(G)+25$ (we list tighter bounds later).  We present a
weaker version of this, showing
that $\chil(G^2)\le 2\Delta(G)+34$ when $G$ is planar.  To keep our additive
constant small, we use an enhanced version of Lemma~\ref{kotzig}.  Instead of
letting each $5^-$-vertex $v$ take $\FR{6-d(v)}{d(v)}$ from each
$7^+$-neighbor, change the rule for $5$-vertices to take $\FR14$ from each
$7^+$-neighbor.  Now a $5$-vertex becomes happy when it has at least four
$7^+$-neighbors.  With $4^-$-neighbors forbidden, a $7$-vertex lost at most
$\FR35$ before, now at most $\FR34$, so again it remains happy.  Hence we
conclude the following.

\begin{lemma}\label{kotzig566}
Every normal plane map $G$ has a $3$-vertex with a $10^-$-neighbor, or a
$4$-vertex with a $7^-$-neighbor, or a $5$-vertex with two $6^-$-neighbors.
\end{lemma}

Lemma~\ref{kotzig566} strengthens an early result of Franklin~\cite{F}: if $G$
is planar with minimum degree $5$, then $G$ has a $5$-vertex with a
$5$-neighbor or two $6$-neighbors.

Our first bound on $\chil(G^2)$ helps when $\Delta(G)$ is small.
It slightly refines an idea from~\cite{Jo}.

\begin{thm}\label{sq7k}
If $G$ is a planar graph, then
$\chil(G^2)\le\begin{cases}
 \Delta(G)^2+1& \textrm{when }\Delta(G)\le5,\\
 7\Delta(G)-7& \textrm{when }\Delta(G)\ge6.  \end{cases} $
\end{thm}
\begin{proof}
For $\Delta(G)\le5$, the claim is just the trivial upper bound from
$\Delta(G^2)$, so we may assume $\Delta(G)\ge6$.

\nobreak
Index the vertices from $v_n$ to $v_1$ as follows.  Having chosen 
$\VEC vn{i+1}$, let $G_i=G-\{\VEC vn{i+1}\}$.  If $\delta(G_i)\le 3$,
then let $v_i$ be a vertex of minimum degree; otherwise let $v$ be a vertex
as guaranteed by Lemma~\ref{kotzig566}.  Let $S_i=\{\VEC v1i\}$.  We choose
colors for vertices in the order $\VEC v1n$ so that the coloring of $S_i$
satisfies all the constraints in the full graph $G^2$ from pairs of vertices in
$S_i$.  

Let $j=|N(v_i)\cap S_i|$; by the choice of the ordering, $j\le 5$.  The other
neighbors of $v_i$ occur later and are not yet colored.  However, $v_i$ must
avoid the colors on the neighbors in $S_i$ of these vertices; there may be up
to $4(\Delta(G)-j)$ such colors.  For $u\in N(v_i)\cap S_i$, all neighbors of
$u$ may lie in $S_i$, so there may be as many as $d(u)$ colors that $v_i$ must
avoid due to $u$.

If $j\le 3$, then the number of colors $v_i$ must avoid is at most
$4(\Delta(G)-j)+j\Delta(G)$.  If $j=4$, then the bound is $7\Delta(G)-9$, since
$v_i$ has a $7^-$-neighbor in $S_i$.  If $j=5$, then it is $7\Delta(G)-8$,
since $v_i$ has two $6^-$-neighbors in $S_i$.  Hence always the bound is at
most $7\Delta(G)-8$.
\end{proof}

Using Lemma~\ref{12vert} instead of Lemma~\ref{kotzig566}, a vertex with $j$
earlier neighbors must avoid at most $6\Delta(G)+7j-22$ colors.  Hence
$\chil(G^2)\le 6\Delta(G)+14$.  We strengthen this bound.

\begin{thm}\label{sqdelta}
If $G$ is a planar graph, then $\chil(G^2)\le 2\Delta(G)+34$.
\end{thm}

\begin{proof}
Theorem~\ref{sq7k} provides upper bounds that are at most $2\Delta(G)+34$ when
$\Delta(G)\le 8$.  Hence we may assume $\Delta(G)\ge 9$.

Let $G$ be a minimal counterexample, with list assignment $L$ from which no
such coloring can be chosen.  We will form $G'$ by contracting an edge
of $G$ incident to a $5^-$-vertex $v$, viewed as absorbing $v$ into the other
endpoint $u$; the new vertex retains the list assigned to $u$.  All the
constraints forcing vertices of $G$ to have distinct colors are present also in
$G'$, so any proper coloring of the square of $G'$ can also be used on $V(G')$
in $G$.  If $\Delta(G')$ is small enough, then the induction hypothesis applies
to properly color the square of $G'$ from lists of the desired size, and the
task is only to show that few enough other vertices are within distance $2$ of
$v$ in $G$, leaving a color available for $v$ to complete an $L$-coloring of
$G^2$.

\nobreak
If $\delta(G)\le 2$, then let $v$ be a vertex of minimum degree.  The
contracted vertex has degree at most $\Delta(G)$.  At most $2\Delta(G)$
vertices are within distance $2$ of $v$, which leaves a color available for
$v$.  Hence we may assume $\delta(G)\ge3$,
so Lemmas~\ref{kotzig566} and \ref{12vert} apply.

{\it Case 1: $9\le \Delta(G)\le 13$.}
In this case we prove $\chil(G^2)\le 52\le 2\Delta(G)+34$.  Let $v$ be a vertex
as guaranteed by Lemma~\ref{kotzig566}: a $3$-vertex with a $10^-$-neighbor
$u$, a $4$-vertex with a $7^-$-neighbor $u$, or a $5$-vertex with
$6^-$-neighbors $u$ and $u'$.  Contract the edge $uv$ into a new vertex;
it has degree at most $11$.  Thus $\Delta(G')\le 13$, and the induction
hypothesis yields a proper $L$-coloring of the square of $G'$ from lists of
size $52$ (if $\Delta(G')=8$, then $\chil(G'^2)\le 2\Delta(G')+34\le52$).  The
bound on the number of other vertices within distance $2$ of $v$ is
$2\Delta(G)+10$ for $d(v)=3$, $3\Delta(G)+7$ for $d(v)=4$, and $3\Delta(G)+12$
for $d(v)=5$.  Since $\Delta(G)\le13$, the value in each case is at most $51$.

{\it Case 2: $\Delta(G)\ge 14$.}
Let $v$ be a vertex as guaranteed by Lemma~\ref{12vert}; note that $d(v)\le 5$.
Since $d(v)\ge3$ and $v$ has at most two $12^+$-neighbors, $v$ has an
$11^-$-neighbor $u$; contract the edge $uv$.  The contracted vertex has degree
at most $14$, so the induction hypothesis applies.  Also, the number of
vertices within distance $2$ of $v$ in $G$ is bounded by $2\Delta(G)+33$.
\end{proof}

The improved upper bound of $2\Delta(G)+25$ in~\cite{HM} uses a slightly
stronger version of Lemma~\ref{12vert} to improve the argument for large
$\Delta(G)$ (see Exercise~\ref{HMex}).  The main additional work was proving
a second lemma specifically for graphs with small maximum degree.

Havet, van den Heuvel, McDiarmid, and Reed~\cite{HHMR} proved for planar graphs
that $\chil(G^2)\le (\FR32+o(1))\Delta(G)$ as $\Delta(G)\to\infty$, by
probabilistic methods.
Hence we also seek bounds below $2\Delta(G)$ when $\Delta(G)$ is ``small''.
Borodin et al.~\cite{BBGH2} proved for planar graphs that $\chil(G^2) \le 59$
when $\Delta(G)\le20$ and
$\chil(G^2)\le\max\{\Delta(G)+39,\CL{\FR95\Delta(G)}+1\}$ when
$\Delta(G)>20$.  In particular, if $\Delta(G)\ge47$, then
$\chil(G^2)\le\CL{\FR95\Delta(G)}+1$ (\cite{AH} proved this bound for
$\Delta(G)\ge750$).
Also,~\cite{BBGH2} proved that $G^2$ is $k$-degenerate, where
$k =\max\{\Delta(G)+38,\CL{\FR95\Delta(G)}\}$.
For the coloring problem alone, Molloy and Salavatipour~\cite{MS} proved 
$\chi(G^2)\le \FR53\Delta(G)+78$ for all planar $G$.
We explore results in terms of $\mad(G)$ (without planarity) in the exercises.

\medskip
{\small


\begin{exercise}
(Cranston--Kim~\cite{CK})
Apply Exercise~\ref{ex:square} to prove that if $\Delta(G)\le 3$ and
$\mad(G)\le \FR{14}5$, then $\chil(G^2)\le 7$.
\end{exercise}


\begin{exercise}\label{dynamick}
(Kim--Park~\cite{KP})
Prove that if $\delta(G)\ge 2$ and $\avd(G)< \FR{4k}{k+2}$ with $k\ge4$, then
$G$ has a $3^-$-vertex with a $(k-1)^-$-neighbor.  Guarantee a $2$-vertex with
a $(k-1)^-$-neighbor when $k\le6$.  Conclude that if $\mad(G)<\FR{4k}{k+2}$ with
$k\ge4$ (and no components are $5$-cycles if $k=4$), then from any lists of
size at least $k$ a proper coloring of $G$ can be chosen so that every vertex
with degree at least $2$ has neighbors with distinct colors.  Show also that
this is sharp: there exists $G$ with $\mad(G)=\FR{4k}{k+2}$ and an assignment
of $k$-lists from which no such coloring can be chosen.
\end{exercise}

\begin{exercise}
In Problem~\ref{G2prob}, prove that $b_{1,k}\ge2$.  Show that equality holds
when $k\in\{2,3\}$.
\end{exercise}

\begin{exercise}
(Cranston--Erman--\v{S}krekovski~\cite{CES})
Prove that a cycle of length divisible by $3$ with vertices whose degrees
cycle repeatedly through $2,2,3$ is reducible for $5$-choosability of $G^2$.
Use discharging to conclude that if $\Delta(G)\le4$ and $\mad(G)<16/7$.
then $\chil(G^2)\le 5$.
\end{exercise}

\begin{exercise}
(Cranston--Erman--\v{S}krekovski~\cite{CES})
Prove that if $\Delta(G)\le 4$ and $\avd(G)< \FR{18}7$, then $G$ contains one
of: (C1) a $1^-$-vertex, (C2) two adjacent $2$-vertices, (C3) a $3$-vertex with
three $2$-neighbors, or (C4) a four-vertex path alternating between
$2$-vertices and $3$-vertices.  Conclude that if $\Delta(G)\le 4$ and
$\mad(G)< \FR{18}7$, then $\chil(G^2)\le 7$.
\end{exercise}

\begin{exercise}
(Cranston--Erman--\v{S}krekovski~\cite{CES})
Prove that if $\Delta(G)\le 4$ and $\avd(G)\le \FR{10}3$, then $G$ contains one
of: (C1) a $1^-$-vertex, (C2) a $2$-vertex with a $3^-$-neighbor, (C3) a
$3$-vertex with two $3$-neighbors, or (C4) a $4$-vertex with a $2$-neighbor and
a $3^-$-neighbor.  Construct infinitely many graphs with average degree
$\FR{10}3$ and maximum degree $4$ that contain no such configuration.  Prove
that if $\Delta(G)\le 4$ and $\Mad(G)<\FR{10}3$, then $\chi_l(G^2)\le 12$.
\end{exercise}

\begin{exercise}\label{36.13red}
(Cranston--Kim--Yu~\cite{CKY2})
Complete the proof of Theorem~\ref{36.13dis} by showing that those
configurations are reducible for $\chi^i(G)\le5$ in the family of graphs with
$\Delta(G)\le3$.
\end{exercise}

\begin{exercise}
(Cranston--Kim--Yu~\cite{CKY2})
Prove that if $\avd(G)<\FR{14}5$ and $\Delta(G)\ge6$, then $G$ contains one of
the following configurations: (C1) a $1^-$-vertex, (C2) adjacent $2$-vertices,
(C3) a $3$-vertex with neighbors of degrees $2,a,b$, where $a+b\le\Delta(G)+2$,
or (C4) a $4$-vertex having four $2$-neighbors, one of which has other neighbor
of degree less than $\Delta(G)$.  Argue that none of these configurations can
appear in a minimal graph $G$ such that $\Delta(G)\ge 6$ and
$\chi^i(G)>\Delta(G)+2$.  Reducibility of the first two configurations and
part of (C3) is already requested in Exercise~\ref{36.13red}.
\end{exercise}


\begin{exercise}\label{HMex}
(van den Heuvel--McGuinness~\cite{HM})
Prove that every planar graph $G$ with $\delta(G)\ge3$ has a $5^-$-vertex $v$
with at most two $12^+$-neighbors such that $v$ has a $7^-$-neighbor if
$d(v)\in\{4,5\}$, and $v$ has an additional $6^-$-neighbor if $d(v)=5$.
Use this to prove that $\chi(G^2)\le 2\Delta(G)+25$ when $\Delta(G)\ge12$.
(Hint: Extend the proof of Lemma~\ref{12vert} by allowing $5^-$-vertices to
take some charge from their $11^-$-neighbors.)
\end{exercise}

}

\section{Edge-coloring and List Edge-coloring}\label{seclistedge}

We have mentioned the famous result of Vizing~\cite{V0,V2} and Gupta~\cite{G}
known as {\it Vizing's Theorem}.  It gives an upper bound for $\chi'(G)$ when
$G$ is a multigraph (allowing multiedges) and specializes to
$\chi'(G)\le \Delta(G)+1$ when $G$ is a graph.  Deciding whether $\chi'(G)$
equals $\Delta(G)$ or $\Delta(G)+1$ is NP-complete~\cite{Hol}, so we seek
sufficient conditions for equality.

\begin{conj}[Vizing's Planar Graph Conjecture~\cite{V1,V3}]\label{vizconj}
\label{planar-graph-conj}
If $G$ is a planar graph and $\Delta(G)\ge6$, then $\chi'(G) = \Delta(G)$.
\end{conj}

Both conditions in Vizing's Conjecture are needed.  The complete graph $K_7$ is 
$6$-regular but not planar.  Each color can be used on at most three edges,
so $\chi'(K_7)\ge \FR{21}3=7$.  Similarly, obtain $G$ from a $5$-regular planar
graph with $2k$ vertices by subdividing one edge.  Since $G$ has $5k+1$ edges,
and at most $k$ edges can receive the same color, $\chi'(G)\ge6$.  This
difficulty does not arise for $\Delta(G)\ge6$, because regular planar graphs
have degree at most $5$.

Vizing~\cite{V1} proved Conjecture~\ref{vizconj} for $\Delta(G)\ge8$, using
Vizing's Adjacency Lemma (VAL).  It is common to say that $G$ is \emph{Class 1}
if $\chi'(G)=\Delta(G)$, \emph{Class 2} otherwise.  An
{\it edge-critical graph} $G$ is then a Class 2 graph such that
$\chi'(G-e)=\Delta(G)$ for all $e\in E(G)$.  In fact, VAL implies that
every edge-critical graph has at least three vertices of maximum degree,
so $\Delta(G)=\Delta(G-e)$.  Note also that every Class 2 graph contains an
edge-critical graph with the same maximum degree.

\begin{theorem}[Vizing's Adjacency Lemma \cite{V1}]\label{VAL}
If $x$ and $y$ are adjacent in an edge-critical graph $G$, then at least
$\max\{1+\Delta(G)-d(y),2\}$ neighbors of $x$ have degree $\Delta(G)$.
\end{theorem}

Using VAL, Vizing proved the conjecture for $\Delta(G)\geq 8$ via counting
arguments about vertices of various degrees.  The proof is clearer in the
language of discharging, which was not then in use.  Luo and Zhang~\cite{LZ}
used VAL and discharging to prove $\chi'(G)=\Delta(G)$ for the larger family of
graphs $G$ with $\mad(G)\le6$ and $\Delta(G)\ge8$.  We present a slightly
simpler proof of a slightly weaker result, requiring $\mad(G)<6$.  In fact,
Miao and Sun~\cite{LQ} proved $\chi'(G)=\Delta(G)$ also when $\Delta(G)\ge8$
and $\mad(G)<\FR{13}2$.  Their result (and that of~\cite{LZ}) uses additional
adjacency lemmas.  Here VAL is used instead of reducibility arguments.

\begin{thm}[\cite{LZ}]\label{vizingplanar}
If $G$ is a graph with $\mad(G)<6$ and $\Delta(G)\ge 8$, then
$\chi'(G)=\Delta(G)$.
\end{thm}

\begin{proof}
Let $G$ be a minimal counterexample, and let $k=\Delta(G)$.  Since $\chi'(G)>k$
requires an edge-critical subgraph with the same maximum degree, we may assume
that $G$ is edge-critical.  By VAL, each vertex has at least two $k$-neighbors,
so $\delta(G)\geq 2$.  We use discharging with initial charge
$d(v)$; it suffices to show that each vertex ends with charge at least $6$.

\medskip\noindent
(R1) If $d(v)\le 4$, then $v$ takes $\FR{6-d(v)}{d(v)}$ from each neighbor.

\noindent
(R2) If $d(v)\in\{5,6\}$, then $v$ takes $\FR14$ from each $6^+$-neighbor.
\medskip

For $v\in V(G)$, let $j$ be the least degree among vertices in $N_G(v)$. 
If $j<k$, then $v$ has at least $k+1-j$ neighbors of degree $k$, by VAL.
Hence $k+1-j\le d(v)-1$, which yields $j\ge 10-d(v)$ since $k\ge8$.  
Note that $7^+$-vertices take no charge.

If $d(v)\le 4$, then $j\ge6$, so $v$ loses no charge, and (R1) sends enough to
make $v$ happy.

If $d(v)=5$, then $j\ge5$.  Furthermore, $j=5$ yields $k-4$ neighbors with
degree $k$.  Since $k\ge8$, charge at least $4(\FR14)$ comes to $v$, no charge
is given away, and $v$ is happy.

The remaining cases are all similar.  We show representative cases in
Figure~\ref{figVAL}.  Note that $v$ has at most $j+d(v)-9$ neighbors with
degree less than $k$.

If $d(v)=6$, then $j\ge4$.  At most $j-3$ neighbors have degree less than $k$.
For $j\in\{4,5,6\}$, $v$ gives at most $\FR24,\FR24,\FR34$ and receives at
least $\FR54,\FR44,\FR64$, respectively, ending happy.

If $d(v)=7$, then $j\ge3$.  At most $j-2$ neighbors have degree less than $k$.
For $j\in\{3,4,5,6\}$, $v$ gives at most $\FR33,\FR44,\FR34,\FR44$,
respectively, and remains happy.

If $d(v)\ge8$, then $j\ge2$.  At most $j-1$ neighbors have degree less than $k$.
For $j\in\{2,3,4,5,6\}$, $v$ gives at most $\FR42,\FR63,\FR64,\FR44,\FR54$,
respectively, and remains happy.
\end{proof}

\begin{figure}[h]
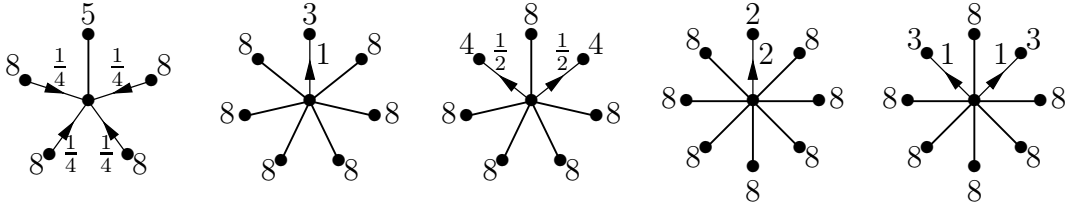

\gpic{
\expandafter\ifx\csname graph\endcsname\relax \csname newbox\endcsname\graph\fi
\expandafter\ifx\csname graphtemp\endcsname\relax \csname newdimen\endcsname\graphtemp\fi
\setbox\graph=\vtop{\vskip 0pt\hbox{%
    \graphtemp=.5ex\advance\graphtemp by 0.093in
    \rlap{\kern 0.423in\lower\graphtemp\hbox to 0pt{\hss $\bu$\hss}}%
    \graphtemp=.5ex\advance\graphtemp by 0.333in
    \rlap{\kern 0.754in\lower\graphtemp\hbox to 0pt{\hss $\bu$\hss}}%
    \graphtemp=.5ex\advance\graphtemp by 0.722in
    \rlap{\kern 0.628in\lower\graphtemp\hbox to 0pt{\hss $\bu$\hss}}%
    \graphtemp=.5ex\advance\graphtemp by 0.722in
    \rlap{\kern 0.219in\lower\graphtemp\hbox to 0pt{\hss $\bu$\hss}}%
    \graphtemp=.5ex\advance\graphtemp by 0.333in
    \rlap{\kern 0.093in\lower\graphtemp\hbox to 0pt{\hss $\bu$\hss}}%
    \graphtemp=.5ex\advance\graphtemp by 0.440in
    \rlap{\kern 0.423in\lower\graphtemp\hbox to 0pt{\hss $\bu$\hss}}%
    \special{pn 8}%
    \special{pa 754 333}%
    \special{pa 556 397}%
    \special{fp}%
    \special{sh 1.000}%
    \special{pa 651 391}%
    \special{pa 556 397}%
    \special{pa 637 347}%
    \special{pa 651 391}%
    \special{fp}%
    \special{pa 556 397}%
    \special{pa 423 440}%
    \special{fp}%
    \special{pa 628 722}%
    \special{pa 505 553}%
    \special{fp}%
    \special{sh 1.000}%
    \special{pa 541 642}%
    \special{pa 505 553}%
    \special{pa 578 614}%
    \special{pa 541 642}%
    \special{fp}%
    \special{pa 505 553}%
    \special{pa 423 440}%
    \special{fp}%
    \special{pa 219 722}%
    \special{pa 342 553}%
    \special{fp}%
    \special{sh 1.000}%
    \special{pa 268 614}%
    \special{pa 342 553}%
    \special{pa 306 642}%
    \special{pa 268 614}%
    \special{fp}%
    \special{pa 342 553}%
    \special{pa 423 440}%
    \special{fp}%
    \special{pa 93 333}%
    \special{pa 291 397}%
    \special{fp}%
    \special{sh 1.000}%
    \special{pa 210 347}%
    \special{pa 291 397}%
    \special{pa 196 391}%
    \special{pa 210 347}%
    \special{fp}%
    \special{pa 291 397}%
    \special{pa 423 440}%
    \special{fp}%
    \special{pn 11}%
    \special{pa 423 93}%
    \special{pa 423 440}%
    \special{fp}%
    \graphtemp=.5ex\advance\graphtemp by 0.719in
    \rlap{\kern 0.516in\lower\graphtemp\hbox to 0pt{\hss $\FR14$\hss}}%
    \graphtemp=.5ex\advance\graphtemp by 0.267in
    \rlap{\kern 0.574in\lower\graphtemp\hbox to 0pt{\hss $\FR14$\hss}}%
    \graphtemp=.5ex\advance\graphtemp by 0.267in
    \rlap{\kern 0.273in\lower\graphtemp\hbox to 0pt{\hss $\FR14$\hss}}%
    \graphtemp=.5ex\advance\graphtemp by 0.000in
    \rlap{\kern 0.423in\lower\graphtemp\hbox to 0pt{\hss 5\hss}}%
    \graphtemp=.5ex\advance\graphtemp by 0.267in
    \rlap{\kern 0.820in\lower\graphtemp\hbox to 0pt{\hss 8\hss}}%
    \graphtemp=.5ex\advance\graphtemp by 0.787in
    \rlap{\kern 0.693in\lower\graphtemp\hbox to 0pt{\hss 8\hss}}%
    \graphtemp=.5ex\advance\graphtemp by 0.787in
    \rlap{\kern 0.153in\lower\graphtemp\hbox to 0pt{\hss 8\hss}}%
    \graphtemp=.5ex\advance\graphtemp by 0.267in
    \rlap{\kern 0.027in\lower\graphtemp\hbox to 0pt{\hss 8\hss}}%
    \graphtemp=.5ex\advance\graphtemp by 0.719in
    \rlap{\kern 0.331in\lower\graphtemp\hbox to 0pt{\hss $\FR14$\hss}}%
    \graphtemp=.5ex\advance\graphtemp by 0.093in
    \rlap{\kern 1.582in\lower\graphtemp\hbox to 0pt{\hss $\bu$\hss}}%
    \graphtemp=.5ex\advance\graphtemp by 0.223in
    \rlap{\kern 1.854in\lower\graphtemp\hbox to 0pt{\hss $\bu$\hss}}%
    \graphtemp=.5ex\advance\graphtemp by 0.518in
    \rlap{\kern 1.921in\lower\graphtemp\hbox to 0pt{\hss $\bu$\hss}}%
    \graphtemp=.5ex\advance\graphtemp by 0.754in
    \rlap{\kern 1.734in\lower\graphtemp\hbox to 0pt{\hss $\bu$\hss}}%
    \graphtemp=.5ex\advance\graphtemp by 0.754in
    \rlap{\kern 1.431in\lower\graphtemp\hbox to 0pt{\hss $\bu$\hss}}%
    \graphtemp=.5ex\advance\graphtemp by 0.518in
    \rlap{\kern 1.243in\lower\graphtemp\hbox to 0pt{\hss $\bu$\hss}}%
    \graphtemp=.5ex\advance\graphtemp by 0.223in
    \rlap{\kern 1.311in\lower\graphtemp\hbox to 0pt{\hss $\bu$\hss}}%
    \graphtemp=.5ex\advance\graphtemp by 0.440in
    \rlap{\kern 1.582in\lower\graphtemp\hbox to 0pt{\hss $\bu$\hss}}%
    \special{pa 1582 440}%
    \special{pa 1854 223}%
    \special{fp}%
    \special{pa 1582 440}%
    \special{pa 1921 518}%
    \special{fp}%
    \special{pa 1582 440}%
    \special{pa 1734 754}%
    \special{fp}%
    \special{pa 1582 440}%
    \special{pa 1431 754}%
    \special{fp}%
    \special{pa 1582 440}%
    \special{pa 1243 518}%
    \special{fp}%
    \special{pa 1582 440}%
    \special{pa 1311 223}%
    \special{fp}%
    \special{pn 8}%
    \special{pa 1582 440}%
    \special{pa 1582 232}%
    \special{fp}%
    \special{sh 1.000}%
    \special{pa 1559 325}%
    \special{pa 1582 232}%
    \special{pa 1606 325}%
    \special{pa 1559 325}%
    \special{fp}%
    \special{pa 1582 232}%
    \special{pa 1582 93}%
    \special{fp}%
    \graphtemp=.5ex\advance\graphtemp by 0.158in
    \rlap{\kern 1.920in\lower\graphtemp\hbox to 0pt{\hss 8\hss}}%
    \graphtemp=.5ex\advance\graphtemp by 0.518in
    \rlap{\kern 2.014in\lower\graphtemp\hbox to 0pt{\hss 8\hss}}%
    \graphtemp=.5ex\advance\graphtemp by 0.819in
    \rlap{\kern 1.799in\lower\graphtemp\hbox to 0pt{\hss 8\hss}}%
    \graphtemp=.5ex\advance\graphtemp by 0.819in
    \rlap{\kern 1.366in\lower\graphtemp\hbox to 0pt{\hss 8\hss}}%
    \graphtemp=.5ex\advance\graphtemp by 0.518in
    \rlap{\kern 1.151in\lower\graphtemp\hbox to 0pt{\hss 8\hss}}%
    \graphtemp=.5ex\advance\graphtemp by 0.158in
    \rlap{\kern 1.245in\lower\graphtemp\hbox to 0pt{\hss 8\hss}}%
    \graphtemp=.5ex\advance\graphtemp by 0.000in
    \rlap{\kern 1.582in\lower\graphtemp\hbox to 0pt{\hss 3\hss}}%
    \graphtemp=.5ex\advance\graphtemp by 0.209in
    \rlap{\kern 1.652in\lower\graphtemp\hbox to 0pt{\hss 1\hss}}%
    \graphtemp=.5ex\advance\graphtemp by 0.093in
    \rlap{\kern 2.741in\lower\graphtemp\hbox to 0pt{\hss $\bu$\hss}}%
    \graphtemp=.5ex\advance\graphtemp by 0.223in
    \rlap{\kern 3.013in\lower\graphtemp\hbox to 0pt{\hss $\bu$\hss}}%
    \graphtemp=.5ex\advance\graphtemp by 0.518in
    \rlap{\kern 3.081in\lower\graphtemp\hbox to 0pt{\hss $\bu$\hss}}%
    \graphtemp=.5ex\advance\graphtemp by 0.754in
    \rlap{\kern 2.893in\lower\graphtemp\hbox to 0pt{\hss $\bu$\hss}}%
    \graphtemp=.5ex\advance\graphtemp by 0.754in
    \rlap{\kern 2.590in\lower\graphtemp\hbox to 0pt{\hss $\bu$\hss}}%
    \graphtemp=.5ex\advance\graphtemp by 0.518in
    \rlap{\kern 2.402in\lower\graphtemp\hbox to 0pt{\hss $\bu$\hss}}%
    \graphtemp=.5ex\advance\graphtemp by 0.223in
    \rlap{\kern 2.470in\lower\graphtemp\hbox to 0pt{\hss $\bu$\hss}}%
    \graphtemp=.5ex\advance\graphtemp by 0.440in
    \rlap{\kern 2.741in\lower\graphtemp\hbox to 0pt{\hss $\bu$\hss}}%
    \special{pn 11}%
    \special{pa 2741 440}%
    \special{pa 2741 93}%
    \special{fp}%
    \special{pa 2741 440}%
    \special{pa 3081 518}%
    \special{fp}%
    \special{pa 2741 440}%
    \special{pa 2893 754}%
    \special{fp}%
    \special{pa 2741 440}%
    \special{pa 2590 754}%
    \special{fp}%
    \special{pa 2741 440}%
    \special{pa 2402 518}%
    \special{fp}%
    \special{pn 8}%
    \special{pa 2741 440}%
    \special{pa 2904 310}%
    \special{fp}%
    \special{sh 1.000}%
    \special{pa 2818 350}%
    \special{pa 2904 310}%
    \special{pa 2846 386}%
    \special{pa 2818 350}%
    \special{fp}%
    \special{pa 2904 310}%
    \special{pa 3013 223}%
    \special{fp}%
    \special{pa 2741 440}%
    \special{pa 2579 310}%
    \special{fp}%
    \special{sh 1.000}%
    \special{pa 2637 386}%
    \special{pa 2579 310}%
    \special{pa 2665 350}%
    \special{pa 2637 386}%
    \special{fp}%
    \special{pa 2579 310}%
    \special{pa 2470 223}%
    \special{fp}%
    \graphtemp=.5ex\advance\graphtemp by 0.000in
    \rlap{\kern 2.741in\lower\graphtemp\hbox to 0pt{\hss 8\hss}}%
    \graphtemp=.5ex\advance\graphtemp by 0.518in
    \rlap{\kern 3.173in\lower\graphtemp\hbox to 0pt{\hss 8\hss}}%
    \graphtemp=.5ex\advance\graphtemp by 0.819in
    \rlap{\kern 2.958in\lower\graphtemp\hbox to 0pt{\hss 8\hss}}%
    \graphtemp=.5ex\advance\graphtemp by 0.819in
    \rlap{\kern 2.525in\lower\graphtemp\hbox to 0pt{\hss 8\hss}}%
    \graphtemp=.5ex\advance\graphtemp by 0.518in
    \rlap{\kern 2.310in\lower\graphtemp\hbox to 0pt{\hss 8\hss}}%
    \graphtemp=.5ex\advance\graphtemp by 0.158in
    \rlap{\kern 2.404in\lower\graphtemp\hbox to 0pt{\hss 4\hss}}%
    \graphtemp=.5ex\advance\graphtemp by 0.158in
    \rlap{\kern 3.079in\lower\graphtemp\hbox to 0pt{\hss 4\hss}}%
    \graphtemp=.5ex\advance\graphtemp by 0.185in
    \rlap{\kern 2.904in\lower\graphtemp\hbox to 0pt{\hss $\FR12$\hss}}%
    \graphtemp=.5ex\advance\graphtemp by 0.185in
    \rlap{\kern 2.579in\lower\graphtemp\hbox to 0pt{\hss $\FR12$\hss}}%
    \graphtemp=.5ex\advance\graphtemp by 0.093in
    \rlap{\kern 3.901in\lower\graphtemp\hbox to 0pt{\hss $\bu$\hss}}%
    \graphtemp=.5ex\advance\graphtemp by 0.195in
    \rlap{\kern 4.146in\lower\graphtemp\hbox to 0pt{\hss $\bu$\hss}}%
    \graphtemp=.5ex\advance\graphtemp by 0.440in
    \rlap{\kern 4.248in\lower\graphtemp\hbox to 0pt{\hss $\bu$\hss}}%
    \graphtemp=.5ex\advance\graphtemp by 0.686in
    \rlap{\kern 4.146in\lower\graphtemp\hbox to 0pt{\hss $\bu$\hss}}%
    \graphtemp=.5ex\advance\graphtemp by 0.788in
    \rlap{\kern 3.901in\lower\graphtemp\hbox to 0pt{\hss $\bu$\hss}}%
    \graphtemp=.5ex\advance\graphtemp by 0.686in
    \rlap{\kern 3.655in\lower\graphtemp\hbox to 0pt{\hss $\bu$\hss}}%
    \graphtemp=.5ex\advance\graphtemp by 0.440in
    \rlap{\kern 3.553in\lower\graphtemp\hbox to 0pt{\hss $\bu$\hss}}%
    \graphtemp=.5ex\advance\graphtemp by 0.195in
    \rlap{\kern 3.655in\lower\graphtemp\hbox to 0pt{\hss $\bu$\hss}}%
    \graphtemp=.5ex\advance\graphtemp by 0.440in
    \rlap{\kern 3.901in\lower\graphtemp\hbox to 0pt{\hss $\bu$\hss}}%
    \special{pn 11}%
    \special{pa 3901 440}%
    \special{pa 4146 195}%
    \special{fp}%
    \special{pa 3901 440}%
    \special{pa 4248 440}%
    \special{fp}%
    \special{pa 3901 440}%
    \special{pa 4146 686}%
    \special{fp}%
    \special{pa 3901 440}%
    \special{pa 3901 788}%
    \special{fp}%
    \special{pa 3901 440}%
    \special{pa 3655 686}%
    \special{fp}%
    \special{pa 3901 440}%
    \special{pa 3553 440}%
    \special{fp}%
    \special{pa 3901 440}%
    \special{pa 3655 195}%
    \special{fp}%
    \special{pn 8}%
    \special{pa 3901 440}%
    \special{pa 3901 232}%
    \special{fp}%
    \special{sh 1.000}%
    \special{pa 3877 325}%
    \special{pa 3901 232}%
    \special{pa 3924 325}%
    \special{pa 3877 325}%
    \special{fp}%
    \special{pa 3901 232}%
    \special{pa 3901 93}%
    \special{fp}%
    \graphtemp=.5ex\advance\graphtemp by 0.129in
    \rlap{\kern 4.212in\lower\graphtemp\hbox to 0pt{\hss 8\hss}}%
    \graphtemp=.5ex\advance\graphtemp by 0.440in
    \rlap{\kern 4.341in\lower\graphtemp\hbox to 0pt{\hss 8\hss}}%
    \graphtemp=.5ex\advance\graphtemp by 0.752in
    \rlap{\kern 4.212in\lower\graphtemp\hbox to 0pt{\hss 8\hss}}%
    \graphtemp=.5ex\advance\graphtemp by 0.927in
    \rlap{\kern 3.901in\lower\graphtemp\hbox to 0pt{\hss 8\hss}}%
    \graphtemp=.5ex\advance\graphtemp by 0.752in
    \rlap{\kern 3.589in\lower\graphtemp\hbox to 0pt{\hss 8\hss}}%
    \graphtemp=.5ex\advance\graphtemp by 0.440in
    \rlap{\kern 3.460in\lower\graphtemp\hbox to 0pt{\hss 8\hss}}%
    \graphtemp=.5ex\advance\graphtemp by 0.129in
    \rlap{\kern 3.589in\lower\graphtemp\hbox to 0pt{\hss 8\hss}}%
    \graphtemp=.5ex\advance\graphtemp by 0.000in
    \rlap{\kern 3.901in\lower\graphtemp\hbox to 0pt{\hss 2\hss}}%
    \graphtemp=.5ex\advance\graphtemp by 0.209in
    \rlap{\kern 3.970in\lower\graphtemp\hbox to 0pt{\hss 2\hss}}%
    \graphtemp=.5ex\advance\graphtemp by 0.093in
    \rlap{\kern 5.060in\lower\graphtemp\hbox to 0pt{\hss $\bu$\hss}}%
    \graphtemp=.5ex\advance\graphtemp by 0.195in
    \rlap{\kern 5.305in\lower\graphtemp\hbox to 0pt{\hss $\bu$\hss}}%
    \graphtemp=.5ex\advance\graphtemp by 0.440in
    \rlap{\kern 5.407in\lower\graphtemp\hbox to 0pt{\hss $\bu$\hss}}%
    \graphtemp=.5ex\advance\graphtemp by 0.686in
    \rlap{\kern 5.305in\lower\graphtemp\hbox to 0pt{\hss $\bu$\hss}}%
    \graphtemp=.5ex\advance\graphtemp by 0.788in
    \rlap{\kern 5.060in\lower\graphtemp\hbox to 0pt{\hss $\bu$\hss}}%
    \graphtemp=.5ex\advance\graphtemp by 0.686in
    \rlap{\kern 4.814in\lower\graphtemp\hbox to 0pt{\hss $\bu$\hss}}%
    \graphtemp=.5ex\advance\graphtemp by 0.440in
    \rlap{\kern 4.712in\lower\graphtemp\hbox to 0pt{\hss $\bu$\hss}}%
    \graphtemp=.5ex\advance\graphtemp by 0.195in
    \rlap{\kern 4.814in\lower\graphtemp\hbox to 0pt{\hss $\bu$\hss}}%
    \graphtemp=.5ex\advance\graphtemp by 0.440in
    \rlap{\kern 5.060in\lower\graphtemp\hbox to 0pt{\hss $\bu$\hss}}%
    \special{pn 11}%
    \special{pa 5060 440}%
    \special{pa 5060 93}%
    \special{fp}%
    \special{pa 5060 440}%
    \special{pa 5407 440}%
    \special{fp}%
    \special{pa 5060 440}%
    \special{pa 5305 686}%
    \special{fp}%
    \special{pa 5060 440}%
    \special{pa 5060 788}%
    \special{fp}%
    \special{pa 5060 440}%
    \special{pa 4814 686}%
    \special{fp}%
    \special{pa 5060 440}%
    \special{pa 4712 440}%
    \special{fp}%
    \special{pn 8}%
    \special{pa 5060 440}%
    \special{pa 5207 293}%
    \special{fp}%
    \special{sh 1.000}%
    \special{pa 5125 342}%
    \special{pa 5207 293}%
    \special{pa 5158 375}%
    \special{pa 5125 342}%
    \special{fp}%
    \special{pa 5207 293}%
    \special{pa 5305 195}%
    \special{fp}%
    \special{pa 5060 440}%
    \special{pa 4912 293}%
    \special{fp}%
    \special{sh 1.000}%
    \special{pa 4961 375}%
    \special{pa 4912 293}%
    \special{pa 4994 342}%
    \special{pa 4961 375}%
    \special{fp}%
    \special{pa 4912 293}%
    \special{pa 4814 195}%
    \special{fp}%
    \graphtemp=.5ex\advance\graphtemp by 0.000in
    \rlap{\kern 5.060in\lower\graphtemp\hbox to 0pt{\hss 8\hss}}%
    \graphtemp=.5ex\advance\graphtemp by 0.440in
    \rlap{\kern 5.500in\lower\graphtemp\hbox to 0pt{\hss 8\hss}}%
    \graphtemp=.5ex\advance\graphtemp by 0.752in
    \rlap{\kern 5.371in\lower\graphtemp\hbox to 0pt{\hss 8\hss}}%
    \graphtemp=.5ex\advance\graphtemp by 0.927in
    \rlap{\kern 5.060in\lower\graphtemp\hbox to 0pt{\hss 8\hss}}%
    \graphtemp=.5ex\advance\graphtemp by 0.752in
    \rlap{\kern 4.748in\lower\graphtemp\hbox to 0pt{\hss 8\hss}}%
    \graphtemp=.5ex\advance\graphtemp by 0.440in
    \rlap{\kern 4.619in\lower\graphtemp\hbox to 0pt{\hss 8\hss}}%
    \graphtemp=.5ex\advance\graphtemp by 0.129in
    \rlap{\kern 4.748in\lower\graphtemp\hbox to 0pt{\hss 3\hss}}%
    \graphtemp=.5ex\advance\graphtemp by 0.129in
    \rlap{\kern 5.371in\lower\graphtemp\hbox to 0pt{\hss 3\hss}}%
    \graphtemp=.5ex\advance\graphtemp by 0.209in
    \rlap{\kern 4.920in\lower\graphtemp\hbox to 0pt{\hss 1\hss}}%
    \graphtemp=.5ex\advance\graphtemp by 0.209in
    \rlap{\kern 5.199in\lower\graphtemp\hbox to 0pt{\hss 1\hss}}%
    \hbox{\vrule depth0.927in width0pt height 0pt}%
    \kern 5.500in
  }%
}%
}
\caption{Some cases in Theorem~\ref{vizingplanar} ending with charge $6$\label{figVAL}}
\end{figure}

Sanders and Zhao~\cite{SZ1} and Zhang~\cite{Zhang} proved
Conjecture~\ref{planar-graph-conj} for planar graphs with $\Delta(G)=7$,
extended in~\cite{SZ3} to graphs with maximum degree at least $7$ that embed in
a surface of nonnegative Euler characteristic.  Since the proof above uses
only $\mad(G)<6$, it holds also for graphs in the projective plane.  Graphs
on the torus (or Klein bottle) also satisfy $\mad(G)<6$ unless they triangulate
the surface, in which case $\avd(G)=6$.

Although Conjecture~\ref{planar-graph-conj} remains open when $\Delta(G)=6$, it
has been proved for various classes of planar graphs with certain subgraphs
forbidden, such as short cycles with chords (see~\cite{BW,WC1,WX}).  Note that
$\mad(G)<6$ is not sufficient when $\Delta(G)=6$; planarity really is needed.
Although $K_7$ is forbidden by $\mad(G)<6$, consider the graph $G$ obtained
from $K_7$ by subdividing one edge with a new $2$-vertex $v$; we have
$\Delta(G)=6$ and $\mad(G)<6$.  In a proper edge-coloring of $G$, only two
colors can appear four times (using edges at $v$); hence six colors can cover
only $20$ edges, but $G$ has $22$ edges.

Proper edge-coloring of $G$ is equivalent to proper coloring of the line
graph $L(G)$.  Since the line graph has a clique of size $\Delta(G)$,
Vizing's Theorem states that the optimization problem of proper coloring  
behaves much better when restricted to line graphs.
The same phenomenon seems to occur with the list version of the problem.

\begin{definition}
An \emph{edge-list assignment} $L$ assigns lists of available colors to the
edges of a graph $G$.  Given an edge-list assignment $L$, an
\emph{$L$-edge-coloring} of $G$ is a proper edge-coloring
$\phi$ such that $\phi(e)\in L(e)$ for all $e\in E(G)$.
A graph $G$ is \emph{$k$-edge-choosable} if $G$ is $L$-edge-colorable whenever
each list has size at least $k$.  The \emph{list edge-chromatic number} of $G$,
written $\chil'(G)$, is the least $k$ such that $G$ is $k$-edge-choosable.
\end{definition}

\begin{conjecture}[List Coloring Conjecture]\label{lcc}
$\chil'(G)=\chi'(G)$ for every graph $G$.
\looseness-1
\end{conjecture}

This conjecture was posed independently by many researchers.  It was first
published by Bollob\'as and Harris~\cite{BH}, but it was independently
formulated earlier by Albertson and Collins in 1981 and by Vizing as early as
1975 (both unpublished).  Kahn~\cite{Ka} proved the conjecture
asymptotically: $\chil'(G)\le (1+o(1))\chi'(G)$.

We first consider Vizing's weaker conjecture that always
$\chil'(G)\le \Delta(G)+1$.  Borodin~\cite{ovb} proved it for
planar $G$ with $\Delta(G)\ge9$.  Bonamy~\cite{Bon} extended this
to $\Delta(G)=8$, by a much longer proof with 11 reducible configurations.
It was also proved for planar graphs with $\Delta(G)\ge6$ having no two
$3$-faces sharing an edge~\cite{C1}.  It was proved in~\cite{JMS} for
$\Delta(G)\le4$ (including nonplanar graphs), and for $\Delta(G)=5$ it is
known for planar graphs with no $3$-cycle~\cite{ZW}, no $4$-cycle~\cite{C1}, or
no $5$-cycle~\cite{WL}.  The proofs for $\Delta(G)=5$ use discharging.

We present a recent use of balanced charging~\cite{CH} to prove the result of
Borodin~\cite{ovb}.  Balanced charging is natural when neither the graphs nor
their duals are triangulations.  Again we use a pot of charge (see
Theorem~\ref{listsq}).  In this proof, the pot facilitates moving charge from
maximum-degree vertices to $3$-vertices; we need not name specific recipients.

\begin{theorem}[\cite{ovb}]\label{listedge9}
If $G$ is a planar graph and $\Delta(G)\ge9$, then $\chil'(G)\le \Delta(G)+1$.
\end{theorem}
\begin{proof}
(Cohen and Havet~\cite{CH})
Let $G$ be a minimal counterexample, with an edge-list assignment $L$ such that
each list has size $\Delta(G)+1$ and $G$ has no $L$-edge-coloring.  An edge
with weight at most $\Delta(G)+2$ is reducible.  Hence we may assume that
$\delta(G)\ge3$ and that every neighbor of a $j$-vertex has degree at least
$\Delta(G)+3-j$.  Let $k=\Delta(G)$; since $k\ge9$, the degree-sum of any two
adjacent vertices is at least $12$.

We use balanced charging, with initial charge equal to degree or length minus
$4$.  Initially, the pot of charge is empty.  The discharging rules must make
each vertex and face happy and keep the charge in the pot nonnegative to
contradict the assumption of a counterexample.

\medskip\noindent
(R1) Every $3$-vertex takes $1$ from the pot, and every $k$-vertex gives
$\FR12$ to the pot.

\noindent
(R2) Each $3$-face takes $\FR12$ from each incident $8^+$-vertex and 
$\FR{j-4}j$ from each incident $j$-vertex with $j\in\{5,6,7\}$.

\medskip
To ensure positive charge in the pot, we prove $n_k>2n_3$, where $n_j$ is the
number of $j$-vertices in $G$.  The edges incident to $3$-vertices form a
bipartite graph $H$; its parts are the $3$-vertices and the $k$-vertices.
If $H$ has a cycle $C$, then $C$ has even length, since $H$ is bipartite.  By
the minimality of the counterexample, $G-E(C)$ has an $L$-edge-coloring.  Each
edge of $C$ is incident to $\Delta(G)-1$ edges that have now been colored, so
there remain at least two available colors on each edge (see
Figure~\ref{figlist9}).  Since even cycles are $2$-edge-choosable (by
Lemma~\ref{evencyc} and cycles being isomorphic to their line graphs), the
$L$-edge-coloring extends to $G$.  Since $G$ is a counterexample, we thus may
assume that $H$ is acyclic and therefore has fewer than $n_3+n_k$ edges.  Since
it also has $3n_3$ edges, we have $3n_3<n_3+n_k$, as desired.

For vertices, (R1) immediately makes $3$-vertices happy.  A $j$-vertex $v$ with
$j\in\{4,5,6,7\}$ loses altogether at most $j-4$, its initial charge.
An $8$-vertex loses at most $4$, since $k\ge9$.  For $j\ge9$, possibly sending
$\FR12$ to the pot, a $j$-vertex loses at most $\FR{j+1}2$ and is happy.

For faces, the $4^+$-faces lose no charge and remain happy; we must show that
each $3$-face $f$ gains at least $1$.  Let $j$ be the least degree among
vertices incident to $f$.  If $j\le 4$, then two incident $8^+$-vertices give
$\FR12$ each.  If $j=5$, then two incident $7^+$-vertices give at least $\FR37$
each, plus $\FR15$ for the $5$-vertex.  If $j\ge6$, then each vertex incident
to $f$ gives at least $\FR13$ to $f$.
\end{proof}

\begin{figure}[h]
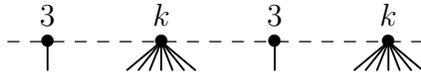

\gpic{
\expandafter\ifx\csname graph\endcsname\relax \csname newbox\endcsname\graph\fi
\expandafter\ifx\csname graphtemp\endcsname\relax \csname newdimen\endcsname\graphtemp\fi
\setbox\graph=\vtop{\vskip 0pt\hbox{%
    \graphtemp=.5ex\advance\graphtemp by 0.119in
    \rlap{\kern 0.208in\lower\graphtemp\hbox to 0pt{\hss $\bu$\hss}}%
    \graphtemp=.5ex\advance\graphtemp by 0.119in
    \rlap{\kern 0.803in\lower\graphtemp\hbox to 0pt{\hss $\bu$\hss}}%
    \graphtemp=.5ex\advance\graphtemp by 0.119in
    \rlap{\kern 1.397in\lower\graphtemp\hbox to 0pt{\hss $\bu$\hss}}%
    \graphtemp=.5ex\advance\graphtemp by 0.119in
    \rlap{\kern 1.992in\lower\graphtemp\hbox to 0pt{\hss $\bu$\hss}}%
    \special{pn 8}%
    \special{pa 0 119}%
    \special{pa 2200 119}%
    \special{da 0.059}%
    \special{pn 11}%
    \special{pa 208 119}%
    \special{pa 208 268}%
    \special{fp}%
    \special{pa 1397 119}%
    \special{pa 1397 268}%
    \special{fp}%
    \graphtemp=.5ex\advance\graphtemp by 0.000in
    \rlap{\kern 0.208in\lower\graphtemp\hbox to 0pt{\hss $3$\hss}}%
    \graphtemp=.5ex\advance\graphtemp by 0.000in
    \rlap{\kern 0.803in\lower\graphtemp\hbox to 0pt{\hss $k$\hss}}%
    \graphtemp=.5ex\advance\graphtemp by 0.000in
    \rlap{\kern 1.397in\lower\graphtemp\hbox to 0pt{\hss $3$\hss}}%
    \graphtemp=.5ex\advance\graphtemp by 0.000in
    \rlap{\kern 1.992in\lower\graphtemp\hbox to 0pt{\hss $k$\hss}}%
    \special{pa 803 119}%
    \special{pa 624 268}%
    \special{fp}%
    \special{pa 803 119}%
    \special{pa 684 268}%
    \special{fp}%
    \special{pa 803 119}%
    \special{pa 743 268}%
    \special{fp}%
    \special{pa 803 119}%
    \special{pa 803 268}%
    \special{fp}%
    \special{pa 803 119}%
    \special{pa 862 268}%
    \special{fp}%
    \special{pa 803 119}%
    \special{pa 922 268}%
    \special{fp}%
    \special{pa 803 119}%
    \special{pa 981 268}%
    \special{fp}%
    \special{pa 1992 119}%
    \special{pa 1814 268}%
    \special{fp}%
    \special{pa 1992 119}%
    \special{pa 1873 268}%
    \special{fp}%
    \special{pa 1992 119}%
    \special{pa 1932 268}%
    \special{fp}%
    \special{pa 1992 119}%
    \special{pa 1992 268}%
    \special{fp}%
    \special{pa 1992 119}%
    \special{pa 2051 268}%
    \special{fp}%
    \special{pa 1992 119}%
    \special{pa 2111 268}%
    \special{fp}%
    \special{pa 1992 119}%
    \special{pa 2170 268}%
    \special{fp}%
    \hbox{\vrule depth0.268in width0pt height 0pt}%
    \kern 2.200in
  }%
}%
}
\caption{Excluded cycles in Theorem~\ref{listedge9}\label{figlist9}}
\end{figure}
This proof fits the model of discharging to produce an unavoidable set of
reducible configurations.  The reducible configurations are light edges
(degree-sum at most $\Delta(G)+2$) and cycles alternating between $3$-vertices
and $\Delta(G)$-vertices.  The first use of arbitrarily large reducible
configurations (cycles alternating between $2$-vertices and
$\Delta(G)$-vertices) was in Borodin~\cite{B89}.  Notions analogous to the pot
of charge for long-distance transfer of charge appear in~\cite{HS}
and~\cite{BI0}; a general term for such methods is ``global discharging''.

\medskip

Now we return to the full List Coloring Conjecture $\chil'(G)=\chi'(G)$.
This was proved for bipartite multigraphs by Galvin~\cite{Ga}, where
always $\chi'(G)=\Delta(G)$.  With Vizing conjecturing $\chi'(G)=\Delta(G)$
when $G$ is planar and $\Delta(G)\ge6$ (Conjecture~\ref{vizconj}), we also
seek $\chil'(G)=\Delta(G)$ for such graphs.  Borodin~\cite{ovb} proved it for
$\Delta(G)\ge14$.  This later was strengthened to $\Delta(G)\ge12$ by Borodin,
Kostochka, and Woodall~\cite{BKW}.  We present an alternative proof of the 
result of Borodin~\cite{ovb}.  The result in~\cite{BKW} uses similar
discharging, but it requires more reducible configurations and more detailed
analysis.  

A {\it $t$-alternating} cycle alternates between $t$-vertices and vertices of
higher degree (introduced in Borodin~\cite{B89}).  We used $3$-alternating
cycles in Theorem~\ref{listedge9}.

\begin{lem}[\cite{ovb}]\label{altcyc}
If $G$ is a simple plane graph with $\delta(G)\ge 2$,
then $G$ contains \\
(C1) an edge $uv$ with $d(u)+d(v)\leq 15$, or\\
(C2) a $2$-alternating cycle $C$.
\end{lem}
\begin{proof}
In a counterexample $G$, we have $d(u)+d(v)\ge 16$ for every edge $uv$.  
Both neighbors of any $2$-vertex are $14^+$-vertices.  Since $G$ is simple,
every $2$-vertex lies on a $4^+$-face.  

To obtain a contradiction, we use face charging, with initial charge $2d(v)-6$
at each vertex $v$ and $\ell(f)-6$ at each face $f$.  We also keep a central
pot of charge (initially empty) and use the following discharging rules
(see Figure~\ref{fig2alt}).

\medskip\noindent
(R1) Each $14^+$-vertex gives charge $1$ to the pot, and each
$2$-vertex takes $1$ from the pot.

\noindent
(R2) Each $4^+$-vertex distributes its charge remaining after (R1) equally to
its incident faces.

\noindent
(R3) Each $4^+$-face gives charge 1 to each incident $2$-vertex.

\vspace{-1pc}

\begin{figure}[h]
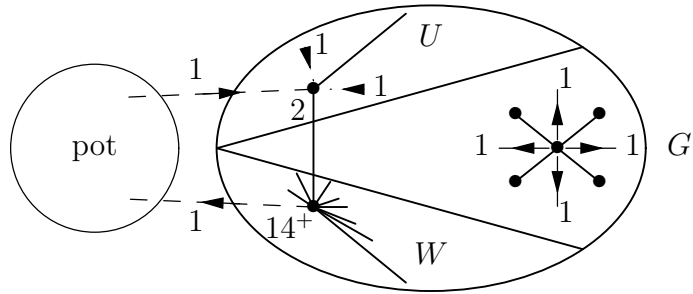

\gpic{
\expandafter\ifx\csname graph\endcsname\relax \csname newbox\endcsname\graph\fi
\expandafter\ifx\csname graphtemp\endcsname\relax \csname newdimen\endcsname\graphtemp\fi
\setbox\graph=\vtop{\vskip 0pt\hbox{%
    \special{pn 8}%
    \special{ar 440 925 440 440 0 6.28319}%
    \special{pn 11}%
    \special{ar 2201 925 1123 748 0 6.28319}%
    \special{pa 2995 395}%
    \special{pa 1079 925}%
    \special{fp}%
    \special{pa 1079 925}%
    \special{pa 2995 1454}%
    \special{fp}%
    \graphtemp=.5ex\advance\graphtemp by 1.233in
    \rlap{\kern 1.585in\lower\graphtemp\hbox to 0pt{\hss $\bu$\hss}}%
    \graphtemp=.5ex\advance\graphtemp by 0.616in
    \rlap{\kern 1.585in\lower\graphtemp\hbox to 0pt{\hss $\bu$\hss}}%
    \graphtemp=.5ex\advance\graphtemp by 0.925in
    \rlap{\kern 2.862in\lower\graphtemp\hbox to 0pt{\hss $\bu$\hss}}%
    \graphtemp=.5ex\advance\graphtemp by 1.101in
    \rlap{\kern 3.082in\lower\graphtemp\hbox to 0pt{\hss $\bu$\hss}}%
    \graphtemp=.5ex\advance\graphtemp by 1.101in
    \rlap{\kern 2.642in\lower\graphtemp\hbox to 0pt{\hss $\bu$\hss}}%
    \graphtemp=.5ex\advance\graphtemp by 0.748in
    \rlap{\kern 2.642in\lower\graphtemp\hbox to 0pt{\hss $\bu$\hss}}%
    \graphtemp=.5ex\advance\graphtemp by 0.748in
    \rlap{\kern 3.082in\lower\graphtemp\hbox to 0pt{\hss $\bu$\hss}}%
    \special{pa 2862 925}%
    \special{pa 3082 1101}%
    \special{fp}%
    \special{pa 2862 925}%
    \special{pa 2642 1101}%
    \special{fp}%
    \special{pa 2862 925}%
    \special{pa 2642 748}%
    \special{fp}%
    \special{pa 2862 925}%
    \special{pa 3082 748}%
    \special{fp}%
    \special{pa 1585 1233}%
    \special{pa 1453 1145}%
    \special{fp}%
    \special{pa 1585 1233}%
    \special{pa 1497 1057}%
    \special{fp}%
    \special{pa 1585 1233}%
    \special{pa 1673 1101}%
    \special{fp}%
    \special{pa 1585 1233}%
    \special{pa 1717 1189}%
    \special{fp}%
    \special{pa 1585 1233}%
    \special{pa 1761 1233}%
    \special{fp}%
    \special{pa 1585 1233}%
    \special{pa 1805 1321}%
    \special{fp}%
    \special{pa 1585 1233}%
    \special{pa 1893 1409}%
    \special{fp}%
    \special{pn 8}%
    \special{pa 2906 925}%
    \special{pa 3214 925}%
    \special{da 0.088}%
    \special{pa 3029 925}%
    \special{pa 3091 925}%
    \special{fp}%
    \special{sh 1.000}%
    \special{pa 2985 898}%
    \special{pa 3091 925}%
    \special{pa 2985 951}%
    \special{pa 2985 898}%
    \special{fp}%
    \special{pa 2862 969}%
    \special{pa 2862 1277}%
    \special{da 0.088}%
    \special{pa 2862 1092}%
    \special{pa 2862 1153}%
    \special{fp}%
    \special{sh 1.000}%
    \special{pa 2888 1048}%
    \special{pa 2862 1153}%
    \special{pa 2835 1048}%
    \special{pa 2888 1048}%
    \special{fp}%
    \special{pa 2862 881}%
    \special{pa 2862 572}%
    \special{da 0.088}%
    \special{pa 2862 757}%
    \special{pa 2862 696}%
    \special{fp}%
    \special{sh 1.000}%
    \special{pa 2835 801}%
    \special{pa 2862 696}%
    \special{pa 2888 801}%
    \special{pa 2835 801}%
    \special{fp}%
    \special{pa 2818 925}%
    \special{pa 2509 925}%
    \special{da 0.088}%
    \special{pa 2694 925}%
    \special{pa 2633 925}%
    \special{fp}%
    \special{sh 1.000}%
    \special{pa 2738 951}%
    \special{pa 2633 925}%
    \special{pa 2738 898}%
    \special{pa 2738 951}%
    \special{fp}%
    \graphtemp=.5ex\advance\graphtemp by 0.925in
    \rlap{\kern 0.440in\lower\graphtemp\hbox to 0pt{\hss pot\hss}}%
    \graphtemp=.5ex\advance\graphtemp by 1.497in
    \rlap{\kern 2.201in\lower\graphtemp\hbox to 0pt{\hss $W$\hss}}%
    \graphtemp=.5ex\advance\graphtemp by 0.352in
    \rlap{\kern 2.201in\lower\graphtemp\hbox to 0pt{\hss $U$\hss}}%
    \graphtemp=.5ex\advance\graphtemp by 0.741in
    \rlap{\kern 1.504in\lower\graphtemp\hbox to 0pt{\hss $2$\hss}}%
    \graphtemp=.5ex\advance\graphtemp by 1.357in
    \rlap{\kern 1.460in\lower\graphtemp\hbox to 0pt{\hss $14^+$\hss}}%
    \graphtemp=.5ex\advance\graphtemp by 0.925in
    \rlap{\kern 3.500in\lower\graphtemp\hbox to 0pt{\hss $G$\hss}}%
    \special{pn 11}%
    \special{pa 2069 1629}%
    \special{pa 1585 1233}%
    \special{fp}%
    \special{pa 1585 1233}%
    \special{pa 1585 616}%
    \special{fp}%
    \special{pa 1585 616}%
    \special{pa 2069 220}%
    \special{fp}%
    \special{pn 8}%
    \special{pa 1585 1233}%
    \special{pa 619 1192}%
    \special{da 0.088}%
    \special{pa 1199 1216}%
    \special{pa 1006 1208}%
    \special{fp}%
    \special{sh 1.000}%
    \special{pa 1110 1239}%
    \special{pa 1006 1208}%
    \special{pa 1112 1186}%
    \special{pa 1110 1239}%
    \special{fp}%
    \special{pa 619 657}%
    \special{pa 1585 616}%
    \special{da 0.088}%
    \special{pa 1006 641}%
    \special{pa 1199 633}%
    \special{fp}%
    \special{sh 1.000}%
    \special{pa 1092 611}%
    \special{pa 1199 633}%
    \special{pa 1094 664}%
    \special{pa 1092 611}%
    \special{fp}%
    \graphtemp=.5ex\advance\graphtemp by 1.321in
    \rlap{\kern 0.969in\lower\graphtemp\hbox to 0pt{\hss $1$\hss}}%
    \graphtemp=.5ex\advance\graphtemp by 0.528in
    \rlap{\kern 0.969in\lower\graphtemp\hbox to 0pt{\hss $1$\hss}}%
    \special{pa 1849 616}%
    \special{pa 1673 616}%
    \special{da 0.088}%
    \special{pa 1779 616}%
    \special{pa 1743 616}%
    \special{fp}%
    \special{sh 1.000}%
    \special{pa 1849 643}%
    \special{pa 1743 616}%
    \special{pa 1849 590}%
    \special{pa 1849 643}%
    \special{fp}%
    \special{pa 1541 396}%
    \special{pa 1585 572}%
    \special{da 0.088}%
    \special{pa 1558 467}%
    \special{pa 1567 502}%
    \special{fp}%
    \special{sh 1.000}%
    \special{pa 1567 393}%
    \special{pa 1567 502}%
    \special{pa 1516 406}%
    \special{pa 1567 393}%
    \special{fp}%
    \graphtemp=.5ex\advance\graphtemp by 0.616in
    \rlap{\kern 1.937in\lower\graphtemp\hbox to 0pt{\hss $1$\hss}}%
    \graphtemp=.5ex\advance\graphtemp by 0.396in
    \rlap{\kern 1.629in\lower\graphtemp\hbox to 0pt{\hss $1$\hss}}%
    \graphtemp=.5ex\advance\graphtemp by 0.925in
    \rlap{\kern 3.258in\lower\graphtemp\hbox to 0pt{\hss $1$\hss}}%
    \graphtemp=.5ex\advance\graphtemp by 1.277in
    \rlap{\kern 2.906in\lower\graphtemp\hbox to 0pt{\hss $1$\hss}}%
    \graphtemp=.5ex\advance\graphtemp by 0.572in
    \rlap{\kern 2.906in\lower\graphtemp\hbox to 0pt{\hss $1$\hss}}%
    \graphtemp=.5ex\advance\graphtemp by 0.925in
    \rlap{\kern 2.465in\lower\graphtemp\hbox to 0pt{\hss $1$\hss}}%
    \hbox{\vrule depth1.849in width0pt height 0pt}%
    \kern 3.500in
  }%
}%
}

\vspace{-1.5pc}
\caption{Discharging for Lemma~\ref{altcyc}\label{fig2alt}}
\end{figure}

\vspace{-.5pc}
\medskip
To keep the charge in the pot nonnegative, we need $|U|\le|W|$, where $U$ and
$W$ denote the sets of $2$-vertices and $14^+$-vertices, respectively.
Let $H$ be the bipartite subgraph of $G$ with vertex set $U\cup W$ and edge set
consisting of all edges with endpoints in both $U$ and $W$.  Since (C2) does
not occur in $G$, the components of $H$ are trees.  Also (C1) does not occur,
so $2|U|=|E(H)|<|U|+|W|$.  Thus $|U|<|W|$.

A $2$-vertex takes $1$ from the pot and $1$ from an incident $4^+$-face (since
$G$ is simple) and ends happy.  A $3$-vertex starts and ends with no charge.
By (R2), a $4^+$-vertex also ends with charge $0$.  Hence all vertices are
happy.

\nobreak
Faces give charge to $2$-vertices and take charge from $4^+$-vertices.  Under
(R2), a face takes charge $\FR{2j-6}j$ or $\FR{2j-7}j$ from an incident
$j$-vertex when $j\ge4$, the latter when $j\ge14$.  Thus the value is at least
$\FR12$ when $j\ge4$, at least $1$ when $j\ge6$, and at least $\FR32$ when
$j\ge12$.

If a face $f$ has no incident $3^-$-vertices, then it receives at least
$\FR12\ell(f)$; its final charge is at least $\FR32\ell(f)-6$, which is
nonnegative when $\ell(f)\ge4$.  When $f$ is a $3$-face or a face incident to
some $3^-$-vertex, let $k$ be the least degree among the vertices incident to
$f$.  Prohibiting (C1) gives $f$ two incident $(16-k)^+$-vertices.  

A $3$-face needs to receive charge at least $3$.  When $k\ge6$, it receives at
least $1$ from each incident vertex.  When $k=2$, the other incident vertices
have degree at least $14$, and each provides $\FR32$.  When $3\le k\le 5$, it
receives at least $\FR{2k-6}k+2\cdot\FR{26-2k}{16-k}$, which is at least $3$.  

A $4^+$-face $f$ needs $2$ (or maybe less) to become happy.  If $k\ge 3$ or
$f$ has exactly one incident $2$-vertex, then $f$ receives at least $3$ and
gives away at most $1$.  If $f$ has at least two incident $2$-vertices, then
each is followed on $f$ (in a consistent direction) by a $14^+$-vertex, which
contributes at least $\FR32$.  These pairs net at least $\FR12$ each for $f$.
If $G$ has no $2$-alternating cycle, then $f$ has another incident
$14^+$-vertex that has not been counted, which provides more than enough charge
to $f$.
\end{proof}

\begin{theorem}[\cite{ovb}]\label{list14}
If $G$ is a plane graph with $\Delta(G)\ge 14$, then $\chi'_\ell(G)=\Delta(G)$.
\end{theorem}
\begin{proof}
Let $G$ be a minimal counterexample, having no $L$-edge-coloring from edge-list
assignment $L$.  If $G$ has a $1$-vertex with incident edge $e$, then $G-e$ has
an $L$-edge-coloring, and it extends to $e$.  Thus $\delta(G)\ge 2$.  By
Lemma~\ref{altcyc}, $G$ has an edge $uv$ with $d(u)+d(v)\le 15$ or a
$2$-alternating cycle $C$.  In the first case, we can extend an
$L$-edge-coloring of $G-uv$, since $|L(uv)|\ge14$ and at most $13$ colors are
restricted from use on $uv$.  In the other case, by minimality $G-E(C)$ has an
$L$-edge-coloring.  Since each list has size at least $\Delta(G)$, each edge of
$C$ has at least two colors remaining available, and the $2$-edge-choosability
of even cycles allows us to extend the edge-coloring.
\end{proof}

Finally, we come full circle and return to the role of bounding the maximum
average degree.  Vizing~\cite{V3} conjectured that an $n$-vertex edge-critical
graph $G$ must have at least $\FR12 [n(\Delta(G)-1)+3]$ edges, which in our
language translates to ``$\chi'(G)=\Delta(G)$ when $\mad(G)\le\Delta(G)-1$''.
Based on the List Coloring Conjecture, Woodall~\cite{W2} conjectured that
$\mad(G)<\Delta(G)-1$ also implies $\chil'(G)=\Delta(G)$.  In this direction,
it is known that $\chil'(G)=\Delta(G)$ when $\mad(G)<\sqrt{2\Delta(G)}$.  The
result is implicit in~\cite{BKW}, using the following tool.

\begin{theorem}[Borodin--Kostochka--Woodall~\cite{BKW}]\label{BKW}
If lists on the edges of a bipartite multigraph $G$ satisfy
$|L(uv)|\ge \max\{d_G(u),d_G(v)\}$ for $uv\in E(G)$, then $G$ has an
$L$-edge-coloring.
\end{theorem}


Woodall~\cite{W2} rephrased the argument using discharging, introducing an
exciting new way of moving charge in successive stages.  When the average
degree is large, vertices with very small degree need a lot of charge.  It may
be too hard to specify exactly where it all comes from.  Hence he allows charge
to move in phases, which we call {\it iterated discharging}.  Besides light
edges, we will need another reducible configuration (see Figure~\ref{figialt}).

\begin{definition}
In a multigraph $G$, an \emph{$i$-alternating subgraph} is a bipartite
submultigraph $F$ with parts $U$ and $W$ such that $d_F(u)=d_G(u)\le i$
when $u\in U$ and $d_G(w)-d_F(w)\le \Delta(G)-i$ when $w\in W$.
Note that cycles in $F$ alternate between $W$ and $i^-$-vertices in $U$.
\end{definition}

\vspace{-1pc}

\begin{figure}[h]
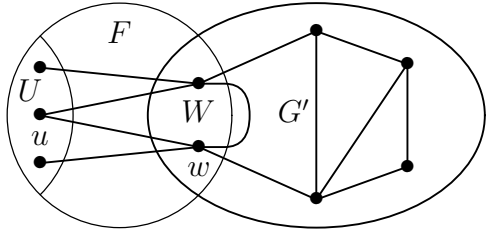

\gpic{
\expandafter\ifx\csname graph\endcsname\relax \csname newbox\endcsname\graph\fi
\expandafter\ifx\csname graphtemp\endcsname\relax \csname newdimen\endcsname\graphtemp\fi
\setbox\graph=\vtop{\vskip 0pt\hbox{%
    \special{pn 8}%
    \special{ar 588 647 588 588 0 6.28319}%
    \special{pn 11}%
    \special{ar 1618 647 882 588 0 6.28319}%
    \special{pn 8}%
    \special{ar -243 647 588 588 -0.785398 0.785398}%
    \graphtemp=.5ex\advance\graphtemp by 0.897in
    \rlap{\kern 0.172in\lower\graphtemp\hbox to 0pt{\hss $\bu$\hss}}%
    \graphtemp=.5ex\advance\graphtemp by 0.647in
    \rlap{\kern 0.172in\lower\graphtemp\hbox to 0pt{\hss $\bu$\hss}}%
    \graphtemp=.5ex\advance\graphtemp by 0.397in
    \rlap{\kern 0.172in\lower\graphtemp\hbox to 0pt{\hss $\bu$\hss}}%
    \graphtemp=.5ex\advance\graphtemp by 0.813in
    \rlap{\kern 1.004in\lower\graphtemp\hbox to 0pt{\hss $\bu$\hss}}%
    \graphtemp=.5ex\advance\graphtemp by 0.481in
    \rlap{\kern 1.004in\lower\graphtemp\hbox to 0pt{\hss $\bu$\hss}}%
    \graphtemp=.5ex\advance\graphtemp by 0.206in
    \rlap{\kern 1.618in\lower\graphtemp\hbox to 0pt{\hss $\bu$\hss}}%
    \graphtemp=.5ex\advance\graphtemp by 1.088in
    \rlap{\kern 1.618in\lower\graphtemp\hbox to 0pt{\hss $\bu$\hss}}%
    \graphtemp=.5ex\advance\graphtemp by 0.378in
    \rlap{\kern 2.095in\lower\graphtemp\hbox to 0pt{\hss $\bu$\hss}}%
    \graphtemp=.5ex\advance\graphtemp by 0.916in
    \rlap{\kern 2.095in\lower\graphtemp\hbox to 0pt{\hss $\bu$\hss}}%
    \special{pn 11}%
    \special{pa 172 897}%
    \special{pa 1004 813}%
    \special{fp}%
    \special{pa 1004 813}%
    \special{pa 172 647}%
    \special{fp}%
    \special{pa 172 647}%
    \special{pa 1004 481}%
    \special{fp}%
    \special{pa 1004 481}%
    \special{pa 172 397}%
    \special{fp}%
    \special{pa 1004 813}%
    \special{pa 1269 813}%
    \special{pa 1269 481}%
    \special{pa 1004 481}%
    \special{sp}%
    \special{pa 1004 813}%
    \special{pa 1618 1088}%
    \special{fp}%
    \special{pa 1618 1088}%
    \special{pa 2095 916}%
    \special{fp}%
    \special{pa 2095 916}%
    \special{pa 2095 378}%
    \special{fp}%
    \special{pa 2095 378}%
    \special{pa 1618 206}%
    \special{fp}%
    \special{pa 1618 206}%
    \special{pa 1004 481}%
    \special{fp}%
    \special{pa 1618 206}%
    \special{pa 1618 1088}%
    \special{fp}%
    \special{pa 1618 1088}%
    \special{pa 2095 378}%
    \special{fp}%
    \graphtemp=.5ex\advance\graphtemp by 0.235in
    \rlap{\kern 0.588in\lower\graphtemp\hbox to 0pt{\hss $F$\hss}}%
    \graphtemp=.5ex\advance\graphtemp by 0.647in
    \rlap{\kern 1.500in\lower\graphtemp\hbox to 0pt{\hss $G'$\hss}}%
    \graphtemp=.5ex\advance\graphtemp by 0.765in
    \rlap{\kern 0.172in\lower\graphtemp\hbox to 0pt{\hss $u$\hss}}%
    \graphtemp=.5ex\advance\graphtemp by 0.931in
    \rlap{\kern 1.004in\lower\graphtemp\hbox to 0pt{\hss $w$\hss}}%
    \graphtemp=.5ex\advance\graphtemp by 0.534in
    \rlap{\kern 0.119in\lower\graphtemp\hbox to 0pt{\hss $U$\hss}}%
    \graphtemp=.5ex\advance\graphtemp by 0.647in
    \rlap{\kern 1.000in\lower\graphtemp\hbox to 0pt{\hss $W$\hss}}%
    \hbox{\vrule depth1.235in width0pt height 0pt}%
    \kern 2.500in
  }%
}%
}

\vspace{-1pc}
\caption{A $2$-alternating subgraph $F$\label{figialt}}
\end{figure}

\vspace{-.5pc}

\begin{lemma}[\cite{BKW,W2}]\label{ialt}
$i$-alternating subgraphs are reducible for the property that
edge-choosability equals maximum degree.
\end{lemma}
\begin{proof}
Let $L$ be a $\Delta(G)$-uniform edge-list assignment for such a multigraph $G$.
Let $F$ be an $i$-alternating subgraph of $G$, and let $G'=G-E(F)$ (see
Figure~\ref{figialt}).  Choose an $L$-edge-coloring of $G'$, and delete the
chosen colors from the lists of their incident edges in $F$.  We claim that the
lists remain large enough to apply Theorem~\ref{BKW} to $F$.

For $uw\in E(F)$, no colors have been lost to edges incident at $u$, since all
edges incident to $u$ lie in $F$.  The number of colors lost to edges incident
to $w$, by definition, is at most $d_G(w)-d_F(w)$.  Since $d_G(w)\le\Delta(G)$,
the list on $uw$ retains at least $d_F(w)$ colors.  Also
$d_F(u)\le i\le \Delta(G)-(d_G(w)-d_F(w))$, so the list on $uw$ also retains
at least $d_F(u)$ colors.  Now Theorem~\ref{BKW} applies to complete the 
$L$-edge-coloring of $G$.
\end{proof}

To avoid technicalities, we make the bound on $\mad(G)$ slightly tighter
than needed.

\begin{theorem}[\cite{BKW,W2}]\label{woodall}
If $\mad(G)\le\sqrt{2\Delta(G)}-1$, then $\chil'(G)=\Delta(G)$.
\end{theorem}
\begin{proof}
Let $b=\sqrt{2\Delta(G)}-1$.  It suffices to show that every graph $G$ with
average degree at most $b$ contains a edge with weight at most $\Delta(G)+1$ or
an $i$-alternating subgraph with $i\le b$.  Suppose that $G$ contains neither.
An edge incident to a $1$-vertex would be light, so we may assume
$\delta(G)\ge2$.

Use degree charging.  In phase $i$ of discharging, for $2\le i\le \FL{b}$, each
$i^-$-vertex receives charge $1$ from a neighbor.  We want every vertex to end
with charge at least $\CL{b}$.

To begin phase $i$, let $U=\{v\st d_G(v)\le i\}$, and let $W$ be the set of
all vertices having neighbors in $U$.  Note that $U$ is independent (no light
edge).  Let $F$ be the subgraph with vertex set $U\cup W$ containing all edges
incident to $U$.  Since $G$ has no $i$-alternating subgraph, there exists
$w\in W$ such that $d_F(w)\le d_G(w)+i-\Delta(G)-1\le i-1$.  Move
charge $1$ from this vertex $w$ to each of its neighbors in $U$.

Now delete $\{w\}\cup (N(w)\cap U)$ from $F$.  Each deleted vertex in $U$ 
has received charge $1$, and $w$ lost at most $i-1$.  Iterate.  What remains of
$U$ and $W$ at each step cannot form an $i$-alternating subgraph, so we
continue to find the desired vertex until $U$ is empty.

Since each vertex with degree at most $i$ receives a unit of charge in phase
$i$, vertices with degree less than $\FL{\sqrt{2\Delta(G)}}$ have their charge
increased to at least $\FL{\sqrt{2\Delta(G)}}$ (and they never lose charge).
Since there is no light edge, vertices with larger degree $j$ lose charge
only on rounds $i$ with $i\ge\Delta(G)+2-j$.  Hence such a vertex loses charge
at most $\sum_{i=\Delta(G)+2-j}^{\FL{b}}(i-1)$.  With each reduction of $1$ in
$j$, the amount of lost charge declines by more than $1$, so it suffices to
show that vertices with degree $\Delta(G)$ keep sufficient charge.
Their lost charge is bounded by $\FR12 b(b-1)$, so they keep charge at least
$\FR32b$, which is more than enough.
\end{proof}

This use of discharging in~\cite{W2} replaced extensive manipulations of finite
sums in~\cite{BKW}; it is another illustration of the notion of ``amortized
counting'' we mentioned earlier.  Woodall also gave an example to show that the
discharging argument is essentially sharp, meaning that more reducible
configurations will be needed to weaken the hypothesis on $\mad(G)$.

For the Vizing conjecture saying approximately that $\mad(G)\le \Delta(G)-1$
implies $\chi'(G)=\Delta(G)$, Fiorini~\cite{Fio} made the first major step,
proving that $\mad(G)<\FR12(\Delta(G)+1)$ suffices (by counting edges in
edge-critical graphs).  After a number of papers using similar manipulations
of finite sums, Sanders and Zhao~\cite{SZ2} greatly simplified the proof by
using discharging and improved the result; they showed that
$\mad(G)<\FR12(\Delta(G)+\sqrt{2\Delta(G)-1})$ suffices.  Woodall~\cite{W0}
then proved that $\mad(G)<\FR23(\Delta(G)+)$ suffices.  The list version seems
to be much harder, and the more restrictive requirement of
$\mad(G)<\sqrt{2\Delta(G)}$ in Theorem~\ref{woodall} is a first step.

\medskip

{\small

%

\begin{exercise}
Let $G$ be a graph with maximum degree at least $8$ that embeds on the torus.
By a closer examination of the proof of Theorem~\ref{vizingplanar}, prove
that $\chi'(G)=\Delta(G)$ except possibly when $G$ is obtained from a 
$6$-regular triangulation $H$ of the torus by inserting vertices of degree $3$
into one-third of the faces in $H$, chosen so that each vertex in $H$ lies on
exactly two of the chosen faces, and making each new vertex adjacent to the
vertices of $H$ on its face.  It suffices to show that otherwise every vertex
ends with charge at least $6$ and some vertex ends with larger charge.
\end{exercise}


\begin{exercise}
Prove that if $\Delta(G)\le6$ and $\avd(G)<\FR72$, then $G$ contains an isolated
vertex, an edge with weight at most $7$, or a cycle alternating between
$2$-vertices and $6$-vertices.  Conclude that if $\Delta(G)\le6$ and
$\mad(G)<\FR72$, then $G$ is $6$-edge-choosable.  
%
\end{exercise}

\begin{exercise}
(Borodin~\cite{B89}, Borodin--Kostochka--Woodall~\cite{BKW})
A \emph{total coloring} assigns colors to both edges and vertices, so that
elements get distinct colors if they are either incident or adjacent.
Adapt the proofs of Theorems~\ref{list14} and~\ref{woodall} to prove analogous
versions for choosing total colorings from lists.  In each case, the bound for
the size of lists to permit choosing a total coloring is larger by $1$ than
that for choosing a proper edge-coloring.
\end{exercise}


}

\end{document}